\documentclass[11pt]{amsart}
\usepackage{etex}
\usepackage{amsmath,amsfonts,amssymb,amsthm}
\usepackage[alphabetic]{amsrefs}
\usepackage{hyperref}
\usepackage{tikz}
\usepackage{graphicx,color}
\usepackage{pictexwd,dcpic,epsf}
\usepackage{enumerate}
\usepackage{enumitem}

\newcommand{\area}{\operatorname{Area}}
\newcommand{\length}{\operatorname{length}}
\newcommand{\lk}{\operatorname{lk}}
\newcommand{\St}{\operatorname{St}}
\newcommand{\QI}{\operatorname{QI}}
\newcommand{\Isom}{\operatorname{Isom}}
\newcommand{\Aut}{\operatorname{Aut}}

\newcommand {\B}{\mathcal B}
\newcommand {\C}{\mathcal C}
\newcommand {\D}{\mathbb D}
\newcommand {\I}{\mathcal I}
\newcommand {\V}{\mathcal V}
\newcommand {\F}{\mathcal F}
\newcommand {\h}{\mathcal H}
\newcommand {\p}{\mathcal P}
\newcommand {\W}{\mathcal W}
\newcommand {\s}{\mathcal S}

\newtheorem{theorem}{Theorem}[section]
\newtheorem{lemma}[theorem]{Lemma}
\newtheorem{question}[theorem]{Question}
\newtheorem{prop}[theorem]{Proposition}
\newtheorem{conj}[theorem]{Conjecture}

\newtheorem{corollary}[theorem]{Corollary}

\theoremstyle{definition} 
\newtheorem{definition}[theorem]{Definition}

\newtheorem{remark}[theorem]{Remark}
\usepackage{enumitem}

\newcommand{\angled}[1]{\langle#1\rangle}

\newcommand {\mr}{\mathrm}

\newcommand{\act}{\curvearrowright}

\newcommand{\Xa}{X^{\ast}}
\newcommand{\Pa}{\Pi^{\ast}}
\newcommand{\Xb}{X^{\scalebox{0.53}{$\triangle$}}}

\newcommand{\ui}[1]{u_{#1}^{-1}}
\newcommand{\di}[1]{d_{#1}^{-1}}
\newcommand{\sk}[2]{#1^{(#2)}}

\newcommand{\HOdiscII}{Theorem 2.8}
\newcommand{\morse}{Theorem 3.9}
\newcommand{\corfourthree}{Corollary 4.3}
\newcommand{\lemfourseven}{Lemma 4.7}
\newcommand{\lemfivetwo}{Lemma 5.2}
\newcommand{\lemfivethree}{Lemma 5.3}
\newcommand{\lemfivesix}{Lemma 5.6}
\newcommand{\lemfiveseven}{Lemma 5.7}
\newcommand{\lemfiveeight}{Lemma 5.8}
\newcommand{\lemfiveeleven}{Lemma 5.11}
\newcommand{\lemfivefourteen}{Lemma 5.13}
\newcommand{\lemfivesixteen}{Lemma 5.15}
\newcommand{\lemfiveseventeen}{Lemma 5.16}
\newcommand{\lemfivenineteen}{Lemma 5.18}
\newcommand{\lemsixtwo}{Lemma 6.3}
\newcommand{\lemsixthree}{Lemma 6.4}
\newcommand{\lemsixsix}{Lemma 6.7}
\newcommand{\lemsixseven}{Theorem 6.1}
\newcommand{\lemsevenseven}{Theorem 7.7}

\newcommand{\im}{\operatorname{Im}}

\begin{document}

\title[
Quasi-Euclidean tilings and their applications
]{
Quasi-Euclidean tilings over $2$--dimensional Artin groups and their applications
}

\author{Jingyin Huang}
\address{Max Planck Institute for Mathematics, Vivatsgasse 7, 53111 Bonn, Germany}
\email{jingyin@mpim-bonn.mpg.de}

\author{Damian Osajda}
\address{Instytut Matematyczny,
Uniwersytet Wroc\l awski\\
pl.\ Grun\-wal\-dzki 2/4,
50--384 Wroc{\l}aw, Poland}
\address{Institute of Mathematics, Polish Academy of Sciences\\
\'Sniadeckich 8, 00-656 War\-sza\-wa, Poland}
\email{dosaj@math.uni.wroc.pl}

\subjclass[2010]{{20F65, 20F36, 20F67, 20F69}} \keywords{quasi-isometric rigidity, two-dimensional Artin group, metric systolicity}

\date{\today}

\begin{abstract}
We describe the structure of quasiflats in two-dimensio\-nal Artin groups. 
We rely on the notion of
metric systolicity developed in our previous work. Using this weak form of non-positive curvature
and analyzing in details the combinatorics of tilings of the plane we describe precisely 
the building blocks for quasiflats in all two-dimensional Artin groups -- atomic sectors.
This allows us to provide useful quasi-isometry invariants for such groups -- completions of atomic sectors,  stable lines, and the intersection pattern of certain abelian subgroups.
These are described combinatorially, in terms of the structure of the graph defining
an Artin group. As an important tool, we introduce an analogue of the curve complex in the context of two-dimensional Artin groups -- the intersection graph. We show quasi-isometric invariance of the intersection graph under natural assumptions.

As immediate consequences we present a number of results concerning quasi-isometric rigidity for
the subclass of CLTTF Artin groups. 
We give a necessary and sufficient condition for such groups to be strongly rigid (self quasi-isometries are close to automorphisms), 
we describe quasi-isometry groups, we indicate when quasi-isometries imply isomorphisms 
for such groups.
In particular, there exist many strongly rigid large-type Artin groups. In contrast, none of the right-angled Artin groups are strongly rigid by a previous work of Bestvina, Kleiner and Sageev. 
\end{abstract}

\maketitle

\setcounter{tocdepth}{2}
\tableofcontents

\section{Introduction}
\label{s:intro}
\subsection*{Overview}
Let $\Gamma$ be a finite simple graph with each edge labeled by an integer $\ge 2$. An \emph{Artin group with defining graph $\Gamma$}, denoted $A_\Gamma$, is given by the following presentation. Generators of $A_\Gamma$ are in one to one correspondence with vertices of $\Gamma$, and there is a relation of the form
\begin{center}
$\underbrace{aba\cdots}_{m}=\underbrace{bab\cdots}_{m}$
\end{center}
whenever two vertices $a$ and $b$ are connected by an edge labeled by $m$. 

Despite the seemingly simple presentation, the most basic questions (torsion, center, word problem and cohomology) on Artin groups remain open, though partial results are obtained by various authors. We refer to the survey papers by Godelle and Paris \cite{MR3203644}, and McCammond \cite{jonproblems}. Other fundamental and natural questions for Artin groups which are equally exciting and difficult can be found in Charney \cite{charney2016problems}.

One common feature of various classes of Artin groups studied so far is that they all exhibit features of non-positive curvature of certain form:
\begin{itemize}
	\item small cancellation \cite{AppelSchupp1983,Appel1984,Pride,Peifer};
	\item bi-automaticity, automaticity, or some form of combing \cite{Charney1992,MR1314589,charney2003k,MR2208796,holt2011artin,MR2985512,MR3351966};
	\item other notions of combinatorial non-positive curvature \cite{Bestvina1999,Artinsystolic};
	\item $CAT(0)$ \cite{charney1995k,BradyMcCammond2000,brady2002two,bell2005three,brady2010braids,haettel20166};
	\item hierarchical hyperbolicity (hence coarse median) \cite{CharneyCrispAutomorphism,MR3650081,gordon2004artin,behrstock2015hierarchically};
	\item relative hyperbolicity \cite{kapovich2004relative,charney2007relative};
	\item acylindrical hyperbolicity \cite{calvez2016acylindrical,charney2019artin} and hyperbolicity in a statistical sense \cite{cumplido2017loxodromic,yang2016statistically}.
\end{itemize}
Conjecturally, all Artin groups should be non-positively curved in an appropriate sense, and this is intertwined with understanding many fundamental questions about Artin groups.


As most of the Artin groups have many abelian subgroups intersecting in a highly non-trivial way \cite{davis2017determining}, it is natural to compare them with other \textquotedblleft higher rank spaces with non-positive-curvature features\textquotedblright\, like symmetric spaces of non-compact type of rank $>1$, Euclidean buildings, mapping class groups and Teichmuller spaces of surfaces etc., and ask how much of the properties on geometry and rigidity of these spaces still hold for Artin groups and what are the new phenomena for Artin groups. 



Motivated by such considerations, we study Gromov's program of understanding quasi-isometric classification and rigidity of groups and metric spaces in the realm of \emph{$2$--dimensional} Artin groups. We build upon a previous result \cite{Artinmetric}, where it was shown that all $2$--dimensional Artin groups satisfy a form of non-positive curvature called metric systolicity. In the current paper, we will show that certain \textquotedblleft $0$--curvature chunks\textquotedblright\ of the group have a very specific structure, which gives rise to rigidity results.

\subsection*{Background}
Previous works on quasi-isometric rigidity and classification of Artin groups fall into the following three classes, listed from the most rigid to the least rigid situation.
\begin{enumerate}
	\item Some affine type Artin groups are commensurable to mapping class groups of surfaces \cite{CharneyCrispAutomorphism}. Hence the quasi-isometric rigidity results for mapping class groups \cite{behrstock2012geometry,hamenstaedt2005geometry} apply to them.
	\item Atomic right-angled Artin groups \cite{bestvina2008asymptotic} and their right-angled generalizations beyond dimension 2 \cite{MR3692971,huang2016groups,huang2016commensurability,huang2016quasi}.
	\item Artin groups whose defining graphs are trees \cite{behrstock2008quasi} (they are fundamental groups of graph manifolds) and a right-angled generalization beyond dimension 2 \cite{MR2727658}.
\end{enumerate}
Most of the mapping class groups of surfaces enjoy the strong quasi-isometric rigidity property that any quasi-isometry of the group to itself is uniformly close to an automorphism. However, it follows implicitly from \cite[Section 11]{bestvina2008asymptotic} that none of the right-angled Artin groups satisfies such form of rigidity (see \cite[Example 4.14]{MR3692971} for a more detailed explanation). 

One is more likely to obtain stronger rigidity result when the complexity of the intersection pattern of flats in the space is higher. Thus we are motivated to study Artin groups which are not necessarily right-angled, whose structure is generally more intricate than the one of right-angled ones. One particular case is the class of \emph{large-type} Artin group, where the label of each edge in the defining graph is $\ge 3$. A commensuration rigidity result was proved by Crisp \cite{MR2174269} for certain class of large-type Artin groups. 

An Artin group is \emph{$n$--dimensional} if it has cohomological dimension $n$. One-dimensional Artin groups are free groups. The class of two-dimensional Artin groups is much more abundant -- all large-type Artin groups have dimension $\le 2$. By Charney and Davis \cite{CharneyDavis}, $A_\Gamma$ has dimension $\le 2$ if and only if for any triangle $\Delta\subset\Gamma$ with its sides labeled by $p,q,r$, we have $\frac{1}{p}+\frac{1}{q}+\frac{1}{r}\le 1$. A model example to keep in mind is $A_\Gamma$ with $\Gamma$ being a complete graph with all its edges labeled by $3$.

\subsection*{Structure of quasiflats}
Let $A_\Gamma$ be a two-dimensional Artin group. We study $2$--dimensional quasiflats in $A_\Gamma$, since the structure of top-dimensional quasiflats often plays a fundamental role in quasi-isometric rigidity of non-posi\-ti\-ve\-ly curved spaces, see the list of references after Theorem~\ref{thm:intro1}.


Let $\Xa_\Gamma$ be the universal cover of the standard presentation complex of $A_\Gamma$. Modulo some technical details, quasiflats can be viewed as subcomplexes of $\Xa_\Gamma$ which are homeomorphic to $\mathbb R^2$ and are quasi-isometric to $\mathbb E^2$ with the induced metric. Such subcomplexes are called \emph{quasi-Euclidean tilings over $\Xa_\Gamma$} and studying such tilings of $\mathbb R^2$ is of independent interests. This viewpoint is motivated by both the geometric aspects of quasiflats explored in \cite{bks} and diagrammatic aspects studied in \cite{AppelSchupp1983,Pride,olshanskii2017flat}.

To construct such a tiling, one can first search for tilings of Euclidean sectors appearing in $\Xa_\Gamma$ (see Figure~\ref{f:atomic} for some examples), then glue these sectors along their boundaries in a cyclic fashion to form a quasi-Euclidean tiling. The following theorem says that this is essentially the only way one can obtain a quasiflat. Inside $\Xa_\Gamma$ one can build a list of sectors which naturally arise from certain abelian subgroups, or centralizers of certain elements in $A_\Gamma$. We call them \emph{atomic sectors} (cf.\ Section~\ref{subsec:atomic sector}, in particular, Table~\ref{t:atomic} on page \pageref{t:atomic}) since they are building blocks of quasiflats. Some of them are shown in Figure~\ref{f:atomic}.
\begin{theorem}[=Theorem~\ref{thm:main2}]
	\label{thm:intro1}
Suppose $A_\Gamma$ is two-dimensional. Then any $2$--dimensional quasiflat $Q$ of $\Xa_\Gamma$ is at finite Hausdorff distance away from a union of finitely many atomic sectors in $\Xa_\Gamma$.
\end{theorem}

\begin{figure}[h!]
	\centering
	\includegraphics[width=0.75\textwidth]{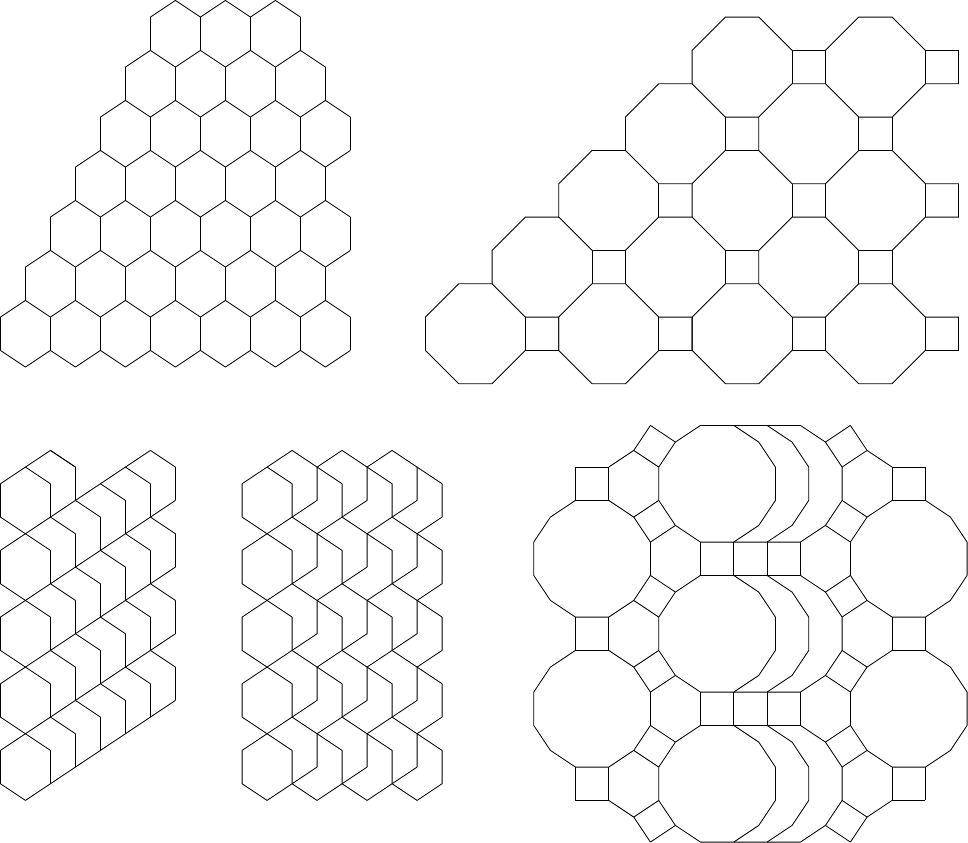}
	\caption{Some atomic sectors.}
	\label{f:atomic}
\end{figure}

A simpler but less concise restatement of the above theorem is that every 2-dimensional quasiflat is made of \textquotedblleft fragments\textquotedblright\ of subgroups of $A_\Gamma$ of form $F_n\times F_m$ ($n,m\ge 1$), where $F_n$ denotes the free group with $n$ generators.

Previously, structure theorems for quasiflats were proved for Euclidean buildings and symmetric spaces of non-compact type \cite{kleiner1997rigidity,eskin1997quasi,wortman2006quasiflats}, universal covers of certain Haken manifolds \cite{kapovich1997quasi}, certain CAT(0) complexes \cite{bks,MR3654109} and hierarchically hyperbolic spaces \cite{behrstock2017quasiflats}.

Our viewpoint of tilings is convenient for, simultaneously, analyzing relevant group theoretic information and studying geometric aspects of quasiflats. The combinatorics of the group structure make up a substantial part. This is quite different from the more geometric situation in \cite{bks,MR3654109,behrstock2017quasiflats}. Some of the atomic sectors are actually half-flats (see Definitions~\ref{def:DCH}--\ref{def:PCH} below). This might seems strange compared to several previous quasiflat results, however, these sectors arise naturally when considering the group structure and they are convenient for our later applications.

Atomic sectors are not necessarily preserved under quasi-isometries. However, atomic sectors have natural \emph{completions} (cf.\ Section~\ref{subsec:completion}) which are preserved under quasi-isometries (see Corollary~\ref{cor:qi and completion} for a precise statement). 

The following theorem says that certain $\mathbb Z$--subgroups corresponding to boundaries of atomic sectors are preserved under quasi-isometries. These $\mathbb Z$--subgroups act on what we call \emph{stable lines}, and they are analogues of $\mathbb Z$--subgroups generated by Dehn twists in mapping class groups.

\begin{theorem}[=Theorem~\ref{thm:preservation of stable lines}]
	\label{thm:stable line}
	Suppose $A_{\Gamma_1}$ and $A_{\Gamma_2}$ are two-dimensional Artin groups. Let $q\colon\Xa_{\Gamma_1}\to \Xa_{\Gamma_2}$ be an $(L,A)$--quasi-isometry. Then there exists a constant $D$ such that for any stable line $L_1\subset \Xa_{\Gamma_1}$, there is a stable line $L_2\subset \Xa_{\Gamma_2}$ with $d_H(q(L_1),L_2)<D$.
\end{theorem}

For each $2$--dimensional Artin group $A_\Gamma$, we define an \emph{intersection graph} $\I_\Gamma$ (Definition~\ref{def:intersection graph}), which describes the intersection pattern of certain abelian subgroups. This object can be viewed as an analogue of the spherical building at infinity for symmetric spaces of non-compact type, or the curve graph in the case of mapping class groups. When $\Gamma$ is a triangle with its edges labeled by $3$, the group $A_\Gamma$ is commensurable to the mapping class group of the $5$--punctured sphere \cite{CharneyCrispAutomorphism}, and $\I_\Gamma$ is isomorphic to the curve complex of the $5$--punctured sphere.
\begin{theorem}
	\label{thm:intro2}
	Let $A_{\Gamma}$ and $A_{\Gamma'}$ be two $2$--dimensional Artin groups and let $q\colon A_\Gamma\to A_{\Gamma'}$ be a quasi-isometry. Then $q$ induces an isomorphism from the stable subgraph of $\I_{\Gamma}$ onto the stable subgraph of $\I_{\Gamma'}$. In particular, if both $A_\Gamma$ and $A_{\Gamma'}$ have finite outer automorphism group and $\Gamma$ has more than two vertices, then $q$ induces an isomorphism between their intersection graphs.
\end{theorem}
We refer to Theorem~\ref{thm:invariance} for a more general version of Theorem~\ref{thm:intro2}.

Analogously, quasi-isometries of higher rank symmetric spaces of non-compact type and Euclidean buildings induce automorphisms of their spherical buildings at infinity \cite{kleiner1997rigidity,eskin1997quasi}; and quasi-isometries of most mapping class groups induce automorphisms of their curve complexes \cite{hamenstaedt2005geometry,behrstock2012geometry,behrstock2017quasiflats}.

Recall that a finitely generated group $H$ is \emph{virtually} $A_\Gamma$ if there exists a finite index subgroup $H'\le H$ and a homomorphism $\phi\colon H'\to A_\Gamma$ with finite kernel and finite index image.

\begin{corollary}
	\label{cor:injective}
	Let $A_{\Gamma}$ be a $2$--dimensional Artin group with finite outer automorphism group. Suppose $\Gamma$ has more than two vertices. If $q\colon A_\Gamma\to A_\Gamma$ is an $(L,A)$--quasi-isometry inducing the identity map on $\I_\Gamma$, then there exists $C=C(L,A,\Gamma)$ such that $d(q(x),x)\le C$ for any $x\in A_\Gamma$. Thus the map $\QI(A_\Gamma)\to \Aut(\I_\Gamma)$ described in Theorem~\ref{thm:intro2} is an injective homomorphism.
	
	Thus if we know in addition that the homomorphism $A_\Gamma\to \Aut(\I_\Gamma)$ induced by the action $A_\Gamma\act \Aut(\I_\Gamma)$ has finite index image, then any finitely generated groups quasi-isometric to $A_\Gamma$ is virtually $A_\Gamma$.
\end{corollary}

We refer to Corollary~\ref{cor:injective2} for a more general version of Corollary~\ref{cor:injective2}.

Unlike the case of mapping class groups, $\Aut(\I_\Gamma)$ could be much larger than $\QI(A_\Gamma)$ in the context of Corollary~\ref{cor:injective}. This happens for example in the right-angled case \cite[Corollary 4.20]{MR3692971}. However, we expect more rigidity to occur outside the right-angled world -- see the next subsection for some special cases.

\subsection*{Rigidity of certain large-type Artin groups}
We now discuss some immediate consequences of the results from the previous subsection. We will restrict ourselves to the class of \emph{CLTTF} Artin groups introduced by Crisp \cite{MR2174269} in order to use his results directly. Nevertheless, we believe that similar goals for more general classes of Artin groups can be achieved using the invariants described above. See the list of questions in Section~\ref{rmk:ending remark} for possible further directions.

\begin{definition}
	\label{def:CLTTF}
An Artin group $A_\Gamma$ is CLTTF if all of the following conditions are satisfied:
\begin{itemize}
	\item (C) $\Gamma$ is connected and has at least three vertices;
	\item (LT) $A_\Gamma$ is of large type;
	\item (TF) $\Gamma$ is triangle-free, i.e.\ $\Gamma$ does not contain any triangles.
\end{itemize}
\end{definition}
In particular, CLTTF Artin groups are $2$--dimensional. As pointed out by Crisp, (C) serves to rule out $2$--generator Artin groups, which are best treated as a separate case. Here, we improve the commensurability rigidity of \cite[Theorem 3]{MR2174269} to quasi-isometric rigidity results and point out some new phenomena in the setting of quasi-isometries.

\begin{theorem}[=Corollary~\ref{cor:qi and iso}]
	\label{thm:introduction 3}
Let $A_\Gamma$ and $A_{\Gamma'}$ be CLTTF Artin groups. Suppose $\Gamma$ does not have separating vertices and edges. Then $A_\Gamma$ and $A_{\Gamma'}$ are quasi-isometric if and only if $\Gamma$ and $\Gamma'$ are isomorphic as labeled graphs. 
\end{theorem}

By \cite[Theorem 1]{MR2174269}, a CLTTF Artin group has finite outer automorphism group if and only if its defining graph has no separating vertices and edges. So Theorem~\ref{thm:introduction 3} can be compared to \cite[Theorem 1.1]{MR3692971}.

\begin{theorem}[=Theorem~\ref{thm:CLTTF2}]
	\label{thm:intro4}
	Let $A_\Gamma$ be a CLTTF Artin group such that $\Gamma$ does not have separating vertices and edges. Let $\QI(A_\Gamma)$ be the quasi-isometry group of $A_\Gamma$. Let $\Isom(A_\Gamma)$ be the isometry group of $A_\Gamma$ with respect to the word distance for the standard generating set. Then the following hold.
	\begin{enumerate}
		\item Any quasi-isometry from $A_\Gamma$ to itself is uniformly close to an element in $\Isom(A_\Gamma)$.
		\item There are isomorphisms $\QI(A_\Gamma)\cong \Isom(A_\Gamma)\cong \Aut(D_\Gamma)$, where $\Aut(D_\Gamma)$ is the simplicial automorphism group of the Deligne complex $D_\Gamma$ of $A_\Gamma$.
	\end{enumerate}
\end{theorem}

There are counterexamples if we drop the condition that $\Gamma$ does not have separating vertices and edges, see \cite[Lemma 42]{MR2174269}.

If $A_\Gamma$ is right-angled then $\QI(A_\Gamma)$ is much smaller than $\Aut(D_\Gamma)$, see \cite[Corollary 4.20]{MR3692971}. This suggests that the Artin groups in Theorem~\ref{thm:intro4} are more rigid than right-angled Artin groups. On the other hand, the Artin groups in Theorem~\ref{thm:intro4} may not be as rigid as mapping class groups, since elements in $\Isom(A_\Gamma)$ are not necessarily uniformly close to automorphisms of $A_\Gamma$. A group $G$ is \emph{strongly rigid} if any element in $\QI(G)$ is uniformly close to an element in $\Aut(G)$. Now we characterize all strongly rigid members of the class of large-type and triangle-free Artin groups.

\begin{theorem}[=Theorem~\ref{thm:largechar2}]
	\label{thm:largechar}
	Let $A_\Gamma$ be a large-type and triangle-free Artin group. Then $A_\Gamma$ is strongly rigid if and only if $\Gamma$ satisfies all of the following conditions:
	\begin{enumerate}
		\item $\Gamma$ is connected and has $\ge 3$ vertices;
		\item $\Gamma$ does not have separating vertices and edges;
		\item any label preserving automorphism of $\Gamma$ which fixes the neighborhood of a vertex is the identity.
	\end{enumerate}
Moreover, if a large-type and triangle-free Artin group $A_\Gamma$ satisfies all the above conditions and $H$ is a finitely generated group quasi-isometric to $A_\Gamma$, then $H$ is virtually $A_\Gamma$.
\end{theorem}

\begin{remark}
We note that the above quasi-isometric rigidity results do not use the full strength of Theorem~\ref{thm:intro1} and Theorem~\ref{thm:stable line}, since proving the two theorems in the special cases of triangle-free Artin groups and Artin groups of hyperbolic type in the sense of \cite{MR2174269} are much easier (though still non-trivial). Moreover, if an Artin group is triangle-free, then it acts geometrically on a $2$--dimensional CAT(0) complex \cite{brady2000three}, and it is C(4)-T(4) with respect to the standard presentation \cite{Pride}. There are alternative starting points of studying quasiflats in such complexes, either by using \cite{bks}, or using \cite{olshanskii2017flat}. However, several new ingredients are needed to deal with the general $2$--dimensional case. We present the above quasi-isometric rigidity results in order to demonstrate the potential of using Theorem~\ref{thm:intro1} and Theorem~\ref{thm:stable line} to obtain similar results for more general classes of Artin groups.
\end{remark}


\subsection*{Comments on the proof}
First we discuss the proof of Theorem~\ref{thm:intro1}. Let $\Xa_\Gamma$ be the universal cover of the standard presentation complex of a $2$--dimensional Artin group $A_\Gamma$. To avoid technicalities, we assume the quasiflat is a subcomplex $Q$ of $\Xa_\Gamma$ homeomorphic to $\mathbb R^2$ and quasi-isometric to $\mathbb E^2$, and we would like to understand the tiling of $Q$.
\medskip

\noindent
\emph{Step 0:} The general idea is to use geometry of $A_\Gamma$ to control the tiling of $Q$. We start by showing that $A_\Gamma$ is non-positively curved in an appropriate sense. We would like to use CAT(0) geometry, however, at the time of writing, it is not known whether all $2$--dimensional Artin groups are CAT(0). Moreover, \cite{brady2002two} implies that if certain $2$--dimensional Artin groups act geometrically on CAT(0) complexes, the dimension of the complexes is $\ge 3$. There are quite non-trivial technicalities in studying $2$--quasiflats in higher dimensional complexes and relating them to the combinatorics of groups, which we would like to bypass.

In \cite{Artinmetric}, we built a geometric model $X_\Gamma$ for $A_\Gamma$ with features of both CAT(0) geometry and two-dimensionality. More precisely, $X_\Gamma$ is a thickening of $\Xa_\Gamma$. We equip the $2$--skeleton $\sk{X_\Gamma}{2}$ with a metric and it turns out that $X_\Gamma$ becomes non-positively curved in the sense that any $1$--cycle can be filled by a CAT(0) disc in $\sk{X_\Gamma}{2}$.
Such a complex $X_\Gamma$ is an example of a \emph{metrically systolic complex}. 

\medskip

\noindent
\emph{Step 1.} We study the local structure of $Q$ and show that outside a compact set, $Q$ is locally flat in an appropriate combinatorial sense.

\medskip

\noindent
\emph{Step 1.1}. We approximate $Q$ by a CAT(0) subcomplex $Q'$ in $X_\Gamma$ such that $Q'$ is homeomorphic to $\mathbb R^2$. To do that, we pick larger and larger discs in $Q$, and we replace them by minimal discs in $\sk{X_\Gamma}{2}$, which are CAT(0), then we take a limit. By a version of Morse Lemma proved in \cite{Artinmetric}, $Q'$ is at bounded Hausdorff distance from $Q$. An argument on area growth implies that $Q'$ is flat outside a compact set. 
In fact, we obtain the following result in the more general setting of all metrically systolic complexes (see Theorem~\ref{thm:MSquasiflats} for the detailed statement). 
\begin{theorem}(cf. Theorem~\ref{thm:MSquasiflats})
	\label{thm:intro6}
	Every quasiflat in a metrically systolic complex $X$ can be uniformly approximated by a simplicial map $q'\colon Y \to X$ from a CAT(0) triangulation $Y$ of $\mathbb R^2$ being flat outside a compact set such that $q'$ is injective on neighbourhoods of flat vertices.
\end{theorem}
Recently, independently, Elsner \cite{elsner2017} has obtained an analogous result for systolic complexes. Theorem~\ref{thm:intro6} applies to a larger class of spaces, however, Elsner's result provides better control on the structure of quasiflats in the systolic setting.
\medskip

\noindent
\emph{Step 1.2}. Though $Q'$ is CAT(0) and flat outside a compact set, the cell structure on $Q'$ is not ideal. A priori there may be triangles in $Q'$ such that their angles are irrational multiples of $\pi$, and the cell structure is not compatible with the intersection pattern of quasiflats of $A_\Gamma$. So we go back to $\Xa_\Gamma$, whose combinatorial structure is simpler than the one of $X_\Gamma$. By the design of $X_\Gamma$, there is 
a partial retraction $\rho$ from $X_\Gamma$ (defined on a subcomplex of $X_\Gamma$) to $\Xa_\Gamma$. We argue that outside a compact subset, $\rho$ is well-defined on $Q'$. Let $Q''=\rho(Q')$. Then $Q''$ is a subcomplex with a ``hole" (since $\rho$ may not be defined on all of $Q'$). The map $\rho$ transfers information on local structure of $Q'$ to information on local structure of $Q''$ (cf.\ Lemma~\ref{lem:star homeo}). This step is discussed in Section~\ref{sec:geometric model} and Section~\ref{sec:local flatness}.
\medskip

\noindent
\emph{Step 2.} So far, we have a subcomplex $Q''$ about which we know that its local structure belongs to one of several types (as in Lemma~\ref{lem:star homeo}). The next goal is to run a local-to-global argument in order to produce the sectors as required. 
\medskip

\noindent
\emph{Step 2.1.} We first produce a boundary ray for the sector, which is easier than producing the whole sector. To guess what kind of ray it could be, we look at a motivating example $X=\mathrm{SL}(3,\mathbb R)/\mathrm{SO}(3,\mathbb R)$. Quasiflats in $X$ are Hausdorff close to a finite union of Weyl cones \cite{kleiner1997rigidity,eskin1997quasi}. The boundary of a Weyl cone is a singular ray, i.e.\ this ray is contained in the intersection of two flats. Analogously, we search for \emph{singular rays} in $X_\Gamma$, which are possibly contained in the intersection of two free abelian subgroups of rank $2$. A list of plausible singular rays is given in Section~\ref{sec:singular lines}.
\medskip

\noindent
\emph{Step 2.2.} We show that there is at least one singular ray in $Q''$. The strategy is to start at one vertex of $Q''$, then walk away from this vertex and collect information about local landscape from Step 1.2 along the way to decide where to go further. Note that there is a \textquotedblleft hole\textquotedblright\ in $Q''$ we want to avoid. This can be handled in the following way. Paths in $Q''$ have shadows in $Q'$. Thus we can use the CAT(0) geometry in $Q'$ to control where to go so we can avoid the ``hole". Section~\ref{sec:half coxeter regions}, Section~\ref{sec:orientation of edges of Coxeter regions} and Section~\ref{sec:proof Existence of singular rays} are devoted to the proof of the existence of singular rays in $Q''$. Another thing we need to be careful about is that the local structure discussed in Step 1.2 concerns not just the shapes of the cells around one vertex, but also the orientation of the edges around that vertex. In Section~\ref{sec:orientation of edges of Coxeter regions} we deal with the orientation issues. 
\medskip

\noindent
\emph{Step 2.3.} Now we know there is at least one singular ray in $Q''$. We start with this singular ray, and use the local characterization from Step 1.2 to \textquotedblleft sweep out\textquotedblright\ a sector which ends in another singular ray. This uses a development argument. Now we iterate this process. It turns out that each sector in $Q''$ corresponds to a CAT(0) sector in $Q'$ with angle bounded below by a uniform number. Since $Q'$ has quadratic growth, after finitely many steps we produce enough sectors to fill $Q''$, which finishes the proof of Theorem~\ref{thm:intro1}. This step is discussed in Section~\ref{sec:development}.
\medskip

Finally, we discuss how to deduce the quasi-isometric rigidity theorems for CLTTF Artin groups from Theorem~\ref{thm:intro1}. Again, we proceed in several steps. The first two steps can be generalized in an appropriate way to all $2$--dimensional Artin groups (cf. Theorem~\ref{thm:preservation of stable lines} and Theorem~\ref{thm:invariance}).
\medskip

\noindent
\emph{Step 1:} Let $A_\Gamma$ be a CLTTF Artin. We consider the collection of maximal cyclic subgroups of $A_\Gamma$ with centralizers commensurable to $F_2\times\mathbb Z$. We show that these cyclic subgroups are preserved under quasi-isometries. 
\medskip

\noindent
\emph{Step 2:} Let $\Theta_\Gamma$ be a graph whose vertices correspond to cyclic subgroups from the previous step, and two vertices are adjacent if the associated cyclic subgroups generate a free abelian subgroup of rank $2$. This graph was introduced by Crisp \cite{MR2174269}. We show that any quasi-isometry from $A_\Gamma$ to itself induces an automorphism of $\Theta_\Gamma$.
\medskip

\noindent
\emph{Step 3:} Crisp \cite{MR2174269} showed that any automorphism of $\Theta_\Gamma$ is induced by a canonical bijection from $A_\Gamma$ to itself, under appropriate additional conditions. In such way we approximate quasi-isometries by maps which preserve more combinatorial structure, and then deduce the quasi-isometric rigidity results listed above, see Section~\ref{sec:CLTTF}.

\subsection*{On the structure of the paper}
In Section~\ref{sec:quasiflats in MS complex} we analyze the structure of quasiflats in metrically systolic complexes and we prove Theorem~\ref{thm:intro6} above (Theorem~\ref{thm:MSquasiflats} in the text). 
In Sections~\ref{sec:geometric model}--~\ref{sec:development} we describe the structure of complexes
approximating quasiflats in two-dimensional Artin groups. 
In Section~\ref{sec:development} we prove Theorem~\ref{thm:intro1} (Theorem~\ref{thm:main2}).
In Section~\ref{sec:application} we prove Theorem~\ref{thm:stable line} (Theorem~\ref{thm:preservation of stable lines}) and Theorem~\ref{thm:intro2} (Theorem~\ref{thm:invariance}).
In Section~\ref{sec:CLTTF} we provide proofs of the results concerning CLTTF Artin groups: Theorem~\ref{thm:introduction 3} (Corollary~\ref{cor:qi and iso}), Theorem~\ref{thm:intro4} (Theorem~\ref{thm:CLTTF2}), and Theorem~\ref{thm:largechar} (Theorem~\ref{thm:largechar2}).

\subsection*{Acknowledgment}
J.H.\ would like to thank M.\ Hagen for a helpful discussion in July 2014, and R. Charney and T. Haettel for a helpful discussion in June 2016. J.H.\ also thanks the Max-Planck Institute for Mathematics, where part of the work was done, for its hospitality. Part of the work on the paper was carried out while D.O.\ was visiting McGill University.
We would like to thank the Department of Mathematics and Statistics of McGill University
for its hospitality during that stay.
The authors were partially supported by (Polish) Narodowe Centrum Nauki, grants no.\ UMO-2015/\-18/\-M/\-ST1/\-00050 and UMO-2017/25/B/ST1/01335. This work was partially supported by the grant 346300 for IMPAN from the Simons Foundation and the matching 2015-2019 Polish MNiSW fund.

\section{Quasiflats in metrically systolic complexes}
\label{sec:quasiflats in MS complex}

We start with several notations. Let $X$ be a combinatorial cell complex. For a subset $Y\subset X$, the \emph{carrier} of $Y$ in $X$ is the union of closed cells in $X$ which contain at least one point of $Y$ in their interior. For a vertex $x\in X$, the \emph{closed star} of $v$ in $X$, denoted by $\St(x,X)$, is the union of closed cells in $X$ which contain $x$. When $X$ is a piecewise Euclidean polyhedral complex, i.e.\ $X$ is obtained by taking the disjoint union of a family of convex polyhedra in Euclidean spaces of various dimensions and gluing them along isometric faces, see \cite[Definition I.7.37]{BridsonHaefliger1999}, the $\epsilon$--sphere around each vertex of $X$ inherits a natural polyhedral complex structure from $X$, for $\epsilon$ small (see \cite[Definition I.7.15]{BridsonHaefliger1999}). This polyhedral complex is called the \emph{link} of $x$ in $X$, and is denoted by $\lk(x,X)$. If $X$ is a $2$--dimensional simplicial complex, then we always identify $\lk(x,X)$ with the full subgraph of the one-skeleton $X^{(1)}$ of $X$ spanned by vertices adjacent to $x$.

For a metric space $Z$ and a point $z\in Z$, we use $B(z,R)$ to denote the $R$--ball in $Z$ centered at $z$. For a subset $Y\subset Z$, we use $N_R(Y,R)$ to denote the $R$--neighborhood of $Y$ in $Z$.

\subsection{Preliminaries on metrically systolic complexes}
Let $X$ be a flag simplicial complex. We put a piecewise Euclidean structure on $\sk{X}{2}$ (the $2$--skeleton of $X$) in the following way. Since $\sk{X}{2}$ can be viewed as a disjoint collection of simplices with identifications between their faces, we assume every $2$--simplex (triangle) is isometric to a non-degenerate Euclidean triangle and all the identifications are isometries. This gives a length metric on $\sk{X}{2}$, which we denote by $d$. Since we will work with the $2$--skeleton of $X$, for a vertex $v\in X$, we define its \emph{link} to be the full subgraph of $\sk{X}{1}$
spanned by all vertices adjacent to $v$ as explained as above. Every link is equipped with an \emph{angular metric}, defined as follows. For an edge $\overline{v_1v_2}$, we define the \emph{angular length} of this edge to be the angle $\angle_v(v_1,v_2)$ with the apex $v$. This turns the link into a metric graph, and the angular metric, which we denote by $d_\angle$, is the path metric of this metric graph (note that a priori we do not know $\angle_v(v_1,v_2)=d_\angle(v_1,v_2)$ for adjacent vertices $v_1$ and $v_2$). The \emph{angular length} of a path $\omega$ in the link, which we denote by $\length_\angle(\omega)$, is the summation of angular lengths of edges in this path. In this paper we assume that the following weak form of triangle inequality
holds for angular lengths of edges in $X$: for each $v\in X$ and every three pairwise adjacent vertices $v_1,v_2,v_3$
in the link of $v$ we have that $\angle_v(v_1,v_3)\le \angle_v(v_1,v_2)+\angle_v(v_2,v_3)$.
Then we call $X$ (with metric $d$) a \emph{metric simplicial complex}.

\begin{remark}
We allow that the above inequality becomes equality -- intuitively, it corresponds to degenerate $2$--simplices
	in a link, which correspond to degenerate $3$--simplices in $X$.
\end{remark}

For $k=4,5,6,\ldots$, a simple $k$--cycle $C$ in a simplicial complex is \emph{$2$--full} if 
there is no edge connecting any two vertices in $C$ having a common neighbor in $C$ (that is, there
are no \emph{local diagonals}). 

\begin{definition}[Metrically systolic complex]
	\label{d:metric_syst}
	A link in a metric simplicial complex is \emph{$2\pi$--large} if every $2$--full
	simple cycle in the link has angular length at least $2\pi$.
	A metric simplicial complex $X$ is \emph{locally $2\pi$--large} if every its link is $2\pi$--large.
	A simply connected locally $2\pi$--large metric complex is called a \emph{metrically systolic} complex.
\end{definition}

In Theorem~\ref{thm:CAT(0)diagrem} below we state a fundamental property of metrically systolic complexes.
It concerns filling diagrams for cycles, and it is a main tool 
used in proofs of various results about such complexes in subsequent sections and, previously, in \cite{Artinmetric}.

Let $X$ be a simplicial complex.
A \emph{cycle} in $X$ is a simplicial map from a triangulated circle to $X$ which is injective on each edge. A \emph{singular disc} $D$ is a simplicial complex isomorphic to a finite connected and simply connected subcomplex
of a triangulation of the plane. There is the (obvious) \emph{boundary cycle} for $D$, that is, a map
from a triangulation of $1$--sphere (circle) to the boundary of $D$, which is injective on edges. More precisely, we can view $D$ as a subset of the $2$--sphere $\mathbb S^2$. Then $\mathbb S^2\setminus D$ is an open cell, and the boundary of this open cell gives rise to the boundary cycle of $D$. For a cycle $C\colon  K\to X$ in a simplicial complex $X$, a \emph{singular disc diagram for $C$} is a simplicial
map $f\colon D \to X$ from a singular disc $D$ to $X$ such that $C\colon K\to X$ factors through the boundary cycle of $D$. By the relative simplicial approximation theorem \cite{Zeeman1964}, for every cycle in a simply connected simplicial complex there exists a singular disc diagram (cf.\ also van Kampen's lemma e.g.\ in \cite[pp.\ 150-151]{LSbook}). It is an essential feature of metrically systolic complexes that singular disc diagrams may be modified to ones with some additional properties; see Theorem~\ref{thm:CAT(0)diagrem} below.

A singular disc diagram is called \emph{nondegenerate} if it is injective on all simplices.
It is \emph{reduced} if distinct adjacent triangles (i.e., triangles sharing an edge) are mapped 
into distinct triangles. For a metric simplicial complex $X$ and a nondegenerate singular disc diagram $f\colon D \to X$ we equip
$D$ with a metric in which $f|_{\sigma}$ is an isometry onto its image, for every simplex $\sigma$ in $D$. Then, $f$ is a \emph{CAT(0) singular disc diagram} if $D$ is CAT(0), that is, if the angular length of
every link in $D$ being a cycle (that is, the link of an interior vertex in $D$) is at least $2\pi$. 
An internal vertex in $D$ is \emph{flat} when the angular length of its link is exactly $2\pi$.

\begin{theorem}[CAT(0) disc diagram]
	\cite[\HOdiscII]{Artinmetric}
	\label{thm:CAT(0)diagrem}
	Let $f\colon D\to X$ be a singular disc diagram for a cycle $C$ in a metrically systolic
	complex $X$. Then there exists a singular disc diagram $f' \colon D' \to X$ for $C$ such that 
	\begin{enumerate}
		\item $f' \colon D' \to X$ is a CAT(0) nondegenerate reduced disc diagram;
		\item $f'$ does not use any new vertices in the sense that there is an injective map $i$ from the vertex set of $D'$ to the vertex set of $D$ such that $f=f'\circ i$ on the vertex set of $D$;
		\item if $v\in D'$ is a flat interior vertex, then $f'$ is injective on the closed star of $v$ in $D'$.
	\end{enumerate}	
The number of $2$--simplices in $D'$ is at most the number of $2$--simplices in $D$. 
\end{theorem}

The following result is immediate. 
\begin{corollary}
	\label{cor:quadratic}
	Suppose $X$ is a metrically systolic complex with finitely many isometry types of cells. Then there exists a constant $L>0$ such that for each cycle $C$ with $\le n$ edges, there is a singular disc diagram for $C$ with $\le Cn^2$ triangles in the disc diagram and the image of the singular disc diagram is contained in the $(L\cdot n)$--neighborhood of $C$.
\end{corollary}

Note that whenever there is a statement regarding metric on $X$, we always mean the length metric $d$ with respect to the piecewise Euclidean structure on $\sk{X}{2}$. If there are finitely many isometry types of cells in $\sk{X}{2}$, then $(\sk{X}{2},d)$ is a complete geodesic metric space \cite[Theorem I.7.19]{BridsonHaefliger1999}. Moreover, $d$ is quasi-isometric to the path metric on $\sk{X}{1}$ such that each edge has length $1$ \cite[Proposition I.7.31]{BridsonHaefliger1999}.

We also need the following version of the Morse Lemma for discs in metrically systolic complexes. See \cite[\morse]{Artinmetric} for the proof.
\begin{lemma}
	\label{thm:Morse lemma}[Morse Lemma for $2$--dimensional quasi-discs]
	Let $D$ be a combinatorial ball in the Euclidean plane tiled by equilateral triangles. Let
	$f\colon D \to X$ be a disc diagram for a cycle $C$ in $X$ 
	being an $(L,A)$--quasi-isometric embedding. 
	Let $g\colon D' \to X$ be a singular disc diagram for $C$.
	Then $\mr{im}(f)\subseteq N_a(\mr{im}(g))$, where $a>0$ is a constant depending only on $L$ and $A$.
\end{lemma}

\subsection{The structure of quasiflats}
Recall that $B_X(x,r)$ denotes the ball of radius $r$ centered at $x$ in a metric space $X$. Here and elsewhere we use the notation $d_H$ for the Hausdorff distance. The following is a consequence of \cite[Theorem 4.1]{bks}.
\begin{theorem}
	\label{thm:BKS}
Suppose $X$ is a $2$--dimensional piecewise Euclidean CAT(0) complex such that $X$ is homeomorphic to $\mathbb R^2$ and $X$ has finitely many isometry types of cells. If there is a constant $L>0$ and a base point $x\in X$ such that $\area(B_X(x,r))\le Lr^2$, for any $r>0$, then $X$ is flat outside a compact subset.
\end{theorem}

\begin{theorem}
	\label{thm:MSquasiflats}
Let $X$ be a locally finite metrically systolic complex with finitely many isometry types of cells. Let $q\colon \mathbb E^2\to X$ be an $(L,A)$--quasi-isometric embedding. Then there exist a constant $M_0$ depending only on $L,A$ and $X$, a simplicial complex $Y$, and a reduced nondegenerate simplical map $q'\colon Y\to X$ such that
\begin{enumerate}
\item $Y$ is a $2$--dimensional simplicial complex homeomorphic to $\mathbb R^2$;
\item if we endow $Y$ with the piecewise Euclidean structure such that each simplex in $Y$ is isometric to its $q'$--image, then $Y$ is CAT(0);
\item $Y$ is flat outside a compact subset;
\item for any flat vertex $v\in Y$, $q'$ is injective on $\St(v,Y)$; and if we view $\lk(v,Y)$ as a simple close cycle in $X$, then it does not bound a disc diagram without interior vertices;
\item $d_H(\im q,\im q')<M_0$.
\end{enumerate}
\end{theorem}

\begin{proof}
Let $q\colon\mathbb E^2\to X$ be an $(L,A)$--quasi-isometric embedding. We view $\mathbb E^2$ as a simplicial complex which is tiled by equilateral triangles. Using Corollary~\ref{cor:quadratic} and a skeleton by skeleton approximation argument, we can assume $q$ is also $L'$--Lipschitz. Since $X$ has finitely many isometry types of cells, it follows from $q$ being $L'$--Lipschitz and \cite[Lemma I.7.54]{BridsonHaefliger1999} that if the diameter of equilateral triangles in $\mathbb E^2$ is small enough, then the $q$--image of the closed star of a vertex in $\mathbb E^2$ is contained in the closed star of a vertex in $X$. By a standard simplicial approximation argument, we can assume in addition that $q$ is a simplicial map. The new quasi-isometric constants of $q$ depend only on the old constants and $X$.

Pick a base vertex $x\in \mathbb E^2$. Let $D_n\subset \mathbb E^2$ be the full subcomplex spanned by vertices with combinatorial distance $\le n$ from $x$. Up to attaching a thin annulus to $D_n$ along $\partial D_n$, we assume $q_n=q|_{D_{n}}$ maps each edge in $\partial D_n$ to an edge in $X$. As in Theorem~\ref{thm:CAT(0)diagrem}, for each $n$, we modify $q_n\colon D_n\to X$ to obtain a reduced nondegenerate CAT(0) singular disc diagram $q'_n\colon D'_n\to X$ such that $q'_n$ and $q_n$ have the same boundary cycle. Moreover, we assume $q'_n\colon D'_n\to X$ has the least number of triangles among all singular disc diagrams satisfying the conditions of Theorem~\ref{thm:CAT(0)diagrem} and has the same boundary cycle as $q_n$. Then
\begin{itemize}
\item $d_H(\im q'_n, \im q_n)<M$ with $M$ depending only on $L$ and $A$;
\item Theorem~\ref{thm:MSquasiflats} (4) holds for flat interior vertices in $D'_n$.
\end{itemize}
The first property follows from Lemma~\ref{thm:Morse lemma}. For the second property, if $\lk(v,D'_n)$ bounds a singular disc diagram without interior vertices for some flat interior vertex $v\in D'_n$, then we can replace $\St(v,D'_n)$ with such diagram to obtain a singular disc diagram $\bar q_n\colon\bar D_n\to X$ with fewer triangles. Now, use Theorem~\ref{thm:CAT(0)diagrem} to modify $\bar q_n$ further to obtain a singular disc diagram $\hat q_n\colon\widehat D_n\to X$. Then $\widehat D_n$ contradicts the minimality of $D'_n$.

Let $A_0=\max\{L,A,L',M\}$. For $R\ge 100 A_0$ and $n\ge 100A_0R$, let $K_{R,n}\subset D'_n$ be the largest possible subcomplex such that $q'_n(K_{R,n})$ is contained in the $R$--ball $B_X(q(x),R)$ of $X$ centered at $q(x)$. Then the following hold:
\begin{enumerate}[label=(\alph*)]
	\item $K_{R,n}$ is locally CAT(0);
	\item there is $M'$ depending only on $L$ and $A$ such that $q(\mathbb E^2)\cap B_X(q(x),\frac{R}{2})$ is contained in the $M'$--neighborhood of $q'_n(K_{R,n})$; 
	\item $\lk(v,K_{R,n})$ is a circle for any vertex $v\in K_{\frac{R}{2},n}$;
\end{enumerate}
(b) follows from the inequality $d_H(\im q'_n, \im q_n)<M$, and (c) follows from the fact that $q'_n(v)$ is far away from $q'_n(\partial D'_n)$. Since $q'_n$ does not use new vertices in the sense of Theorem~\ref{thm:CAT(0)diagrem}, the cardinality of $K^{(0)}_{R,n}$ is $\le$ the number of vertices in $\mathbb E^2$ whose $q$--images are contained in $B_X(q(x),R)$. Moreover, $X$ is locally finite by our assumption. Thus by passing to a subsequence, we assume for any $n$ and $m$, there is a simplicial isomorphism $\phi_{n,m}\colon K_{R,n}\to K_{R,m}$ such that $q'_n=q'_m\circ \phi_{n,m}$ on $K_{R,n}$. Now we let $R\to\infty$ and use a diagonal argument to produce $q'\colon Y\to X$ such that for any $R$ and any subcomplex $K\subset Y$ such that $q'(K)$ is contained in the $R$--ball of $x$, there exists $n$ and a simplicial embedding $\phi_K\colon K\to D'_n$ such that $q'=q'_n\circ \phi_K$ on $K$. 

Now we show that $q'$ satisfies all the requirements. First we show $Y$ is simply-connected. Take a closed curve $C\subset Y$ and take $R$ such that $q'(C)\subset B_{X}(q(x),\frac{R}{100})$. Let $K\subset Y$ be the maximal subcomplex such that $q(K)\subset B_X(q(x),R)$ and let $D'_n$ be as above. Since $D'_n$ is CAT(0), we can find a geodesic homotopy in $D'_n$ that contracts $C$ to a point in $C$. Note that $q'_n$ is Lipschitz since it is simplicial. Thus the $q$--image of this homotopy is contained in $B_X(q(x),R)$, and the homotopy actually happens inside $K$. This together with property (a) above implies that $Y$ is CAT(0). By (c), $Y$ is homeomorphic to $\mathbb R^2$. By Theorem~\ref{thm:CAT(0)diagrem}, $q'(Y)$ is contained in the full subcomplex of $X$ spanned by $q(\mathbb E^2)$. This and property (b) above imply $d_H(\im q', \im q)<M_0$ for $M_0=M_0(L,A,X)$. Also Theorem~\ref{thm:MSquasiflats} (4) follows from Theorem~\ref{thm:CAT(0)diagrem} (3) and the properties of $q'_n$ discussed before.

It remains to show $Y$ is flat outside a compact set. By Theorem~\ref{thm:BKS}, it suffices to estimate the area of balls in $Y$. Pick a base point $y\in Y$ such that $d(q'(y),q(x))<M_0$. Let $B_R$ be the $R$--ball in $Y$ centered at $y$ with respect to the CAT(0) metric and let $K_R$ be the union of faces of $Y$ that intersect $B_R$. Then there exists $n$ and a simplicial embedding $\phi_R\colon K_R\to D'_n$ such that $q'=q'_n\circ \phi_R$ on $K_R$. Thus we also view $y,B_R$ and $K_R$ as subsets of $D'_n$. Since  $B_R\subset B_{D'_n}(y,R)$, it suffices to show $\area(B_{D'_n}(y,R))\le L_3 R^2$ for $L_3$ independent of $R$ and $n$.

Note that there exists a constant $0<\delta<1$ independent of $n$ and $R$ such that $B_{D'_n}(y,\delta n)$ does not touch the boundary of $D'_n$. This uses the fact that $q_n$ and $q'_n$ agree on the boundary, $q_n$ is a quasi-isometry and $q'_n$ is simplicial (hence Lipschitz). We assume $n$ is large enough so that $R<\delta n$. For $r<\delta n$, let $h_r\colon B_{D'_n}(y,\delta n)\to B_{D'_n}(y,r)$ be the map which moves every point $p\in B_{D'_n}(y,\delta n)$ towards $y$ along the geodesic $\overline{py}$ by a factor of $\frac{r}{\delta n}$. Then $h_r$ is $\frac{r}{\delta n}$--Lipschitz and we have
\begin{equation*}
\frac{\area(B_{D'_n}(y,\delta n))}{(\delta n)^2}\ge \frac{\area(B_{D'_n}(y,r))}{r^2}
\end{equation*}
On the other hand, 
$$\area(B_{D'_n}(y,\delta n))\le\area(D'_n)\le L_1(\ell(\partial D'_n))^2=L_1(\ell(\partial D_n))^2\le L_2 n^2.$$
Here $L_1$ is the isoperimetric constant for CAT(0) space and $\ell(\partial D'_n)$ denotes the length of $\partial D'_n$. Thus there exists $L_3$ independent of $n$ and $R$ such that $\area(B_{D'_n}(y,r))\le L_3 r^2$ for any $r<\delta n$. This finishes the proof.
\end{proof}

\section{The complexes for dihedral Artin groups}
\label{sec:geometric model}

In this section we recall the local structure of the complexes for dihedral Artin groups constructed in \cite{Artinmetric} and study disc diagrams over such complexes.
\subsection{The complex for dihedral Artin groups}
\label{subsec:precells}
Let $DA_n$ be the $2$--generator Artin group presented by $\angled{a,b\mid \underbrace{aba\cdots}_{n} =
	\underbrace{bab\cdots}_n}$.

Let $P_n$ be the standard presentation complex for $DA_n$. Namely the $1$--skeleton of $P_n$ is the wedge of two oriented circles, one labeled $a$ and one labeled $b$. Then we attach the boundary of a closed $2$--cell $C$ to the $1$--skeleton with respect to the relator of $DA_n$. Let $C\to P_n$ be the attaching map. Let $\Xa_n$ be the universal cover of $P_n$. Then any lift of the map $C\to P_n$ to $C\to \Xa_n$ is an embedding (cf.\ \cite[Corollary 3.3]{Artinsystolic}). These embedded discs in $\Xa_n$ are called \emph{precells}. Figure~\ref{f:precell} depicts a precell $\Pa$. $\Xa_n$ is a union of copies of $\Pa$'s.
We pull back the labeling and orientation of edges in $P_n$ to obtain labeling and orientation of edges in $\Xa_n$.
\begin{figure}[h!]
	\centering
	\includegraphics[width=1\textwidth]{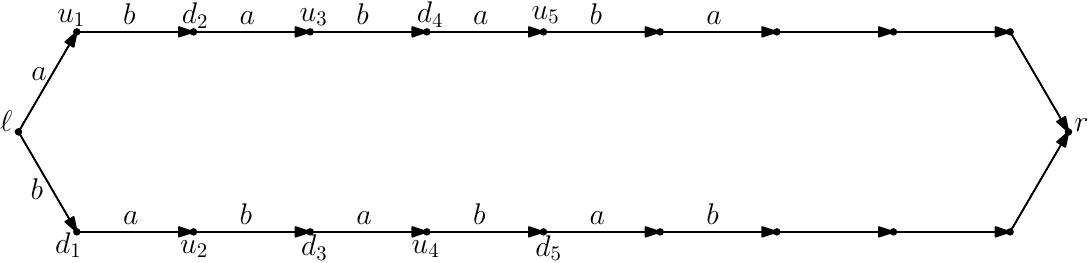}
	\caption{Precell $\Pa$.}
	\label{f:precell}
\end{figure}
We label the vertices of $\Pa$ as in Figure \ref{f:precell}. The vertices $\ell$ and $r$ are called the \emph{left tip} and the \emph{right tip} of $\Pa$. The boundary $\partial\Pa$ is made of two paths. The one starting at $\ell$, going along $\underbrace{aba\cdots}_{n}$ (resp.\ $\underbrace{bab\cdots}_{n}$), and ending at $r$ is called the \emph{upper half} (resp.\ \emph{lower half}) of $\partial\Pa$. The orientation of edges inside one half is consistent, thus each half has an orientation. 

We summarize several facts on how precells intersect each other. 
\begin{lemma}
	\label{cor:connected intersection}
	\cite[Corollary 3.4]{Artinsystolic}
	Let $\Pa_1$ and $\Pa_2$ be two different precells in $\Xa_n$. Then 
	\begin{enumerate}
		\item either $\Pa_1\cap\Pa_2=\emptyset$, or $\Pa_1\cap\Pa_2$ is connected;
		\item if $\Pa_1\cap\Pa_2\neq\emptyset$, $\Pa_1\cap\Pa_2$ is properly contained in the upper half or in the lower half of $\Pa_1$ (and $\Pa_2$);
		\item if $\Pa_1\cap\Pa_2$ contains at least one edge, then one end point of $\Pa_1\cap\Pa_2$ is a tip of $\Pa_1$, and another end point of $\Pa_1\cap\Pa_2$ is a tip of $\Pa_2$, moreover, among these two tips, one is a left tip and one is a right tip.
	\end{enumerate}
\end{lemma}

\begin{lemma}
	\cite[\corfourthree]{Artinmetric}
	\label{lem:unique}
	Let $\Pa_1$ and $\Pa_2$ be two different precells in $\Xa_n$. If $\Pa_1\cap\Pa_2$ contains at least one edge, and $\Pa_3\cap\Pa_2=\Pa_1\cap\Pa_2$, then $\Pa_3=\Pa_1$. 
\end{lemma}

\begin{lemma}
	\label{lem:small intersection}
	Let $\{\Pa_i\}_{i=1}^3$ be three different precells in $\Xa_n$. Suppose
\begin{enumerate}
	\item $\Pa_1\cap\Pa_2$ contains an edge;
	\item $\Pa_1\cap\Pa_3$ contains an edge;
	\item $(\Pa_1\cap\Pa_2)\cap(\Pa_1\cap\Pa_3)$ is either one point or empty.
\end{enumerate}	
Then $\Pa_2\cap\Pa_3$ is either one point or empty.
\end{lemma}

\begin{proof}
By Lemma~\ref{cor:connected intersection}, there are two cases to consider. Either $\Pa_1\cap\Pa_2$ and $\Pa_1\cap\Pa_3$ are in the same half of $\Pa_1$, or they are in different halves. The latter case follows from \cite[Corollary 3.5]{Artinsystolic}. For the former case, we assume without loss of generality that $\Pa_1\cap\Pa_2$ and $\Pa_1\cap\Pa_3$ are contained in the upper half of $\Pa_1$. By Lemma~\ref{cor:connected intersection} (3) and Lemma~\ref{lem:small intersection} (3), we assume $\Pa_1\cap \Pa_2$ (resp.\ $\Pa_1\cap \Pa_3$) contains the left tip (resp.\ right tip) of $\Pa_1$, see Figure~\ref{f:cells}. Assume by contradiction that $\Pa_2\cap\Pa_3$ contains an edge. By \cite[Corollary 3.5]{Artinsystolic}, if $\Pa_2\cap\Pa_3$ and $\Pa_2\cap\Pa_1$ are in different halves of $\Pa_2$, then $\Pa_1\cap\Pa_3$ can not contain any edge, which yields a contradiction. Thus $\Pa_2\cap\Pa_3$ and $\Pa_2\cap\Pa_1$ are in the same half of $\Pa_2$. It follows that $P$ is not an embedded path, where $P$ travels from the left tip of $\Pa_2$ to the left tip of $\Pa_1$, then travel to the right tip of $\Pa_1$ along the upper half of $\Pa_1$, then travel to the right tip of $\Pa_3$. On the other hand, since $P$ represents a word in the positive Artin monoid, $P$ is an embedded path by the injectivity of positive Artin monoid \cite{deligne,brieskorn1972artin}, which leads to a contradiction.
\begin{figure}[ht!]
	\centering
	\includegraphics[width=0.6\textwidth]{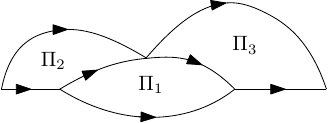}
	\caption{}
	\label{f:cells}
\end{figure}
\end{proof}

We subdivide each precell $\Pa$ in $\Xa_n$ into a simplicial complex by placing a vertex in the middle of the precell, and adding edges which connect this new vertex and vertices in the boundary $\partial\Pa$. This subdivision turns $\Xa_n$ into a simplicial complex $\Xb_n$. A \emph{cell} of $\Xb_n$ is defined to be a subdivided precell, and we use the symbol $\Pi$ for a cell. 
The original vertices of $\Xa_n$ in $\Xb_n$ are called the \emph{real vertices}, and the new vertices of $\Xb_n$ after subdivision are called \emph{fake vertices}. The fake vertex in a cell $\Pi$ is denoted $o$.

For each pair of cells $(\Pi_1,\Pi_2)$ in $\Xb_n$ such that $\Pi_1\cap\Pi_2$ contains at least two edges, we add an edge between the fake vertex of $\Pi_1$ and the fake vertex of $\Pi_2$. Let $X_n$ be the flag completion of the complex obtained by adding these edges to $\Xb_n$. There is a simplicial action $DA_n\act X_n$.

\begin{definition}
	\label{def:length}
	We assign lengths to edges of $X_n$. Edges connecting a real vertex to a fake vertex have length 1. Edges between two real vertices have length equal to the distance between two adjacent vertices in a Euclidean regular $(2n)$--gon with radius 1. 
	
	Now we assign lengths to edges between two fake vertices. First define a function $\phi\colon[0,\pi)\to \mathbb R$ as follows. Let $\triangle(ABC)$ be a Euclidean isosceles triangle with length of $AB$ and $AC$ equal to 1, and $\angle_A(B,C)=\alpha$. Then $\phi(\alpha)$ is defined to be the length of $BC$. Suppose $\Pi_1$ and $\Pi_2$ are cells such that $\Pi_1\cap\Pi_2$ contains $i$ edges ($i\ge 2$). Then the edge between the fake vertex of $\Pi_1$ and the fake vertex of $\Pi_2$ has length $=\phi(\frac{n-i}{2n}2\pi)$.
\end{definition}

\begin{remark}[Intuitive explanation of the construction of $X_n$]
A natural way to metrize $\Xa_n$ to declare each $2$--cell in $\Xa_n$ is a regular polygon in the Euclidean plane. However, if we take $\Pi_1$ and $\Pi_2$ (say, two $n$--gons) such that $P=\Pi_1\cap\Pi_2$ has $\ge 2$ edges, then any interior vertex of $P$ is not non-positively curved. Let $o_i$ be the fake vertex in $\Pi_i$ and let the two endpoints of $P$ be $v_1$ and $v_2$. Let $K$ be the region in $\Pi_1\cup \Pi_2$ bounded by the $4$--gon whose vertices are $o_1$, $o_2$, $v_1$ and $v_2$. Those positively curved cone points are contained in $K$. Now we replace $K$ by something flat as follows. Add a new edge $e$ between $o_1$ and $o_2$ and add two new triangles $\{\Delta_i\}_{i=1}^2$ such that the three sides of $\Delta_i$ are $e$, $\overline{o_1v_i}$ and $\overline{o_2v_i}$. We replace $K$ by $\Delta_1\cup\Delta_2$, which is flat. Moreover, we would like $\angle_{o_1}(v_1,o_2)=\angle_{o_1}(o_2,v_2)=\frac{|P|\pi}{2n}$ so that $o_1$ is still flat ($|P|$ is the number of edges in $P$). That is why we assign the length of $e$ as in Definition~\ref{def:length}.
\end{remark}

By \cite[\lemfourseven]{Artinmetric}, the lengths of the three sides of each triangle in $X^{(1)}_n$ satisfy the strict triangle inequality. Thus we can treat $X^{(2)}_n$ as a piecewise Euclidean complex such that each $2$--simplex is a Euclidean triangle whose lengths of sides coincide with the assigned lengths on $X^{(1)}_n$. 

\subsection{Local structure of $X_n$}
\label{subsec:local}
For a vertex $v\in X_n$, define $\Lambda_v=\lk(v,X^{(2)}_n)$. In this subsection we study the structure of $\Lambda_v$. 

Pick an identification between real vertices of $X_n$ and elements of $DA_n$ via the action $DA_n\act X_n$. Let $\Pi$ be a base cell in $X_n$ with its vertices and edges labeled as in Figure~\ref{f:precell}. We assume $u_0=d_0=\ell$, $u_n=d_n=r$, and $\ell$ is identified with the identity element of $DA_n$.

We first look at the case when $v$ is a real vertex. Up to action of $DA_n$, we assume $v=\ell$. Vertices of $\Lambda_v$ consist of two classes. (1) Real vertices $a^i,a^o,b^i$ and $b^o$, where $a^i$ and $a^o$ are the vertices in $\Lambda_v$ which corresponds to the incoming and outgoing $a$--edge containing $v$ ($b^i$ and $b^o$ are defined similarly). (2) Fake vertices. There is a 1-1 correspondence between such vertices and cells in $X_n$ containing $\ell$. Then the fake vertices of $\Lambda_v$ are of the form $w^{-1}o$ where $w$ is a vertex of $\partial\Pi$ and $o$ is the fake vertex in $\Pi$ ($w^{-1}o$ means the image of $o$ under the action of $w^{-1}$), that is $\{\ell^{-1} o,r^{-1}o,d^{-1}_1o,d^{-1}_2o,\ldots,d^{-1}_{n-1}o,u^{-1}_1o,u^{-1}_2o,\ldots,u^{-1}_{n-1}o\}$.

Edges of $\Lambda_v$ can be divided into two classes. Edges of \emph{type I} in $\Lambda_v$ are edges between a real vertex and a fake vertex. These edges are drawn in Figure~\ref{f:real} (we use the convention that the real vertices are drawn as solid points and the fake vertices as circles). By \cite[\lemfivetwo]{Artinmetric}, each edge of type I has angular length $=\frac{n-1}{4n}2\pi$. 

\begin{figure}[ht!]
	\centering
	\includegraphics[width=0.88\textwidth]{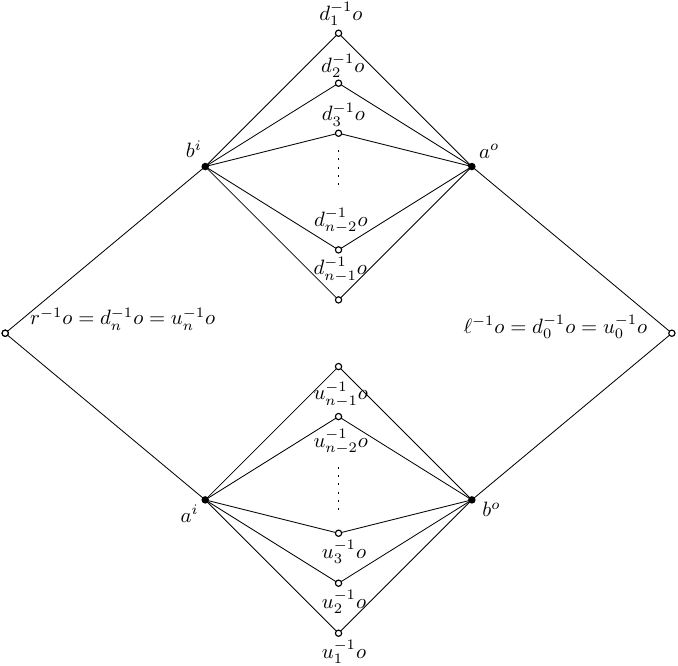}
	\caption{Edges of type I in the link of a real vertex.}
	\label{f:real}
\end{figure}

Edges of \emph{type II} in $\Lambda_v$ are edges between two fake vertices. They are characterized by Lemma~\ref{lem:edge type II}. There do not exist edges of $\Lambda_v$ which are between two real vertices. We write $t\sim s$ (resp.\ $t\nsim s$) if vertices $t$ and $s$ are connected (resp.\ are not connected) by an edge. 
\begin{lemma}\
	\label{lem:edge type II}
\begin{enumerate}
\item $d^{-1}_io\sim d^{-1}_jo$ if and only if $1\le |j-i|\le n-2$, in this case, the edge between $d^{-1}_io$ and $d^{-1}_jo$ has angular length $=\frac{j-i}{2n}2\pi$. A similar statement holds with $d$ replaced by $u$.
\item If $1\le i\le n-1$ and $1\le j\le n-1$, then $d^{-1}_io\nsim u^{-1}_jo$.
\end{enumerate}
\end{lemma}

This lemma is deduced from the fact that $d^{-1}_io\sim d^{-1}_jo$ if and only if $d^{-1}_i\Pi\cap d^{-1}_j\Pi$ has $\ge 2$ edges. We refer to 	\cite[\lemfivethree]{Artinmetric} for a proof. 

Now we study cycles in $\Lambda_v$. Let $\Lambda^+_v$ be the full subgraph of $\Lambda_v$ spanned by $\{b^i,a^o,\di{0}o,\di{1}o,\ldots,\di{n}o\}$. Let $\Lambda^-_v$ be the full subgraph of $\Lambda_v$ spanned by $\{b^o,a^i,\ui{0}o,\ui{1}o,\ldots,\ui{n}o\}$. 

\begin{lemma}	\cite[\lemfiveeight]{Artinmetric}
	\label{lem:cycle real}
Suppose $\omega\subset \Lambda_v$ is a simple cycle. Then at least one of the following two situations happen:
\begin{enumerate}
	\item $\omega\subset\Lambda^+_v$ or $\omega\subset\Lambda^-_v$;
	\item $\ell^{-1}o\in\omega$ and $r^{-1}o\in\omega$.
\end{enumerate} 
\end{lemma}

Thus $\{\ell^{-1}o,r^{-1}o\}$ is called the \emph{necks} of $\Lambda_v$. The next two lemmas give more detailed description of situations (1) and (2) of Lemma~\ref{lem:cycle real} respectively.

\begin{lemma}	
	\label{lem:not two-full}
Suppose $\omega$ is a simple cycle in $\Lambda^+_v$ or $\Lambda^-_v$. Then $\omega$ is not $2$--full.
\end{lemma}
This follows from \cite[\lemfivesix]{Artinmetric} and \cite[\lemfiveseven]{Artinmetric}.

\begin{lemma}	\cite[\lemfiveeleven]{Artinmetric}
	\label{lem:exactly pi}
	Suppose $v$ is real and $\omega$ is an edge path from $r^{-1}o$ to $\ell^{-1}o$. Then $\omega$ has angular length $\ge \pi$. 
	
	If $\omega$ has angular length $=\pi$, then either $\omega\subset\Lambda^+_v$ or $\omega\subset\Lambda^-_v$, and the following are the only possibilities of $\omega$ when $\omega\subset\Lambda^+_v$:
	\begin{enumerate}
		\item $\omega=r^{-1}o\to b^i\to d^{-1}_{1}o\to d^{-1}_{0}o$;
		\item $\omega=d^{-1}_{n}o\to d^{-1}_{n-1}o\to a^o\to \ell^{-1}o$;
		\item $\omega=d^{-1}_{i_1}o\to d^{-1}_{i_2}o\to\cdots\to d^{-1}_{i_k}o$ where $n=i_1>i_2>\cdots>i_k=0$.
	\end{enumerate}
	A similar statement holds for $\omega\subset\Lambda^-_v$.
\end{lemma}

Now we turn to the case when $v$ is a fake vertex. Up to the action of $DA_n$, we assume $v=o$. Vertices of $\Lambda_v$ consists of two classes. (1) Real vertices: these are the vertices in $\partial\Pi$. (2) Fake vertices: these are the fake vertices of cells $\Pi'$ such that $\Pi'\cap\Pi$ contain at least two edges.

By Lemma~\ref{cor:connected intersection}, $\Pi'\cap \Pi$ is a connected path such that it contains exactly one of the tips of $\Pi$ and it is properly contained in a half of $\partial\Pi$. If $\Pi'\cap \Pi$ is a path in the upper half (resp.\ lower half) of $\partial\Pi$ that has $m$ edges and contains the left tip, then we denote the fake vertex in $\Pi'$ by $L_m$ (resp.\ $L'_m$), and write $\Pi'=\Pi_{L_m}$ (resp.\ $\Pi'=\Pi_{L'_m}$). Note that such $\Pi'$ is unique by Lemma~\ref{lem:unique}. Similarly, we define $R_m$, $R'_m$, $\Pi_{R_m}$ and $\Pi_{R'_m}$ respectively; see Figure~\ref{f:fake}. 
\begin{figure}[ht!]
	\centering
	\includegraphics[width=1\textwidth]{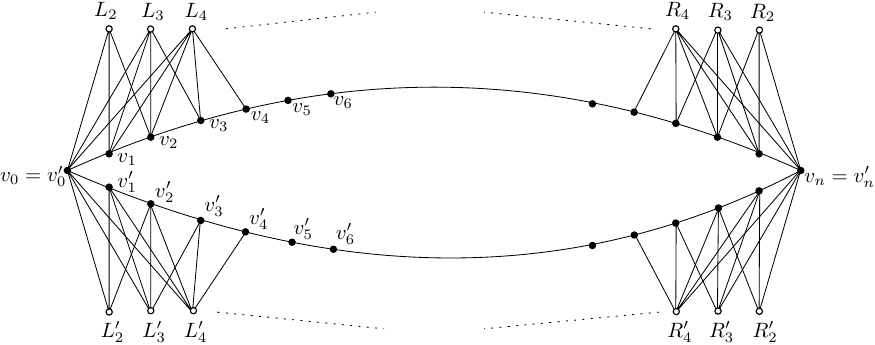}
	\caption{The link of a fake vertex.}
	\label{f:fake}
\end{figure}

The structure of $\Lambda_v$ will be easier to describe if we label vertices of $\partial\Pi$ differently as follows. The vertices in the upper half (resp.\ lower half) of $\partial\Pi$ are called $v_0,v_1,\ldots,v_n$ (resp.\ $v'_0,v'_1,\ldots,v'_n$) from left to right. Note that $v_0=v'_0$ and $v_n=v'_n$.

Edges of $\Lambda_v$ consist of three classes:
\begin{enumerate}
	\item \emph{Edges of type I}. They are edges between real vertices of $\Lambda_v$. Hence they are edges in $\partial\Pi$. Each of them has angular length $=\frac{1}{2n}2\pi$.
	\item \emph{Edges of type II}. They are edges between a real vertex and a fake vertex, and they are characterized by Lemma~\ref{lem:real fake} below.
	\item \emph{Edges of type III}. They are edges between fake vertices of $\Lambda_v$. We will not need information about them in this paper. The interested reader can find a description of them in \cite[\lemfivefourteen]{Artinmetric}.
\end{enumerate}
We refer to Figure~\ref{f:fake} for a picture of $\Lambda_v$. Edges of type I and some edges of type II are drawn. Edges of type III are not drawn in the picture.

\begin{lemma}\
	\label{lem:real fake}
	\begin{enumerate}
		\item The collection of vertices in $\partial\Pi$ adjacent to $L_i$ (resp.\ $L'_i$) is $\{v_0,v_1,\ldots,\\v_i\}$ (resp.\ $\{v'_0,v'_1,\ldots,v'_i\}$). 
		\item The collection of vertices in $\partial\Pi$ adjacent to $R_i$ (resp.\ $R'_i$) is $\{v_n,v_{n-1},\\ \ldots,v_{n-i}\}$ (resp.\ $\{v'_n,v'_{n-1},\ldots,v'_{n-i}\}$). 
		\item The angular length of any edge between $L_i$ and a real vertex of $\Lambda_v$ is $\frac{i}{4n}2\pi$. The same holds with $L_i$ replaced by $L'_i,R_i$ and $R'_i$.
	\end{enumerate}
\end{lemma}

Here (1) and (2) follow from the definition of $L_i$ and $R_i$. (3) follows from Definition~\ref{def:length}.

Let $\Lambda^+_v$ be the full subgraph of $\Lambda_v$ spanned by $$\{v_0,v_1,\ldots,v_n\}\cup \{L_2,L_3,\ldots,L_{n-1}\}\cup \{R_2,R_3,\ldots,R_{n-1}\}.$$ Let $\Lambda^-_v$ be the full subgraph of $\Lambda_v$ spanned by $$\{v'_0,v'_1,\ldots,v'_n\}\cup \{L'_2,L'_3,\ldots,L'_{n-1}\}\cup \{R'_2,R'_3,\ldots,R'_{n-1}\}.$$

The following is proved in \cite[\lemfivesixteen]{Artinmetric} and \cite[\lemfiveseventeen]{Artinmetric}.
It says that
Lemma~\ref{lem:cycle real} and Lemma~\ref{lem:not two-full} continue to hold also in the case of $v$ being fake ($r^{-1}o$ and $\ell^{-1}o$ in the statement of Lemma~\ref{lem:cycle real} should be replaced by $v_0$ and $v_n$).
\begin{lemma}
	\label{lem:not two-full1}
	Suppose $\omega\subset \Lambda_v$ is a simple cycle. Then at least one of the following two situations happen:
	\begin{enumerate}
		\item $\omega\subset\Lambda^+_v$ or $\omega\subset\Lambda^-_v$;
		\item $v_0,v_n\in\omega$.
	\end{enumerate} 
	If $\omega$ is a simple cycle in $\Lambda^+_v$ or $\Lambda^-_v$ then $\omega$ is not $2$--full.		
\end{lemma}

We call $\{v_0,v_n\}$ the \emph{necks} of $\Lambda_v$.
The following is parallel to Lemma~\ref{lem:exactly pi}.

\begin{lemma} 	\cite[\lemfivenineteen]{Artinmetric}
	\label{lem:exactly pi1}
Suppose $v$ is fake and $\omega$ is an edge path in $\Lambda_v$ from $v_0$ to $v_n$. Then $\omega$ has angular length $\ge\pi$.

If $\omega$ has angular length $=\pi$, then either $\omega\subset\Lambda^+_v$ or $\omega\subset\Lambda^-_v$, and the following are the only possibilities of $\omega$ when $\omega\subset\Lambda^+_v$:
	\begin{enumerate}
		\item $\omega$ does not contain fake vertices, i.e.\ $\omega=v_0\to v_1\to\cdots\to v_n$;
		\item $\omega=v_0\to v_1\to\cdots\to v_{n-i_1}\to R_{i_1}\to\cdots \to R_{i_m}\to v_n$, where $i_1>\cdots>i_m\ge2$;
		\item $\omega=v_0\to L_{i_1}\to\cdots\to L_{i_m}\to v_{i_m}\to v_{i_m+1}\to\cdots \to v_n$, where $2\le i_1<\cdots< i_m$;
		\item $\omega=v_0\to L_{i_1}\to\cdots\to L_{i_m}\to v_{i_m}\to v_{i_{m+1}}\to\cdots \to v_{n-i'_1}\to R_{i'_1}\to\cdots \to R_{i'_{m'}}\to v_n$, where $2\le i_1<\cdots<i_m$, $i'_{1}>\cdots>i'_{m'}\ge 2$, and $i_m\le n-i'_1$.
	\end{enumerate}
	A similar statement holds when $\omega\subset\Lambda^-_v$.
\end{lemma}

\begin{remark}
	\label{rmk:non tip}
Note that in all the cases (1)--(4) of Lemma~\ref{lem:exactly pi1}, for any two non-neck vertices $u_1,u_2\in\omega$ with $u_1$ real and $u_2$ fake, either $u_1$ is a tip of the cell $\Pi_{u_2}$ that contains $u_2$, or $u_1\notin \Pi_{u_2}$. This statement can be deduced from Lemma~\ref{lem:real fake} and Lemma~\ref{cor:connected intersection}.
\end{remark}

It follows from Lemma~\ref{lem:cycle real}, Lemma~\ref{lem:not two-full}, Lemma~\ref{lem:exactly pi}, Lemma~\ref{lem:not two-full1} and Lemma~\ref{lem:exactly pi1} that $X_n$ is locally $2\pi$--large (cf.\ Definition~\ref{d:metric_syst}). Moreover, we showed in \cite[\lemsixthree]{Artinmetric} that $X_n$ is simply-connected. 

\begin{theorem}
	\label{thm:X_n}
$X_n$ is metrically systolic.
\end{theorem}
\subsection{Flat vertices in disc diagrams over $X_n$}
In this subsection, we explore links of certain flat vertices in reduced disc diagrams. 
\begin{definition}
	A flat interior vertex in a disc diagram whose all neighbours are flat interior vertices is called
	a \emph{deep flat vertex}.
\end{definition}

\begin{lemma}
	\label{lem:flat vertex}
	Suppose $q'\colon Y\to X_n$ satisfies all the conditions in Theorem~\ref{thm:MSquasiflats}. Let $y\in Y$ be a flat interior vertex and let $v=q'(y)$. Then $q'(\lk(y,Y))$ is a simple cycle in $\Lambda_v$ that is made of two paths of angular length $=\pi$ connecting the two necks of $\Lambda_v$, one in $\Lambda^+_v$ and one in $\Lambda^-_v$.
\end{lemma}

\begin{proof}
Since $q'$ is injective on $\St(y,Y)$, to simplify notation, we identify $\lk(y,Y)$ and $q'(\lk(y,Y))$. By Lemma~\ref{lem:cycle real}, Lemma~\ref{lem:exactly pi} and Lemma~\ref{lem:exactly pi1}, it suffices to rule out the cases $\lk(y,Y)\subset\Lambda^+_v$ and $\lk(y,Y)\subset\Lambda^-_v$. To rule out the former case, note that $\lk(y,Y)$ is a simple cycle in $\Lambda^+_v$, then by Lemma~\ref{lem:not two-full} (or Lemma~\ref{lem:not two-full1}), we can add a local diagonal to $\lk(y,Y)$ to cut it into a triangle and a cycle $\sigma$ with smaller number of edges. Now we apply same procedure to $\sigma$ (note that $\sigma\subset \Lambda^+_v$ since $\Lambda^+_v$ is a full subgraph of $\Lambda_v$ by definition). By repeating this process finitely many times, we know that $\lk(y,Y)$ bounds a disc diagram without interior vertices, which contradicts Theorem~\ref{thm:MSquasiflats} (4). Similarly, we rule out the case $\lk(y,Y)\subset\Lambda^-_v$.
\end{proof}

\begin{lemma}
	\label{lem:local}
Suppose $q'\colon   Y\to X_n$ satisfies all the conditions in Theorem~\ref{thm:MSquasiflats}. Let $v\in Y$ be a deep flat vertex. Then there do not exist two adjacent non-neck fake vertices in $q'(\lk(v,Y))$.
\end{lemma}

To simply notation, denote $q'(v)$ by $v$ and identify $\lk(v,Y)$ and $q'(\lk(v,Y))$.
\begin{proof}
Arguing by contradiction we assume there are two adjacent non-neck fake vertices $o_1,o_2$ in $\lk(v,Y)$. We assume that $v=o$ when $v$ is fake (recall that $o$ is the fake vertex in the base cell $\Pi$), and $v=\ell$ when $v$ is real. For $i=1,2$, let $\Pi_i$ be the cell containing $o_i$.

\emph{Case 1}: $v=\ell$. By Lemma~\ref{lem:flat vertex}, $\lk(v,Y)$ consists of two paths $\sigma_1,\sigma_2$ of angular length $=\pi$ connecting the two necks of $\Lambda_v$. By Lemma~\ref{lem:exactly pi}, we can assume that $o_1,o_2\in\sigma_1\subset\Lambda^+_v$. Then $\sigma_1$ must satisfy Lemma~\ref{lem:exactly pi} (3), moreover, we can assume $\di{i_2}o=o_1$ and $\di{i_3}o=o_2$. See Figure~\ref{f:flat vertex}.
\begin{figure}[h!]
	\centering
	\includegraphics[scale=1.1]{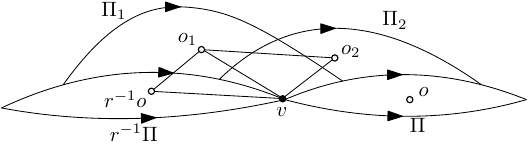}
	\caption{}
	\label{f:flat vertex}
\end{figure}

Now we consider the path $P=r^{-1}o\to v\to o_2$ in $\lk(o_1,Y)$. Since $v$ is not a tip of $\Pi_1$, $v$ is not a neck of $\Lambda_{o_1}$. By Lemma~\ref{lem:flat vertex}, there is a path $\sigma$ of angular length $=\pi$ connecting the two necks of $\Lambda_{o_1}$ such that $P\subset\sigma\subset \lk(o_1,Y)$. Since $P$ has a real vertex between two fake vertices, $\sigma$ satisfies Lemma~\ref{lem:exactly pi1} (4). It follows from Lemma~\ref{cor:connected intersection} that $r^{-1}\Pi\cap\Pi_1\cap\Pi_2$ contains at least one edge. However, by Lemma~\ref{lem:real fake}, in case (4) of Lemma~\ref{lem:exactly pi1}, $\Pi_{L_{i_m}}\cap\Pi\cap\Pi_{R_{i'_1}}$ is either empty (when $i_m<n-i'_1$) or one point (when $i_m=n-i'_1$). This leads to a contradiction.

\emph{Case 2:} $v=o$. By Lemma~\ref{lem:flat vertex}, $o_1$ and $o_2$ are consecutive non-neck fake vertices in a path of angular length $\pi$ between the two necks of $\Lambda_o$. By Lemma~\ref{lem:exactly pi1}, we can assume either $o_1=L_{i_1}$, $o_2=L_{i_2}$, or $o_1=R_{i'_{m'-1}}$, $o_2=R_{i'_{m'}}$. Now we study the former case. See Figure~\ref{f:flat vertex1}.
\begin{figure}[h!]
	\centering
	\includegraphics[scale=0.9]{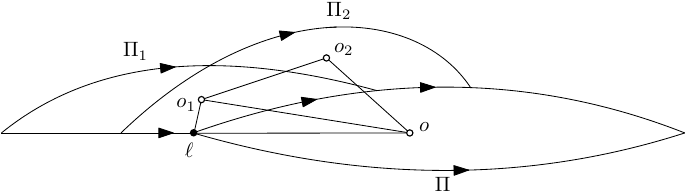}
	\caption{}
	\label{f:flat vertex1}
\end{figure}
Consider the path $P=\ell\to o\to o_2$ in $\lk(o_1,Y)$. Since $\ell$ is not a tip of $\Pi_1$, $\ell$ is not a neck of $\Lambda_{o_1}$. It follows that $P$ is contained in a path $\sigma\subset \Lambda_{o_1}$ of angular length $=\pi$ connecting the two necks of $\Lambda_{o_1}$.  Note that $\ell\in\Pi_2$, but $\ell$ is not a tip of $\Pi_2$ (by applying Lemma~\ref{cor:connected intersection} (3) to $\Pi_2\cap\Pi$). This contradicts Remark~\ref{rmk:non tip}. The case $o_1=R_{i'_{m'-1}}$ and $o_2=R_{i'_{m'}}$ can be treated similarly.
\end{proof}

\section{Local structure of quasi-Euclidean diagrams}
\label{sec:local flatness}
In this section we recall the construction of the metrically systolic complexes for two-dimensional Artin groups, and use these complexes as a tool to study the local structure of quasi-Euclidean diagrams over the universal covers of standard presentation complexes of two-dimensional Artin groups.

\subsection{The complex for $2$--dimensional Artin groups}
\label{s:general}
Let $A_\Gamma$ be an Artin group with defining graph $\Gamma$. Let $\Gamma'\subset\Gamma$ be a full subgraph with induced edge labeling and let $A_{\Gamma'}$ be the Artin group with defining graph $\Gamma'$. Then there is a natural homomorphism $A_{\Gamma'}\to A_{\Gamma}$. By \cite{lek}, this homomorphism is injective. Subgroups of $A_{\Gamma}$ of the form $A_{\Gamma'}$ are called \emph{standard subgroups}. Let $P_{\Gamma}$ be the standard presentation complex of $A_{\Gamma}$, and let $X^{\ast}_{\Gamma}$ be the universal cover of $P_{\Gamma}$. We orient each edge in $P_{\Gamma}$ and label each edge in $P_{\Gamma}$ by a generator of $A_\Gamma$. Thus edges of $X^{\ast}_{\Gamma}$ have induced orientation and labeling. There is a natural embedding $P_{\Gamma'}\hookrightarrow P_\Gamma$. Since $A_{\Gamma'}\to A_{\Gamma}$ is injective, $P_{\Gamma'}\hookrightarrow P_\Gamma$ lifts to various embeddings $X^{\ast}_{\Gamma'}\to X^{\ast}_{\Gamma}$. Subcomplexes of $X^{\ast}_{\Gamma}$ arising in such way are called \emph{standard subcomplexes}.

A \emph{block} of $X^{\ast}_{\Gamma}$ is a standard subcomplex which comes from an edge in $\Gamma$. This edge is called the \emph{defining edge} of the block. Two blocks with the same defining edge are either disjoint, or identical. A block is \emph{large} if the label of the corresponding edge is at least $3$.

We define precells of $X^{\ast}_{\Gamma}$ as in Section \ref{subsec:precells}. Each precell is embedded. We subdivide each precell as in Section~\ref{subsec:precells} to obtain a simplicial complex $\Xb_\Gamma$. Fake vertices and real vertices of $\Xb_{\Gamma}$ are defined in a similar way.

Within each block of $\Xb_\Gamma$, we add edges between fake vertices as in Section~\ref{subsec:precells}. Then we take the flag completion to obtain $X_\Gamma$. The action $A_{\Gamma}\curvearrowright\Xb_\Gamma$ extends to a simplicial action $A_{\Gamma}\curvearrowright X_{\Gamma}$, which is proper and cocompact. A \emph{block} in $X_{\Gamma}$ is defined to be the full subcomplex spanned by vertices in a block of $\Xb_\Gamma$. Intersection of two different blocks of $X_\Gamma$ does not contain fake vertices. Pick a block $B^{\Delta}\subset \Xb_\Gamma$, and let $B\subset X_\Gamma$ be the block containing $B^{\Delta}$. Let $n$ be the label of the defining edge of $B^{\Delta}$. Then by \cite[\lemsixtwo]{Artinmetric}, the natural isomorphism $B^{\Delta}\to \Xb_n$ extends to an isomorphism $B\to X_n$. 


Next we assign lengths to edges of $X_\Gamma$. Let $B\subset X_\Gamma$ be a block and let $B\to X_n$ be the isomorphism in the previous paragraph. We first rescale the edge lengths of $X_n$ defined in Section~\ref{subsec:precells} by a uniform factor such that any edge between two real vertices has length $=1$. Then we pull back these edge lengths to $B$ by the above isomorphism. We repeat this process for each block of $X_\Gamma$. Note that if an edge of $X_\Gamma$ belongs to two different blocks, then this edge is between two real vertices, hence it has a well-defined length (which is $1$). The action of $A_\Gamma$ preserves edge lengths.

\begin{theorem}\cite[\lemsixseven]{Artinmetric}
	\label{thm:msystolic}
If $A_\Gamma$ has dimension $\le 2$, then $X_\Gamma$ with its piecewise Euclidean structure is metrically systolic.
\end{theorem}

From now on we will assume $A_\Gamma$ has dimension $\le 2$.

Now we consider local structure of $X_\Gamma$. If $v\in X_\Gamma$ is a fake vertex, then there is a unique block $B\ni v$. Moreover, $\lk(v,X^{(2)}_\Gamma)=\lk(v,B^{(2)})$ by our construction, which reduces to discussion in Section~\ref{subsec:local}. Links of real vertices are more complicated since they can travel through several blocks. However, cycles of angular length $\le 2\pi$ in the link have a relatively simple characterization as follows.

\begin{lemma}\cite[\lemsixsix]{Artinmetric}
	\label{lem:2pi cycle}
	Let $v\in X_\Gamma$ be a real vertex and let $\omega$ be a simple cycle in $\lk(v,X^{(2)}_\Gamma)$ with angular length $\le 2\pi$ in the link of $v$. Then exactly one of the following four situations happens:
	\begin{enumerate}
		\item $\omega$ is contained in one block;
		\item $\omega$ travels through two different blocks $B_1$ and $B_2$ such that their defining edges intersect in a vertex $a$, and $\omega$ has angular length $=\pi$ inside each block, moreover, there are exactly two vertices in $\omega\cap B_1\cap B_2$ and they corresponding to an incoming $a$--edge and an outgoing $a$--edge based at $v$;
		\item $\omega$ travels through three blocks $B_1,B_2,B_3$ such that the defining edges of these blocks form a triangle $\triangle(abc)\subset \Gamma$ and $\frac{1}{n_1}+\frac{1}{n_2}+\frac{1}{n_3}=1$ where $n_1$, $n_2$ and $n_3$ are labels of the edges of this triangle, moreover, $\omega$ is a $6$--cycle with its vertices alternating between real and fake such that the three real vertices in $\omega$ correspond to an $a$--edge, a $b$--edge and a $c$--edge based at $v$;
		\item $\omega$ travels through four blocks such that the defining edges of these blocks form a full $4$--cycle in $\Gamma$, moreover, $\omega$ is a $4$--cycle with one edge of angular length $\pi/2$ in each block.
	\end{enumerate}
	Note that in cases (2), (3) and (4), $\omega$ actually has angular length $=2\pi$.
\end{lemma}
\subsection{Classification of flat points in diagrams over $X_\Gamma$}
\label{subsec:local structure}

Let $q'\colon Q'\to X_\Gamma$ be a map which satisfies all the requirements of Theorem~\ref{thm:MSquasiflats}. We study the local properties of $q'$ and $Q'$ in this subsection.

We define a partial retraction $\rho$ from $X^{(2)}_\Gamma$ (defined on its subset) to $\Xb_\Gamma$ as follows. The map is the identity on the $2$--skeleton of $\Xb_\Gamma$. Let $\overline{o_1o_2}$ be an edge of $X^{(2)}_\Gamma$ not in $\Xb_\Gamma$, where $o_i$ is the fake vertex in some cell $\Pi_i$ for $i=1,2$. Let $m$ be the middle point of $\partial\Pi_1\cap\partial\Pi_2$ and let $x_1,x_2$ be the two endpoints of $\partial\Pi_1\cap\partial\Pi_2$. We map $\overline{o_1o_2}$ homeomorphically to the concatenation of $\overline{o_1m}$ and $\overline{mo_2}$, and map the triangle $\triangle(o_1o_2x_1)$ (resp.\ $\triangle(o_1o_2x_2)$) homeomorphically to the region in $\Pi_1\cup\Pi_2$ bounded by $\overline{o_1x_1}$, $\overline{x_1o_2}$, $\overline{o_2m}$ and $\overline{mo_1}$ (resp.\ $\overline{o_1x_2}$, $\overline{x_2o_2}$, $\overline{o_2m}$ and $\overline{mo_1}$). This finishes the definition of $\rho$. Note that $\rho$ is not defined on the whole space $X^{(2)}_\Gamma$. That is why we call it a partial retraction. More precisely, for a triangle $\Delta\subset X^{(2)}_\Gamma$, $\rho(\Delta)$ is defined if and only if $\Delta$ satisfies one of the following two cases:
\begin{enumerate}
	\item $\Delta$ has two real vertices and one fake vertices, in this case $\rho(\Delta)=\Delta$;
	\item $\Delta$ has one real vertex $v$ and two interior vertices $o_1$ and $o_2$, moreover, $v$ is a tip of $\Pi_1$ or $\Pi_2$, where $\Pi_i$ is the cell containing $o_i$ for $i=1,2$.
\end{enumerate}

\begin{lemma}
	\label{lem:well-defined}
Let $q=\rho\circ q'$. If $\Delta\subset Q'$ is a triangle such that all of its vertices are flat and deep, then $q(\Delta)$ is defined. 
\end{lemma}

\begin{proof}
If all vertices of $q'(\Delta)$ are fake, then for any vertex $x\in\Delta$, the cycle $q'(\lk(x,Q'))$ in $\lk(q'(x),X^{(2)}_\Gamma)$ contains two consecutive fake vertices, hence contradicts Lemma~\ref{lem:local}. If $q'(\Delta)$ has one real vertex $u$ and two interior vertices $o_1$ and $o_2$, but $u$ is neither the tip of $\Pi_1$ nor the tip of $\Pi_2$ ($\Pi_i$ is the cell containing $o_i$), then $u\in \Pi_1$ and $u$ is not a neck of $\lk(o_2,X^{(2)}_\Gamma)$, which contradicts Remark~\ref{rmk:non tip} (let the non-neck real vertex and the fake vertex in Remark~\ref{rmk:non tip} be $u$ and $o_1$ respectively).
\end{proof}

In what follows we use the map $q$ as in Lemma~\ref{lem:well-defined} (wherever it is well-defined).
A vertex $x\in Q'$ is \emph{real} or \emph{fake} if $q'(x)=q(x)$ is, respectively, real or fake. 

\begin{lemma}
	\label{lem:fake vertex}
Suppose $x\in Q'$ is flat, deep and fake, then the restriction of $q$ to $\St(x,Q')$ is an embedding. Moreover, $q(\St(x,Q'))$ contains the cell of $\Xb_\Gamma$ containing $q(x)$.
\end{lemma}

\begin{proof}
By Theorem~\ref{thm:MSquasiflats}, the restriction of $q'$ to the closed star $\St(x,Q')$ is an embedding. Let $v=q'(x)$ and let $\omega=q'(\lk(x,Q'))$. Then
\begin{enumerate}
	\item $\omega$ is a concatenation of two paths $\omega^+$ and $\omega^-$ such that each of them connects the two necks of $\Lambda_v=\lk(v,X^{(2)}_\Gamma)$, and $\omega^+\subset\Lambda^+_v$, $\omega^-\subset\Lambda^-_v$;
	\item $\omega^+$ and $\omega^-$ do not contain consecutive fake vertices.
\end{enumerate}
(1) follows from Lemma~\ref{lem:flat vertex}, (2) follows from Lemma~\ref{lem:local}. By (2) and Lemma~\ref{lem:exactly pi1}, there are only four possibilities for $\omega^+$:
\begin{enumerate}
		\item $\omega^+$ does not contain fake vertices, i.e.\ $\omega=v_0\to v_1\to\cdots\to v_n$;
		\item $\omega^+=v_0\to v_1\to\cdots\to v_{n-i}\to R_i\to v_n$, where $i\ge2$;
		\item $\omega^+=v_0\to L_i\to v_i\to v_{i+1}\to\cdots \to v_n$, where $2\le i$;
		\item $\omega^+=v_0\to L_i\to v_i\to v_{i+1}\to\cdots \to v_{n-j}\to R_j\to v_n$, where $2\le i$, $j\ge 2$, and $i\le n-j$.
\end{enumerate}
A similar statement holds for $\omega^-$. Then the lemma follows from Lemma~\ref{lem:real fake} and the definition of $\rho$ (see Figure~\ref{f:retraction} for an example).
\end{proof}

\begin{figure}[h!]
	\centering
	\includegraphics[width=1\textwidth]{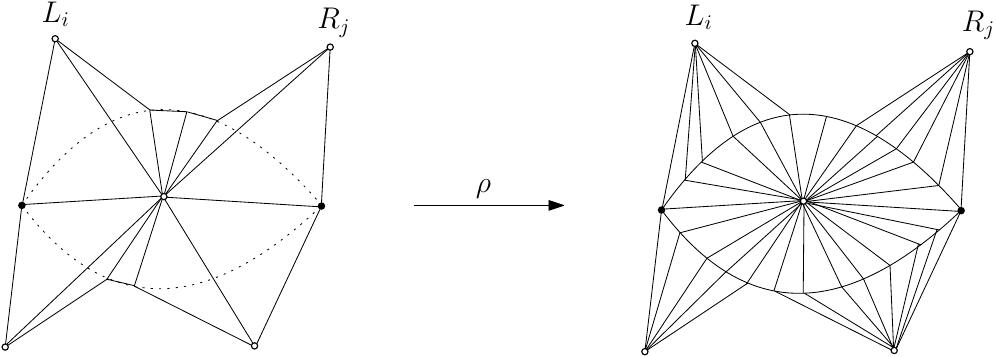}
	\caption{}
	\label{f:retraction}
\end{figure}

Let $x$ be as in Lemma~\ref{lem:fake vertex}. Then for any fake vertex $x'\in\lk(x,Q')$, there are exactly two triangles in $\St(x,Q')$ containing $x'$. The union of these two triangles is called an \emph{ear} of $\St(x,Q')$. Note that $\St(x,Q')$ can have at most four ears. By Lemma~\ref{lem:fake vertex} again, there is a subset of $\St(x,Q')$ which is mapped to the cell containing $q'(x)$ homeomorphically by $q$. We denote this subset by $C_x$. Note that the boundary $\partial C_x$ cuts through the ears of $\St(x,Q')$ and follows the edges in $\lk(x,Q')$ that are between two real vertices.

\begin{lemma}
	\label{lem:real1}
Suppose $x\in Q'$ is flat, deep and real. Let $v=q'(x)$. Suppose the cycle $\omega=q'(\lk(x,Q'))$ in $\lk(v,X^{(2)}_\Gamma)$ satisfies Lemma~\ref{lem:2pi cycle} (1), then the restriction of $q$ to $\St(x,Q')$ is an embedding. Moreover, there are exactly four fake vertices in $\omega$, and the four cells containing these four fake vertices can be ordered as $\{\Pi_i\}_{i=1}^{4}$ such that (see Figure~\ref{f:flat2})
\begin{enumerate}
	\item each $\Pi_i$ contains $v$, and 
	$\Pi_1,\Pi_2,\Pi_3,\Pi_4$ are in the same block;
	\item $v$ is the right tip of $\Pi_2$ and is the left tip of $\Pi_1$;
	\item $v$ is not a tip of $\Pi_3$ or $\Pi_4$;
	\item each of $\Pi_3\cap\Pi_2$, $\Pi_3\cap\Pi_1$, $\Pi_4\cap\Pi_2$ and $\Pi_4\cap\Pi_1$ contains at least one edge, hence $\Pi_i\cap (\Pi_1\cup\Pi_2)$ is a half of $\Pi_i$ for $i=3,4$ by Lemma~\ref{cor:connected intersection};
	\item if in addition each vertex in $\lk(x,Q')$ is deep and flat, then $\St(x,Q')$ is contained in $\cup_{i=1}^4C_{x_i}$, where $x_i$ is the fake vertex in $\lk(x,Q')$ such that $q(x_i)$ is the fake vertex in $\Pi_i$;
	\item under the assumption of (5), $x\in C_{x_i}$ for each $i$.
\end{enumerate}
\end{lemma}

\begin{figure}[h!]
	\centering
	\includegraphics[scale=0.66]{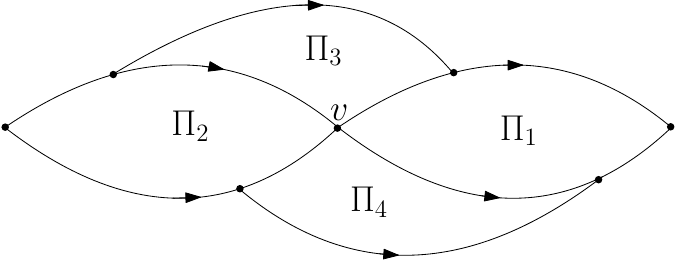}
	\caption{}
	\label{f:flat2}
\end{figure}

\begin{proof}
Let $B$ be the block containing $\omega$. We argue as before to deduce $\omega$ is a concatenation of two paths $\sigma_1$ and $\sigma_2$ connecting two necks of $\lk(v,B^{(1)})$ such that each of them satisfies one of the three possibilities in Lemma~\ref{lem:exactly pi}. Moreover, by Lemma~\ref{lem:local}, $k=3$ in Lemma~\ref{lem:exactly pi} (3). Thus there are four fake vertices in $\omega$. We can assume $\Pi_1=\Pi$ is the base cell of $X_n$. Then the conclusions (1)-(4) of the lemma follows by letting $\Pi_2=r^{-1}\Pi$, $\Pi_3=\di{i}o$ and $\Pi_4=\ui{j}o$ for some $1\le i,j\le n-1$. For (5), it follows from our assumption that the link of each fake vertex in $\lk(x,Q')$ is as described in the proof of Lemma~\ref{lem:fake vertex}. Thus (5) follows from the definition of $\rho$ and $C_{x_i}$. To see (6), note that for each $C_{x_i}$, $x$ is either a real vertex in an ear of $\St(x_i,Q')$, or $x$ is inside an edge of $\lk(x_i,Q')$ whose two end points are real. Thus $x\in C_{x_i}$. We refer to Figure~\ref{f:flat3} for a particular example.
\end{proof}

\begin{figure}[h!]
	\centering
	\includegraphics[width=1\textwidth]{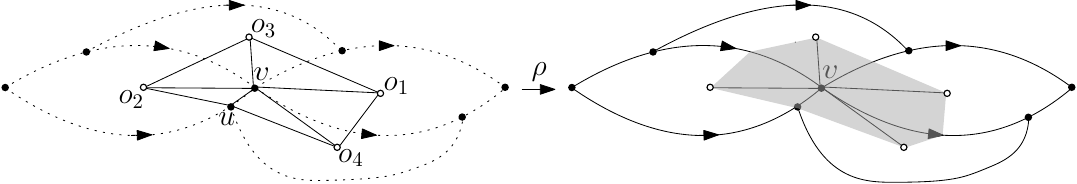}
	\caption{The path $o_2\to o_3\to o_1$ satisfies Lemma~\ref{lem:exactly pi} (3), and the path $o_2\to u\to o_4\to o_1$ satisfies Lemma~\ref{lem:exactly pi} (1).}
	\label{f:flat3}
\end{figure}

\begin{lemma}
	\label{lem:real2}
Let $x,v,\omega$ be as in Lemma~\ref{lem:real1}. If $\omega$ satisfies Lemma~\ref{lem:2pi cycle} (2), then the restriction of $q$ to $\St(x,Q')$ is an embedding. 

Let $\overline{ab},\overline{ac}\subset\Gamma$ be the defining edges of $B_1$ and $B_2$ as in Lemma~\ref{lem:2pi cycle} (2). For $i=1,2$, let the label of the defining edge of $B_i$ be $n_i$. Then there are exactly four fake vertices in $\omega$, and the four cells contains these four fake vertices can be ordered as $\{\Pi_i\}_{i=1}^{4}$ (see Figure~\ref{f:flat4}) such that
\begin{enumerate}
	\item each $\Pi_i$ contains $v$;
	\item $\Pi_1$ and $\Pi_2$ are in $B_1$, and $\Pi_3$ and $\Pi_4$ are in $B_2$;
	\item $\Pi_1\cap\Pi_2$ has $n_1-1$ edges, $\Pi_3\cap\Pi_4$ has $n_2-1$ edges, $\Pi_1\cap\Pi_3$ intersects along an $a$--edge, and $\Pi_2\cap\Pi_4$ intersects along an $a$--edge;
	\item the analogous statements of Lemma~\ref{lem:real1} (5) and (6) hold.
\end{enumerate}
\end{lemma}

\begin{figure}[h!]
	\centering
	\includegraphics[width=1\textwidth]{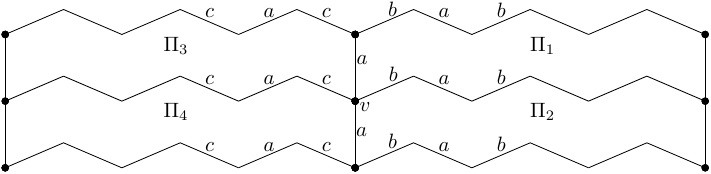}
	\caption{}
	\label{f:flat4}
\end{figure}

\begin{remark}
	\label{rmk:real2}
We did not specify orientations of edges in Figure~\ref{f:flat4}. However, it follows from Lemma~\ref{lem:real2} (3) that all vertical edges in Figure~\ref{f:flat4} have orientations pointing towards the same direction (all up or all down). All non-vertical edges in Figure~\ref{f:flat4} that are in $B_1$ have orientations pointing towards the same direction (all left or all right), and a similar statement holds with $B_1$ replaced by $B_2$.
\end{remark}

\begin{proof}[Proof of Lemma~\ref{lem:real2}]
We denote the two vertices in $\omega\cap B_1\cap B_2$ by $a^i$ and $a^o$ respectively. For $i=1,2$, let $\sigma_i$ be the sub-path of $\omega$ from $a^o$ to $a^i$ in $\lk(v,B^{(1)}_i)$. We now study the possibilities for $\sigma_1$. We identify $B_1$ with $X_{n_1}$ and assume $v=\ell$ in the cell $\Pi$ of $X_{n_1}$.

By the discussion in Section~\ref{subsec:local}, all edges of type II in $\lk(v,B^{(1)}_1)$ (see Figure~\ref{f:real}) are between two interior vertices, and there are no edges between real vertices. Thus to travel from one real vertex to another real vertex in $\lk(v,B^{(1)}_1)$, one has to go through at least two edges of type I. However, only two edges of type I do not bring one from $a^i$ to $a^o$. So we need at least one another edge. By Lemma~\ref{lem:edge type II}, an edge in $\lk(v,B^{(1)}_1)$ has angular length at least $\frac{1}{2n}2\pi$. Thus $\sigma$ is made of two edges of type I and one edges of type II with minimal angular length. Thus $\sigma_1=a^o\to \ui{0}o\to\ui{1}o\to a^i$ or $\sigma_1=a^o\to \di{n_1-1}o\to\di{n}o\to a^i$. Note that $\ui{0}\Pi\cap\ui{1}\Pi$ has $n_1-1$ edges, and $\di{n-1}\Pi\cap\di{n}\Pi$ has $n_1-1$ edges (see Figure~\ref{f:flat1}). A similar statement holds for $\sigma_2$. Now the lemma follows from the definition of $\rho$.
\end{proof}

\begin{figure}[h!]
	\centering
	\includegraphics[scale=1.1]{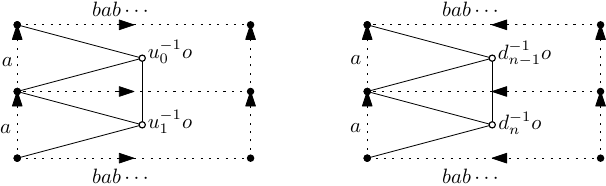}
	\caption{}
	\label{f:flat1}
\end{figure}

\begin{lemma}
\label{lem:real3}
Let $x,v,\omega$ be as in Lemma~\ref{lem:real1}. If $\omega$ satisfies Lemma~\ref{lem:2pi cycle} (3), then the restriction of $q$ to $\St(x,Q')$ is an embedding. Moreover, let $\{B_i\}_{i=1}^3$ be as in  Lemma~\ref{lem:2pi cycle} (3). Then there are three cells $\{\Pi_i\subset B_i\}_{i=1}^3$ as in Figure~\ref{f:flat5} (left) such that
\begin{enumerate}
	\item $v\in\Pi_i$ for $1\le i\le 3$;
	\item each of $\Pi_1\cap\Pi_2$, $\Pi_2\cap\Pi_3$, and $\Pi_3\cap\Pi_1$ consists of one edge;
	\item the analogous statements of Lemma~\ref{lem:real1} (5) and (6) hold.
\end{enumerate}
A similar statement holds when $\omega$ satisfies Lemma~\ref{lem:2pi cycle} (4), where we have four squares around $v$, see Figure~\ref{f:flat5} right.
\end{lemma}

\begin{figure}[h!]
	\centering
	\includegraphics[scale=1.3]{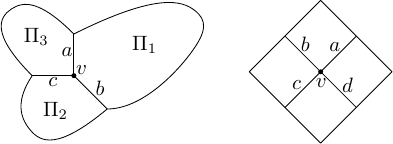}
	\caption{}
	\label{f:flat5}
\end{figure}

\begin{proof}
Take $\omega=\sigma_1\cup\sigma_2\cup\sigma_3$ where each $\sigma_i$ is a sub-path of $\omega$ made of two edges such that $\sigma_i\subset B_i$. For $1\le i\le 3$, we take $\Pi_i$ to be the cell that contains the fake vertex in $\sigma_i$. Then $\sigma_i\subset\Pi_i$. Thus $q'(\St(x,Q'))\subset\Xb_{\Gamma}$. Consequently, $\rho$ is the identity on $q'(\St(x,Q'))$. Hence the lemma follows.
\end{proof}

\subsection{A new cell structure on the quasiflat}
\label{subsec:new cell structure}
It follows from Lemma~\ref{lem:well-defined}, Lemma~\ref{lem:fake vertex}, Lemma~\ref{lem:real1}, Lemma~\ref{lem:real2} and Lemma~\ref{lem:real3} that if $x\in Q'$ is flat and deep, then the restriction of $q$ to $\St(x,Q')$ is well-defined (in particular, it is well-defined outside a compact set) and it is an embedding. Now we think of the range of $q$ as $\Xa_\Gamma$. In this subsection we want to \textquotedblleft pull back\textquotedblright\ the cell structure on $\Xa_\Gamma$ to an appropriate subset of $Q'$ via $q$.

Choose a base point $x_0\in Q'$. For $R>0$, let $Q'_R$ be the smallest subcomplex of $Q'$ that contains $Q'\setminus B(x_0,R)$. Recall that $B(x_0,R)$ is the ball of radius $R$ centered at $x_0$. Let $Q'_{R^+}$ be the subcomplex made of all triangles of $Q'$ which have non-trivial intersection with $Q'_R$. By Theorem~\ref{thm:MSquasiflats}, we choose $R$ large enough such that each vertex in $Q$ that has combinatorial distance $\le 5$ from $Q'_{R^+}$ is flat and deep.

\begin{lemma}
	\label{lem:intersection of cells}
	Suppose $C_{x_1}\cap C_{x_2}\neq\emptyset$ for two fake, deep and flat vertices $x_1,x_2\in Q'$. We also assume each vertex in $\St(x_1,Q')$ is flat and deep. Let $\Pi_i$ be the cell containing $o_i=q(x_i)$. Then $q$ maps $C_{x_1}\cup C_{x_2}$ homeomorphically onto $\Pi_1\cup\Pi_2$. In particular, $C_{x_1}\cap C_{x_2}$ is a connected interval (possibly degenerate).
\end{lemma}

\begin{proof}
Note that $x_1$ and $x_2$ have combinatorial distance $\le 2$ in $Q'$. We claim that if $x_1$ and $x_2$ are not adjacent, then $o_1$ and $o_2$ are not adjacent. To see this, note that there exists $x\in\lk(x_1,Q')$ such that $x_1,x_2\in\lk(x,Q')$. Since $x$ is deep and flat, the claim follows from the descriptions of $\lk(x,Q')$ in Lemma~\ref{lem:fake vertex}, Lemma~\ref{lem:real1}, Lemma~\ref{lem:real2} and Lemma~\ref{lem:real3}.
	
	
	
If $x_1$ and $x_2$ are adjacent, then so are $o_1$ and $o_2$. It is clear that $q(C_{x_1}\cap C_{x_2})\subset\Pi_1\cap\Pi_2$. Let $\omega$ be the arc on $\partial C_{x_1}$ which is mapped homeomorphically to $\Pi_1\cap\Pi_2\subset\partial \Pi_1$. Then $C_{x_1}\cap C_{x_2}\subset\omega$ since $q|_{C_{x_1}}$ is an embedding. However, $\omega$ cuts through an ear of $\St(x_1,Q')$ thus, by definition of $C_{x_i}$, we have $\omega\subset C_{x_1}\cap C_{x_2}$. Hence $C_{x_1}\cap C_{x_2}=\omega$ and the lemma follows.

Suppose $x_1$ and $x_2$ are not adjacent. Since $C_{x_1}\cap C_{x_2}\neq\emptyset$, the intersection $\St(x_1,Q')\cap\St(x_2,Q')$ consists of at least one vertex. If there is a fake vertex $x\in\St(x_1,Q')\cap\St(x_2,Q')$, then let $\Pi_x$ be the cell containing $q(x)$. Since $x_1$ and $x_2$ are fake vertices in $\lk(x,Q')$, by the description of $\lk(x,Q')$ in Lemma~\ref{lem:fake vertex}, we know $\Pi_1\cap\Pi_x$ contains an edge, $\Pi_2\cap\Pi_x$ contains an edge and $(\Pi_1\cap\Pi_x)\cap(\Pi_2\cap\Pi_x)$ is at most one point. Thus $\Pi_1\cap\Pi_2$ is one point by Lemma~\ref{lem:small intersection} (note that $\Pi_1,\Pi_2$ and $\Pi_x$ are in the same block). Thus the lemma follows by the discussion in the previous paragraph. Now suppose there are no fake vertices in $\St(x_1,Q')\cap\St(x_2,Q')$. If there is an edge $e$ in $\St(x_1,Q')\cap\St(x_2,Q')$, then $e\subset C_{x_1}\cap C_{x_2}$ since both endpoints of $e$ are real. Thus $q(e)$ is an edge in $\Pi_1\cap\Pi_2$. However, $\Pi_1\cap\Pi_2$ has at most one edge since $o_1$ and $o_2$ are not adjacent. Thus $e= C_{x_1}\cap C_{x_2}$ and the lemma follows. If there are no edges in $\St(x_1,Q')\cap\St(x_2,Q')$, then let $x$ be a vertex in this intersection. Since $x_1$ and $x_2$ are two fake vertices in $\lk(x,Q')$, in all cases of Lemma~\ref{lem:real1}, Lemma~\ref{lem:real2} and Lemma~\ref{lem:real3}, $\Pi_1\cap\Pi_2$ is a point whenever $\St(x_1,Q')\cap\St(x_2,Q')$ is a real vertex. Hence the lemma follows.
\end{proof}

\begin{lemma}
	\label{lem:cover}
$Q'_R$ is contained in the union of $C_x$ with $x$ varying among vertices of $Q'_{R^+}$ that are flat, deep and fake.
\end{lemma}

\begin{proof}
Recall that there are no triangles with three fake vertices in $X_\Gamma$, thus the same is true for $Q'$. Thus $Q'_R$ is contained in the union of $\St(x,Q')$ with $x$ ranging over real vertices in $Q'_R$. By Lemma~\ref{lem:real1}, Lemma~\ref{lem:real2}, Lemma~\ref{lem:real3} and our choice of $R$, for any real $x\in Q'_R$, $\St(x,Q')$ is contained in the union of $C_y$ with $y$ varying among fake vertices in $\lk(x,Q')$. Note that $y\in Q'_{R^+}$. Then the lemma follows.
\end{proof}

Let $Q_R$ be the union of $C_x$ with $x$ varying among fake vertices of $Q'_{R^+}$. By Lemma~\ref{lem:intersection of cells} and Lemma~\ref{lem:cover}, $Q_R$ has a well-defined cell structure whose closed $2$--cells are the $C_x$, and whose edges (resp.\ vertices) are arcs (resp.\ points) in the boundary of $C_x$ which are mapped to edges (resp.\ vertices) in $\Xa_\Gamma$ by $q$. Also, Lemma~\ref{lem:intersection of cells} implies that we can pullback the orientation and labeling of edges of $\Xa_\Gamma$ to orientation and labeling of edges of $Q_R$. 

A vertex of $Q_R$ is \emph{interior} if it has a neighborhood in $Q_R$ which is homeomorphic to an open disc. Let $\St(x,Q_R)$ be the union of cells of $Q_R$ that contain $x$. Now we look at the structure of $\St(x,Q_R)$.

\begin{lemma}
	\label{lem:star homeo}
Let $x\in Q_R$ be an interior vertex. Then the following are the only possibilities for $\St(x,Q_R)$.
\begin{enumerate}
	\item $\St(x,Q_R)$ is a union of two $2$--cells.
	\item The point $x$ corresponds to a real vertex in $Q'$, and $q$ maps $\St(x,Q_R)$ homeomorphically onto the union of the $\Pi_i$ in $\Xa_\Gamma$ described in Lemma \ref{lem:real1}, Lemma~\ref{lem:real2}, or Lemma~\ref{lem:real3}.
\end{enumerate}
\end{lemma}

\begin{proof}
Suppose $x$ is not a real vertex of $Q'$. There are at least two $2$--cells $C_{x_1}$, $C_{x_2}$ in $Q_R$ that contain $x$. Thus $x$ is in the interior of an ear of $\St(x_i,Q')$ for $i=1,2$. Hence $x_1$ and $x_2$ are adjacent in $Q'$, and $\St(x_1,Q')$ and $\St(x_2,Q')$ share an ear. Now (1) follows.

Suppose $x$ is a real vertex of $Q'$. Then each vertex of $\St(x,Q')$ is flat and deep by our choice of $R$. If $x$ satisfies the assumptions of Lemma~\ref{lem:real1}, then $\cup_{i=1}^4 C_{x_i}\subset \St(x,Q')$, by Lemma~\ref{lem:real1} (6). By Lemma~\ref{lem:real1} (5), $\cup_{i=1}^4 C_{x_i}$ contains a disc neighborhood of $x$, thus $\cup_{i=1}^4 C_{x_i}=\St(x,Q')$. Since, by Lemma~\ref{lem:intersection of cells}, $q$ maps $C_{x_i}\cup C_{x_j}$ homeomorphically onto $\Pi_i\cup\Pi_j$, $q$ maps $\cup_{i=1}^4 C_{x_i}$ homeomorphically onto $\cup_{i=1}^4\Pi_i$. The cases of Lemma~\ref{lem:real2} and Lemma~\ref{lem:real3} are similar.
\end{proof}

\begin{definition}
	\label{d:types}
	An interior vertex $v\in Q_R$ is of \emph{type O} if it satisfies Lemma~\ref{lem:star homeo} (1), and is of \emph{type I, II, or III} if it satisfies Lemma~\ref{lem:real1}, Lemma~\ref{lem:real2}, or Lemma~\ref{lem:real3}, respectively. The \emph{support} of a $2$--cell in $Q_R$ is the defining edge of the block of $\Xa_\Gamma$ that contains the $q$--image of this $2$--cell. The \emph{support} of a vertex of $Q_R$ is the union of the supports of $2$--cells in $Q_R$ that contain this vertex. For a type III vertex $v$, its support is either a triangle or a square, and $v$ is called either a \emph{$\triangle$--vertex} or \emph{a $\square$--vertex}, respectively. (See Table~\ref{t:flat} on page \pageref{t:flat}.) By Lemma~\ref{lem:2pi cycle} (3) and (4), the Coxeter group whose defining graph is the support of $v$ acts on the Euclidean plane. 
\end{definition}

\begin{table}[]
	\begin{tabular}{|c|c|}
		\cline{1-2}
		\begin{tabular}[c]{@{}c@{}}type O\end{tabular} & 
		\begin{tabular}[c]{@{}c@{}}\\ \includegraphics[width=0.39\textwidth]{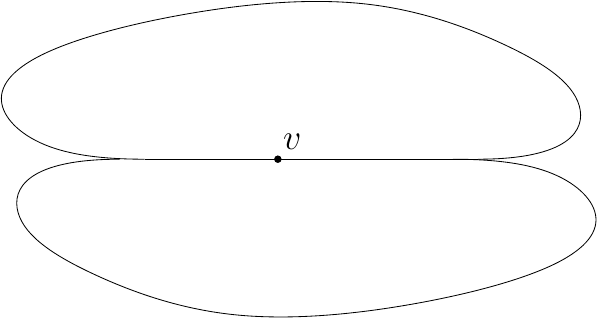} \\ \end{tabular}
		\\ \cline{1-2}
		\begin{tabular}[c]{@{}c@{}}type I\end{tabular} & 
		\begin{tabular}[c]{@{}c@{}}\\ \includegraphics[width=0.41\textwidth]{flat2} \\ \end{tabular}
		\\ \cline{1-2}
		\begin{tabular}[c]{@{}c@{}}type II\end{tabular} & 
		\begin{tabular}[c]{@{}c@{}}\\ \includegraphics[width=0.6\textwidth]{flat4}\end{tabular}
		\\ \cline{1-2}
		\begin{tabular}[c]{@{}c@{}}type III\\ $\triangle$--vertex\end{tabular} & 
		\begin{tabular}[c]{@{}c@{}}\\ \includegraphics[width=0.33\textwidth]{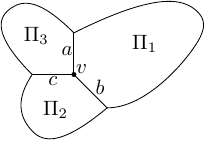}\end{tabular}
		\\ \cline{1-2}
		\begin{tabular}[c]{@{}c@{}}type III \\ $\square$--vertex\end{tabular} & 
		\begin{tabular}[c]{@{}c@{}}\\ \includegraphics[width=0.21\textwidth]{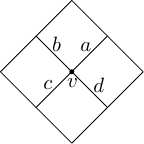}\end{tabular}
		\\ \cline{1-2}
	\end{tabular}
	\caption{Types of flat vertices.}
	\label{t:flat}
\end{table}

Note that $\St(v_1,Q_R)\cap\St(v_2,Q_R)$ contains two $2$--cells for two adjacent vertices $v_1,v_2$ of type III, thus we have the following result.

\begin{lemma}
	\label{lem:same tag}
If two vertices of type III of $Q_R$ are adjacent, then they have the same support.
\end{lemma}

Since $Q'_R\subset Q_R$ by Lemma~\ref{lem:cover}, we will assume the quasiflat is represented by $q\colon Q_R\to \Xa_\Gamma$. Each $2$--cell in $Q_R$ corresponds to a fake vertex in $Q'_R$, which is called the vertex \emph{dual} to this $2$--cell.

\section{Singular lines, singular rays and atomic sectors}
\label{sec:singular lines}
Throughout this section $\Gamma$ will the defining graph of an Artin group $A_\Gamma$ with dimension $\le 2$.
\subsection{Singular lines and singular rays}
\label{subsec:singular lines}
\begin{definition}
	\label{def:diamond line}
A \emph{diamond line} is a locally injective cellular map $k\colon U\to \Xa_\Gamma$ satisfying
\begin{enumerate}
	\item $U=\cup_{i=-\infty}^{\infty} C_i$ such that each $C_i$ is a $2$--cell whose boundary is a $2n$--gon for $n\ge 3$, and each $C_i$ intersects $C_{i+1}$ and $C_{i-1}$ in opposite vertices (and has empty intersection with other $2$--cells), see Figure~\ref{f:singular_line1} for the $n=3$ case;
	\item $k(U)$ is contained in a block;
	\item $k(C_i\cap C_{i+1})$ is a tip of both $k(C_i)$ and $k(C_{i+1})$.
\end{enumerate}
We define a \emph{diamond ray} in a similar way by replacing $\cup_{i=-\infty}^{\infty} C_i$ by $\cup_{i=0}^{\infty} C_i$.
\end{definition}
\begin{figure}[h!]
	\centering
	\includegraphics[width=1\textwidth]{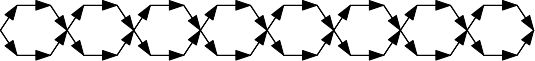}
	\caption{A part of a diamond line.}
	\label{f:singular_line1}
\end{figure}

Let $k,U$ be as in Definition~\ref{def:diamond line} and let $\overline{ab}\subset\Gamma$ be the defining edge of the block containing $k(U)$. Then the two lines in the $1$--skeleton of $U$ corresponding to the bi-infinite alternating word $\cdots ababa\cdots$ are called the \emph{boundary lines} of $U$. 

\begin{remark}
Roughly speaking, each diamond line corresponds to the centralizer of the stabilizer of the block of $\Xa_\Gamma$ that contains this diamond line (note that the stabilizer of a block is a conjugate of the standard subgroup associated with the defining edge of this block).
\end{remark}

Each diamond line is embedded and quasi-isometrically embedded in $\Xa_\Gamma$. This is clear when $\Gamma$ is an edge, and the general case follows from a result by Charney and Paris \cite[Theorem 1.2]{charney2014convexity}.
Note that each large block is a union of diamond lines.

Before we state the next definition, recall that each Coxeter group with defining graph $\Gamma$ gives rise to its \emph{Davis complex} $\D_\Gamma$, whose $1$--skeleton is the Cayley graph of the associated Coxeter group with bigons collapsed to single edges. In particular, each edge of $\D_\Gamma$ is labeled by a vertex of $\Gamma$. A \emph{wall} in $\D_\Gamma$ is the fixed point set of a reflection in the associated Coxeter group. 

\begin{definition}
	\label{def:wall line}
Suppose the defining graph $\Gamma$ of the Artin group $A_\Gamma$ contains a triangle $\Delta\subset\Gamma$ such that the labels $m,n,r$ of the three sides of $\Delta$ satisfy $\frac{1}{m}+\frac{1}{n}+\frac{1}{r}=1$. Let $\D_\Delta$ be the Davis complex of the Coxeter group with the defining graph $\Delta$. Then $\D_\Delta$ is isometric to $\mathbb E^2$ and edges of $\D_\Delta$ are labeled by vertices of $\Delta$. Let $U$ be the \emph{carrier} of a wall in $\D_\Delta$ (i.e.\ $U$ is the union of cells that intersect this wall). See Figure~\ref{f:singular_line2} for an example when $(m,n,r)=(2,3,6)$.
\begin{figure}[h!]
	\centering
	\includegraphics[width=1\textwidth]{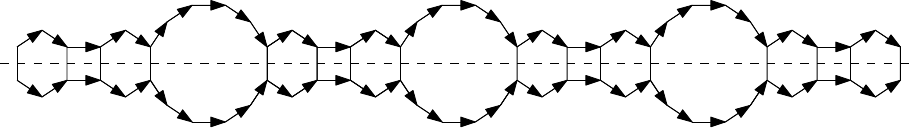}
	\caption{A part of a Coxeter line.}
	\label{f:singular_line2}
\end{figure}
A \emph{Coxeter line} of $\Xa_\Gamma$ is a locally injective cellular map $k\colon U\to \Xa_\Gamma$ such that
	\begin{enumerate}
		\item $k$ preserves the label of edges;
		\item if we pull back the orientation of edges of $\Xa_\Gamma$ to $U$, then there are no orientation reversing vertices in the boundary of $U$.
	\end{enumerate}
If we restrict $k$ to a subcomplex of $U$ homeomorphic to $[0,\infty)\times[0,1]$ (resp.\ $[0,k]\times[0,1]$), we obtain a \emph{Coxeter ray} (resp.\ \emph{Coxeter segment}).

Let $V$ be the carrier (in $\D_\Delta$) of the region bounded by two different parallel walls $\W_1,\W_2$ in $\D_{\Delta}$. A \emph{thickened Coxeter line} is a locally injective cellular map $k'\colon V\to \Xa_\Gamma$ such that $k'$ preserves the label of edges and $k'$ restricted to the carriers of $\W_1$ and $\W_2$ forms two Coxeter lines.
\end{definition}

The maps $k$ and $k'$ in Definition~\ref{def:wall line} are injective and they are quasi-isometric embeddings. This can be deduced by considering the quotient homomorphism from $A_\Gamma$ to the Coxeter group with the defining graph $\Gamma$. 

For a set of vertices $V$ in $\Gamma$, let $V^{\perp}$ be the collection of vertices of $\Gamma$ that are adjacent to each element in $V$ along an edge labeled by $2$. If both $V\neq\emptyset$ and  $V^{\perp}\neq\emptyset$, as $A_\Gamma$ is assumed to have dimension $\le 2$, the subgroup of $A_\Gamma$ generated by $V^\perp$ has to be free, and so is the subgroup of $A_\Gamma$ generated by $V$. In particular, neither $V^{\perp}$ nor $V$ contain a pair of adjacent vertices of $\Gamma$.

\begin{definition}
	\label{def:plain line}
A \emph{plain line} is a line $L$ in the $1$--skeleton of $\Xa_\Gamma$ such that the collection $V_L$ of labels of edges satisfies either $V_L$ is a singleton or $V^{\perp}_L\neq\emptyset$.
A \emph{plain ray} is defined in a similar fashion. A plain line or ray is \emph{single-labeled} if $V_L$ is a singleton. A plain ray is \emph{chromatic} if $V_r$ is not a singleton for any sub-ray $r$ of this plain ray.

Let $U=\cup_{i=-\infty}^{\infty}C_i$ be a union of $2$--cells such that 
\begin{enumerate}
	\item $\partial C_i$ is a $2n$--gon for each $i$ ($n\ge 2$);
	\item $C_i\cap C_{i-1}$ and $C_i\cap C_{i+1}$ are two disjoint connected paths in $\partial C_i$ and each of them has $n-1$ edges;
	\item $C_i\cap C_j=\emptyset$ for $|i-j|\ge 2$.
\end{enumerate}
A \emph{thickened plain line} is a cellular embedding $k\colon U\to\Xa_\Gamma$ such that $k$ restricted to the two boundary lines of $U$ are plain lines.
\end{definition}
Again it follows from \cite[Theorem 1.2]{charney2014convexity} that each plain line is quasi-isometrically embedded for some uniform quasi-isometric constants independent of the plain line.

\begin{definition}
A \emph{singular ray} is either a diamond ray, or a Coxeter ray, or a plain ray. We define \emph{singular line} analogously.
\end{definition}

\subsection{Flats, half-flats and sectors}
\label{subsec:atomic sector}
Since the left action $A_\Gamma\act \Xa_\Gamma$ is simply transitive on the vertex set of $\Xa_\Gamma$, we choose an identification between elements of $A_\Gamma$ and vertices of $\Xa_\Gamma$.

\begin{definition}
	\label{def:diamond-plain flat}
Let $L$ be a diamond line containing the identity element of $A_\Gamma$. A \emph{diamond-plain flat} is a subcomplex of $\Xa_\Gamma$ of the form $g\cup_{i=-\infty}^{\infty} a^i L$ or $g\cup_{i=-\infty}^{\infty} b^i L$, where $a$ and $b$ are the labels of edges in $L$, $a^i L$ means the left translation of $L$ under the group element $a^i$ and $g$ is an element in $A_\Gamma$. Note that each diamond-plain flat can be naturally realized as the image of a locally injective cellular map $f\colon U\to \Xa_\Gamma$ where $U$ is a union of subcomplexes isomorphic to diamond lines; see Figure~\ref{f:flat_diamond-plain}.
\end{definition}

\begin{figure}[h!]
	\centering
	\includegraphics[width=0.7\textwidth]{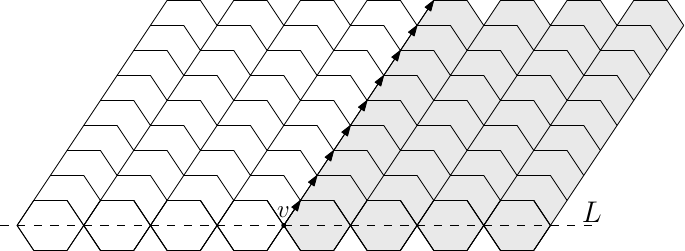}
	\caption{A part of a diamond-plain flat and a diamond-plain sector (shaded).}
	\label{f:flat_diamond-plain}
\end{figure}

\begin{lemma}
	\label{lem:diamond embedding}
The map $f$ above is an embedding.
\end{lemma}

\begin{proof}
Since the image of $f$ is contained in a block, it suffices to consider the case when $\Gamma$ is an edge. One can deduce the injectivity of $f$ from the solution of the word problem for spherical Artin groups \cite{brieskorn1972artin,deligne}. Here we provide a geometric proof depending on a complex constructed by Jon McCammond \cite{McCammond2010}. The construction of such complex was presented in the proof of \cite[Theorem 5.1]{huang2015cocompactly}.

Suppose the edge of $\Gamma$ is labeled by $n$. Let $K_n$ be the cube complex described in the Figure~\ref{f:cube} below.
\begin{figure}[h!]
	\centering
	\includegraphics[width=0.8\textwidth]{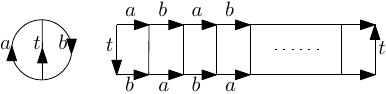}
	\caption{}
	\label{f:cube}
\end{figure}
On the left we see part of the $1$--skeleton of $K_n$ consisting of
three edges labelled by $a,b,t$, and the right side indicates
how to attach a rectangle (subdivided into $n$ squares)
along its boundary path $\underbrace{ab\dots}_{n}t^{-1}
\underbrace{ b^{-1} a^{-1}\dots}_{n}t^{-1}$. It is easy
to check that the link of each of the two vertices in $K_n$ is
isomorphic to the spherical join of two points with $n$ points,
hence $K_n$ is nonpositively curved and its universal cover $\widetilde K_n$ is isometric to tree times $\mathbb E^1$. There is a homotopy equivalence $h\colon K_n\to P_\Gamma$ by collapsing the $t$--edge. This induces a map $\widetilde{K}_n\to\Xa_\Gamma$. Note that $h$ gives a one to one correspondence between lifts of $t$--edges in $\widetilde{K}_n$ and vertices in $\Xa_\Gamma$; as well as one to one correspondence between vertical flat strips in $\widetilde{K}_n$ and diamond lines in $\Xa_\Gamma$. Now the lemma follows from the CAT(0) geometry on $\widetilde{K}_n$.
\end{proof}

Thus each diamond-plain flat $F$ is homeomorphic to $\mathbb R^2$. Moreover, there is a subgroup $A_F\le A_\Gamma$ isomorphic to $\mathbb Z \oplus\mathbb Z$ acting cocompactly on $F$ (for example, when $F=\cup_{i=-\infty}^{\infty} a^i L$, then $A_F$ is generated by $a$ and the centralizer of the standard subgroup of $A_\Gamma$ generated by $a$ and $b$). 

\begin{definition}
		\label{def:diamond-plain sector}
Choose a diamond-plain flat $F$ and let $v\in F$ be a vertex of a diamond line $L$ in $F$ being the intersection of two $2$--cells of $L$. Let $L_1\subset L$ be a diamond ray based at $v$ and let $L_2\subset F$ be a plain ray based at $v$. Then the region in $F$ bounded by $L_1$ and $L_2$ (including $L_1$ and $L_2$) is called a \emph{diamond-plain sector}; see Figure~\ref{f:flat_diamond-plain}. $L_1$ and $L_2$ are called the \emph{boundary rays} of the diamond-plain sector.
\end{definition}

\begin{definition}
Let $L\subset\Xa_\Gamma$ be a Coxeter line containing the identity element of $A_\Gamma$ and let $e$ be the edge in $L$ such that $e$ is dual to the wall in $L$ and $e$ contains the identity element of $A_\Gamma$. Suppose the label of $e$ is $a$. Then a \emph{Coxeter-plain flat} is a subcomplex of $\Xa_\Gamma$ of the form $g\cup_{i=-\infty}^{\infty}a^i L$, where $g\in A_\Gamma$. Let $L'\subset L$ be a Coxeter ray starting at the edge $e$. A \emph{Coxeter-plain sector} is a subcomplex of $\Xa_\Gamma$ of the form $g\cup_{i=0}^{\infty}a^i L'$, where $g\in A_\Gamma$; see Figure~\ref{f:flat_Coxeter-plain}. The Coxeter-plain sector has two \emph{boundary rays}, one is $L'$ and another one is the plain ray containing $e$.
\end{definition}

\begin{figure}[h!]
	\centering
	\includegraphics[width=0.8\textwidth]{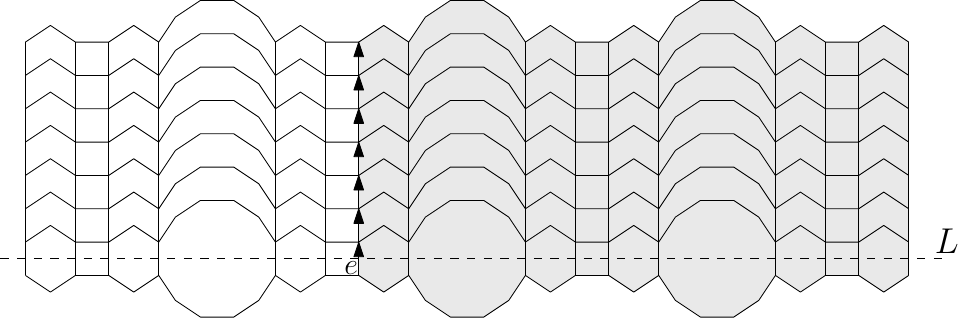}
	\caption{A part of a Coxeter-plain flat and a Coxeter-plain sector (shaded).}
	\label{f:flat_Coxeter-plain}
\end{figure}

Since the Artin monoid injects into the Artin group by the work of Paris \cite{paris2002artin}, $L\cap a^iL=\emptyset$ for $i\ge 2$ and $L\cap aL$ is a boundary line of $L$. Thus each Coxeter-plain flat is a subcomplex of $\Xa_\Gamma$ homeomorphic to $\mathbb R^2$.

\begin{definition}
	\label{def:Coxeter flat}
Let $\Delta$ and $\D_{\Delta}$ be as Definition~\ref{def:wall line}. A \emph{Coxeter flat} is a locally injective cellular map $f\colon \D_{\Delta}\to \Xa_\Gamma$ such that 
\begin{enumerate}
	\item $f$ preserves the label of edges;
	\item the orientation of edges in $\D_\Delta$ induced by $f$ satisfies the following: if two edges of $\D_{\Delta}$ are dual to parallel walls in $\D_{\Delta}$, then they are oriented towards the same direction.
\end{enumerate}	
\end{definition}

By considering the $1$--Lipschitz quotient homomorphism from $A_\Gamma$ to the Coxeter group with defining graph $\Gamma$ as before (cf.\ the discussion after Definition~\ref{def:wall line}), we know each Coxeter flat is embedded and quasi-isometrically embedded.

For each Coxeter flat $F$, there are exactly two families of parallel walls whose carriers in $F$ give rise to Coxeter lines. To see this, choose a $2$--cell $C\subset F$ with the maximal number of edges on its boundary and let $e_1,e_2\subset \partial C$ be the two edges containing a tip $t\in \partial C$. Then by 
Definition~\ref{def:wall line} (2) and Definition~\ref{def:Coxeter flat} (2), the carrier of a wall $\W\subset F$ is a Coxeter line if and only if $\W$ is parallel to the wall of $F$ dual to $e_1$ or $e_2$.

\begin{figure}[h!]
	\centering
	\includegraphics[width=0.6\textwidth]{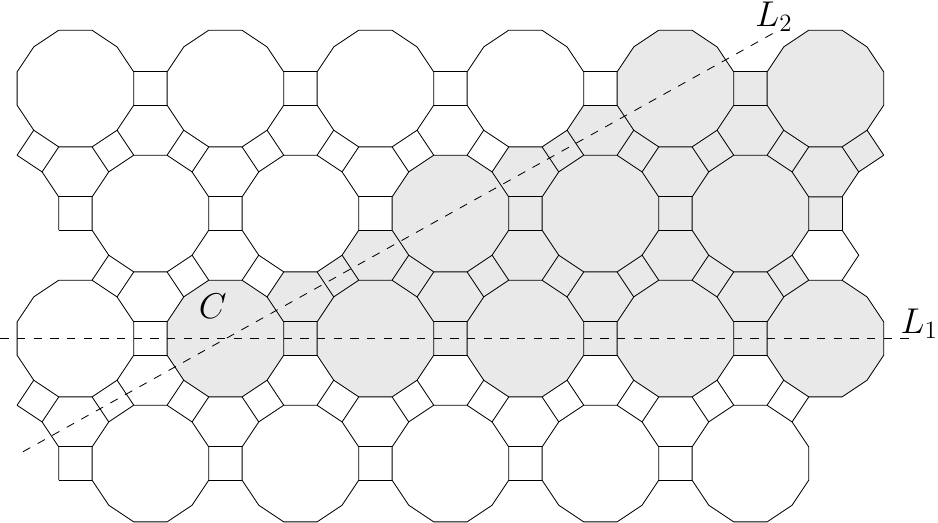}
	\caption{A part of a Coxeter flat and a Coxeter sector (shaded).}
	\label{f:flat_Coxeter}
\end{figure}

\begin{definition}
Let $F$ be a Coxeter flat and choose two Coxeter lines $L_1,L_2$ such that they intersect in a $2$--cell $C\in F$ (such Coxeter lines exists by the discussion in the previous paragraph). For $i=1,2$, let $L'_i\subset L_i$ be a Coxeter ray starting at $C$. Then a \emph{Coxeter sector} is defined to be the region in $F$ bounded by $L'_1$ and $L'_2$ (including $L'_1$ and $L'_2$); see Figure~\ref{f:flat_Coxeter}. $L'_1$ and $L'_2$ are called the \emph{boundary rays} of the Coxeter sector.
\end{definition}

\begin{definition}
	\label{def:plain sector}
Let $Q$ be a quarter plane tiled by unit squares in a standard way. A \emph{plain sector} is a locally injective cellular map $f\colon Q\to \Xa_\Gamma$ such that $f$ restricted to the two \emph{boundary rays} of $Q$ gives two plain rays.
\end{definition}

For each plain sector, there is a full subgraph $\Gamma'\subset\Gamma$ such that $Q$ is contained in a copy of $\Xa_{\Gamma'}$ inside $\Xa_\Gamma$ and $A_{\Gamma'}$ is a right-angled Artin group. Note that both $Q\to \Xa_{\Gamma'}$ and $\Xa_{\Gamma'}\to \Xa_\Gamma$ are injective quasi-isometric embeddings (the first one follows from the fact that the Salvetti complexes of right-angled Artin groups are non-positively curved, and the second one follows from \cite[Theorem 1.2]{charney2014convexity}), thus $f$ is an injective quasi-isometric embedding.

\begin{definition}
	\label{def:DCH}
	A \emph{diamond chromatic half-flat (DCH)} is a locally injective cellular map $f\colon  U=\cup_{i=1}^{\infty} L_i\to \Xa_\Gamma$ such that 
	\begin{enumerate}
		\item $f|_{L_i}$ is a diamond line for each $i$;
		\item $L_i\cap L_{i+1}$ is a boundary line of both $L_i$ and $L_{i+1}$, moreover, $L_i\cap L_j=\emptyset$ for $|i-j|\ge 2$;
		\item there does not exist $i_0\ge 1$ such that $f(\cup_{i=i_0}^{\infty} L_i)$ is contained in a diamond-plain flat.
	\end{enumerate}
The \emph{boundary line} of this DCH is defined to be the diamond line $f|_{L_1}$. 
\end{definition}

\begin{figure}[h!]
	\centering
	\includegraphics[width=0.6\textwidth]{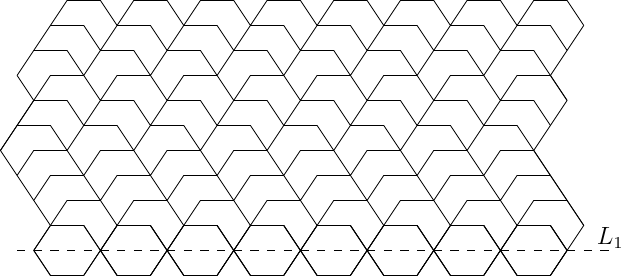}
	\caption{A part of a diamond chromatic half-flat (DCH).}
	\label{f:flat_DCH}
\end{figure}

By using the CAT(0) cube complex $\widetilde{K}_n$ in Lemma~\ref{lem:diamond embedding}, one readily deduces that each DCH is embedded and quasi-isometrically embedded.

\begin{definition}
	\label{def:CCHI}
	A \emph{Coxeter chromatic half-flat (CCH)} of type I is a locally injective cellular map $f\colon  U=\cup_{i=1}^{\infty} L_i\to \Xa_\Gamma$ such that
	\begin{enumerate}
		\item $f|_{L_i}$ is a thickened Coxeter line or a Coxeter line;
		\item $L_i\cap L_j=\emptyset$ for $|i-j|\ge 2$;
		\item $L_i\cap L_{i+1}$ is a boundary component of both $L_i$ and $L_{i+1}$, moreover, each vertex in $L_i\cap L_{i+1}$ is either of type O, or of type II (cf.\ Definition~\ref{d:types}, and Table~\ref{t:flat} on page \pageref{t:flat});
		\item there does not exist $i_0$ such that for all $i\ge i_0$, each vertex in $L_i\cap L_{i+1}$ is of type O; and  there does not exist $i_0$ such that for all $i\ge i_0$, each vertex in $L_i\cap L_{i+1}$ is of type II.
	\end{enumerate}
The \emph{boundary line} of this CCH is defined to be the Coxeter line containing $f|_{\partial U}$, where $\partial U$ is the boundary of $U$ in the topological sense.
\end{definition}

\begin{figure}[h!]
	\centering
	\includegraphics[width=0.6\textwidth]{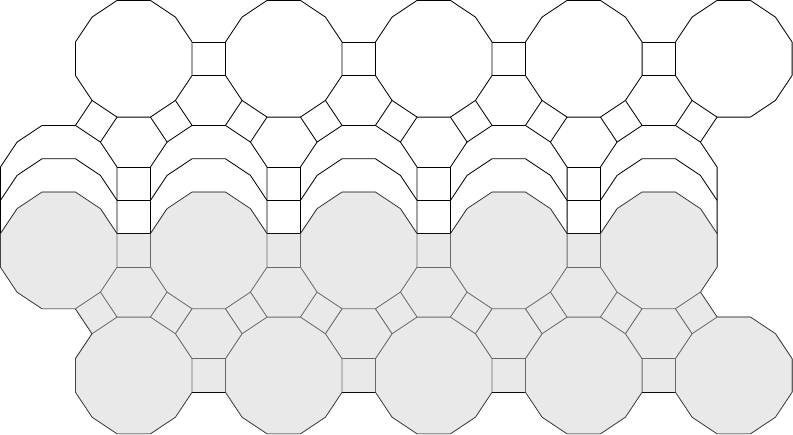}
	\caption{A part of a Coxeter chromatic half-flat (CCH) of type I. A thickened Coxeter line $L_1$ is shaded.}
	\label{f:flat_CCHI}
\end{figure}

\begin{definition}
	\label{def:CCHII}
Let $\Delta$ and $\D_{\Delta}$ be as Definition~\ref{def:wall line}. Let $H$ be the carrier of a halfspace $\h\subset\D_\Delta$ bounded by a wall $\W_1$ of $\D_\Delta$. A \emph{Coxeter chromatic half-flat (CCH)} of type II is a locally injective cellular map $f\colon  H\to \Xa_\Gamma$ satisfying all the following conditions:
	\begin{enumerate}
	\item $f$ preserves the label of edges;
	\item $f$ restricted to the carrier of $\W_1$ is a Coxeter line;
    \item there does not exist a halfspace $\h'\subset \h$ such that the image of the carrier of $\h'$ under $f$ is contained in a Coxeter flat; see Figure~\ref{f:flat_CCHII}.
	\end{enumerate}
The \emph{boundary line} of this CCH is the Coxeter line containing $f|_{\W_1}$. 
\end{definition}

\begin{figure}[h!]
	\centering
	\includegraphics[width=0.9\textwidth]{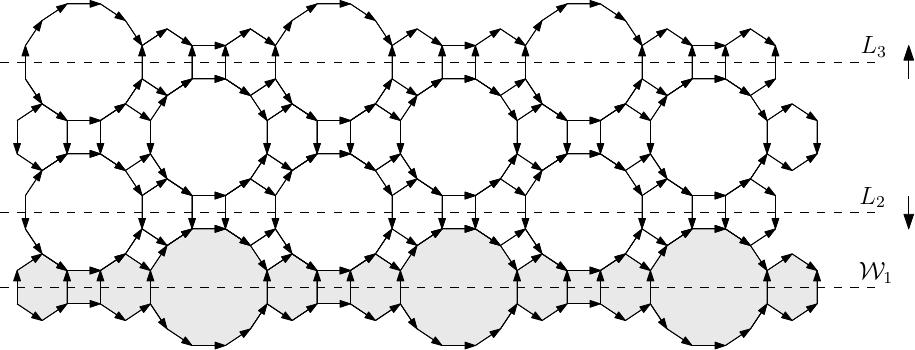}
	\caption{A part of a Coxeter chromatic half-flat (CCH) of type II. The Carrier of $\mathcal W_1$ (shaded) is a Coxeter line. Orientations of edges along $L_2$ and $L_3$ are opposite.}
	\label{f:flat_CCHII}
\end{figure}

As in the discussion after Definition~\ref{def:Coxeter flat}, each CCH of type II is embedded and quasi-isometrically embedded.

\begin{definition}
	\label{def:PCH}
	A \emph{plain chromatic half-flat (PCH)} is a locally injective cellular map $f\colon  U=\cup_{i=1}^{\infty} L_i\to \Xa_\Gamma$ such that
	\begin{enumerate}
		\item $f|_{L_i}$ is a thickened plain line for each $i$;
		\item $L_i\cap L_j=\emptyset$ for $|i-j|\ge 2$;
		\item $L_i\cap L_{i+1}$ is a boundary line of both $L_i$ and $L_{i+1}$;
		\item there does not exist $i_0\ge 1$ such that $f(\cup_{i=i_0}^{\infty} L_i)$ is contained in a block;
		\item there does not exist $i_0\ge 1$ such that $\cup_{i=i_0}^{\infty} L_i$ is made of squares.
	\end{enumerate}
\end{definition}

Note that (5) implies that $f(\partial U)$ is a single-labeled plain line; see Figure~\ref{f:flat_PCH}. The \emph{boundary line} of this PCH is $f|_{\partial U}$.

\begin{figure}[h!]
	\centering
	\includegraphics[width=0.6\textwidth]{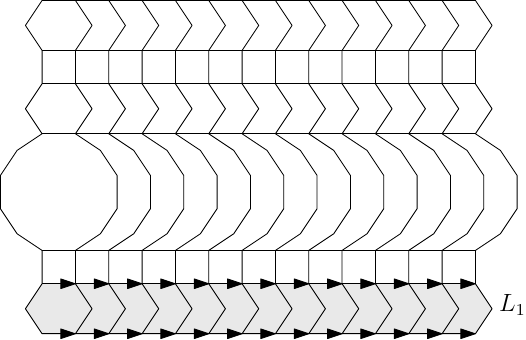}
	\caption{A part of a plain chromatic half-flat (PCH). A thickened plain line $L_1$ is shaded.}
	\label{f:flat_PCH}
\end{figure}

\begin{definition}
\label{d:atomic}
An \emph{atomic sector} is one of the objects among diamond-plain sector, Coxeter-plain sector, Coxeter sector, plain sector, DCH, CCH (of type I or II) and PCH. If $S$ is a DCH, CCH or PCH, then a \emph{boundary ray} of $S$ is a singular ray contained in the boundary line of $S$.
\end{definition}

For further reference, we collect all the types of atomic sectors in Table~\ref{t:atomic} on page \pageref{t:atomic}.

One can build an atomic sector in a local and cell-by-cell fashion. There are plenty of atomic sectors in $\Xa_\Gamma$. Actually, for any vertex in $\Xa_\Gamma$, there is an atomic sector of given type whose image contains this vertex.

\begin{table}[]
	\begin{tabular}{|c|c|}
		\cline{1-2}
		 \begin{tabular}[c]{@{}c@{}}diamond-plain flat \\ and sector \end{tabular} & 
			\begin{tabular}[c]{@{}c@{}}\\ \includegraphics[width=0.39\textwidth]{flat_diamond-plain} \\ \end{tabular}
			   \\ \cline{1-2}
		 \begin{tabular}[c]{@{}c@{}}Coxeter-plain flat \\ and  sector\end{tabular} & 
			\begin{tabular}[c]{@{}c@{}}\\ \includegraphics[width=0.41\textwidth]{flat_Coxeter-plain} \\ \end{tabular}
				\\ \cline{1-2}
		 \begin{tabular}[c]{@{}c@{}}Coxeter flat \\ and sector\end{tabular} & 
			\begin{tabular}[c]{@{}c@{}}\\ \includegraphics[width=0.39\textwidth]{flat_Coxeter}\end{tabular}
				\\ \cline{1-2}
		\begin{tabular}[c]{@{}c@{}}diamond  chromatic \\ half-flat \\ (DCH)\end{tabular} & 
			\begin{tabular}[c]{@{}c@{}}\\ \includegraphics[width=0.39\textwidth]{flat_DCH}\end{tabular}
				\\ \cline{1-2}
		\begin{tabular}[c]{@{}c@{}}Coxeter  chromatic \\ half-flat \\ (CCH) \\ type I\end{tabular} & 
			\begin{tabular}[c]{@{}c@{}}\\ \includegraphics[width=0.35\textwidth]{flat_CCHI}\end{tabular}
				\\ \cline{1-2}
		\begin{tabular}[c]{@{}c@{}}Coxeter  chromatic \\ half-flat \\ (CCH) \\ type II\end{tabular} & 
			\begin{tabular}[c]{@{}c@{}}\\ \includegraphics[width=0.41\textwidth]{flat_CCHII}\end{tabular}
				\\ \cline{1-2}
		\begin{tabular}[c]{@{}c@{}}plain chromatic \\ half-flat \\ (PCH)\end{tabular} & 
			\begin{tabular}[c]{@{}c@{}}\\ \includegraphics[width=0.34\textwidth]{flat_PCH}\end{tabular}
				\\ \cline{1-2}
	\end{tabular}
	\caption{Atomic sectors.}
	\label{t:atomic}
\end{table}

\subsection{Properties of atomic sectors}
\label{subsec:properties of atomic sectors}
For each Artin group $A_\Gamma$, Charney and Davis \cite{CharneyDavis} defined an associated \emph{modified Deligne complex} $D_\Gamma$. Now we recall their construction in the $2$--dimensional case, since it will be used for proving certain properties of atomic sectors. The vertex set $V$ of $D_\Gamma$ is in one to one correspondence with left cosets of the form $gA_{\Gamma'}$, where $g\in A_\Gamma$ and $\Gamma'$ is either the empty-subgraph of $\Gamma$ (in which case $A_{\Gamma'}$ is the identity subgroup), or a vertex of $\Gamma$, or an edge of $\Gamma$. The \emph{rank} of a vertex $gA_{\Gamma'}$ of $V$ is the number of vertices in $\Gamma'$. In other words, rank $2$ vertices of $V$ correspond to blocks of $\Xa_\Gamma$, rank $1$ vertices of $V$ correspond to single-labeled plain lines in $\Xa_\Gamma$, and rank $0$ vertices of $V$ correspond to vertices of $\Xa_\Gamma$.

Note that $V$ has a partial order induced by inclusion of sets. A collection $\{v_i\}_{i=1}^k\subset V$ of vertices spans a $(k-1)$--dimensional simplex if $\{v_i\}_{i=1}^k$ form a chain with respect to the partial order. It is clear that $D_\Gamma$ is a $2$--dimensional simplicial complex, and $A_\Gamma$ acts on $D_\Gamma$ without inversions, i.e.\ if an element of $A_\Gamma$ fixes a simplex of $D_\Gamma$, then it fixes the simplex pointwise.

We endow $D_\Gamma$ with a piecewise Euclidean metric such that each triangle $\triangle(g_1,g_2A_s,g_3A_{\overline{st}})$ is a Euclidean triangle with angle $\pi/2$ at $g_2A_s$ and angle $\frac{\pi}{2n}$ at $g_3A_{\overline{st}}$ with $n$ being the label of the edge $\overline{st}$ of $\Gamma$. By \cite[Proposition 4.4.5]{CharneyDavis}, $D_\Gamma$ is CAT(0) with such metric. As being observed in \cite[Lemma 6]{MR2174269}, the action $A_\Gamma\act D_\Gamma$ is semisimple.

Let $X^b_\Gamma$ be the barycentric subdivision of $\Xa_\Gamma$. Then there is a simplicial map $\pi\colon X^b_\Gamma\to D_\Gamma$ which maps the center of each $2$--cell (resp.\ $1$--cell) in $\Xa_\Gamma$ to the vertex in $D_\Gamma$ representing the block (resp.\ single-labeled plain line) of $\Xa_\Gamma$ that contains this $2$--cell (resp.\ $1$--cell), and maps vertices of $\Xa_\Gamma$ to the corresponding rank $0$ vertices in $D_\Gamma$.

\begin{lemma}
	\label{lem:intersetion of blocks}
If two different blocks $B$ and $B'$ of $\Xa_\Gamma$ satisfy that $B\cap B'$ contains an edge, then $B\cap B'$ is a single-labeled plain line.
\end{lemma}

\begin{proof}
This follows from a more general fact by van der Lek \cite{lek} that for any two full subgraphs $\Gamma_1$ and $\Gamma_2$ of $\Gamma$, we have $A_{\Gamma_1\cap\Gamma_2}=A_{\Gamma_1}\cap A_{\Gamma_2}$.
\end{proof}

\begin{lemma}
	\label{lem:atomic embedding}
Each atomic sector is embedded.
\end{lemma}

\begin{proof}
It remains to prove the lemma for PCH and CCH of type I.
\medskip

\noindent	
\emph{(PCH case.) }Let $f\colon U=\cup_{i=1}^{\infty}L_i\to \Xa_\Gamma$ be as in Definition~\ref{def:PCH}. We group those consecutive $L_i$'s that are mapped by $f$ to the same block to form a new decomposition $U=\cup_{i=1}^{\infty} L'_i$ so that the block $B_i$ containing $f(L'_i)$ satisfies $B_i\neq B_{i+1}$ for any $i\ge 1$. Each $B_i$ gives rise to a rank $2$ vertex in $D_\Gamma$, which we denote by $v_i$, and $B_i\cap B_{i+1}$ gives rise to a rank $1$ vertex in $D_\Gamma$ (cf.\ Lemma~\ref{lem:intersetion of blocks}), which we denote by $w_i$. Then for each $i\ge 1$, $w_i$ is adjacent to $v_i$ and $v_{i+1}$ in $D_\Gamma$. Let $P\to  D_\Gamma$ be the edge path starting from $v_1$ and traveling through $w_1,v_2,w_2,v_3,\ldots$. Note that $v_i\neq v_{i+1}$ (since $B_i\neq B_{i+1}$) and $w_i\neq w_{i+1}$ (since $f(L'_{i+1})$ is embedded in $B_{i+1}$). Thus $P\to D_\Gamma$ is locally injective. 

Let $g\in G$ be the generator of the stabilizer of $w_1$. Let $F_g$ be the fixed point set of $g$. Crisp \cite{MR2174269} observed that $F_g$ is a tree. We repeat his argument here. Note that $F_g$ is a convex subcomplex of $D_\Gamma$. However, $g$ never fixes rank $0$ vertices, so $F_g$ lies in the part of the $1$--skeleton of $D_\Gamma$ which is spanned by rank $1$ and rank $2$ vertices. Thus $F_g$ is a tree. 

Note that the edge path $P$ is actually contained in $F_g$. Since $P$ is locally injective, $P\to F_g$ is an embedding. Now we pick $i,j$ with $|i-j|\ge 2$. Then the segment of $P$ from $v_i$ from $v_j$ is a CAT(0) geodesic segment in $D_\Gamma$ made of $\ge 4$ edges thus, by the way we metrize $D_\Gamma$, there does not exist rank $0$ vertices which are adjacent to both $v_i$ and $v_j$. Also there are no rank $1$ vertices adjacent to both $v_i$ and $v_j$, since any such vertex has to be contained in $F_g$. Let $\pi\colon X^b_\Gamma\to D_\Gamma$ be the simplicial map defined as before. Note that each vertex in $\pi(B_i)$ is adjacent to $v_i$. Thus $B_i\cap B_j=\emptyset$ whenever $|i-j|\ge 2$. From this, Lemma~\ref{lem:intersetion of blocks}, and the fact that $f_{L'_i}$ is injective for each $i$ we deduce that $f$ is injective.
\medskip

\noindent	
\emph{(CCH of type I case.) }
Now let $f\colon U=\cup_{i=1}^{\infty}L_i\to \Xa_\Gamma$ be as Definition~\ref{def:CCHI}. \emph{Walls} in $U$ are defined to be the walls in the domains of Coxeter lines contained in $U$. If $\W\subset U$ is a wall, then $\pi(f(\W))$ is a geodesic line in the $1$--skeleton of $D_\Gamma$. We re-decompose $U$ as $U=\cup_{i=1}^{\infty} L'_i$ such that 
\begin{enumerate}
	\item each $L'_i$ is bounded by two walls;
	\item either each vertex of $U$ in $L'_i$ is of type III (we call such $L'_i$ \emph{non-degenerate}), or each vertex of $U$ in $L'_i$ is of type O or type II (we call such $L'_i$ \emph{degenerate});
	\item $\{L'_i\}_{i=1}^{\infty}$ alternates between non-degenerate and degenerate ones.
\end{enumerate}
If $L'_i$ is non-degenerate, then $\pi\circ f|_{L'_i}$ is locally injective, hence injective by CAT(0) geometry. If $L'_i$ is degenerate, then walls of $L'_i$ are mapped by $\pi\circ f$ to the same geodesic line in $D_\Gamma$. Since each Coxeter line is periodic, pick an element $g\in A_\Gamma$ acting by translation on a Coxeter line $L\subset f(U)$. Note that $f(U)$ is invariant under the action of $g$. Moreover, let $\ell_g$ be the $\pi$--image of the wall in $L$, then $g$ acts by translation on $\ell_g$. Let $M_g$ be the collection of points in $D_\Gamma$ such that $d(x,gx)$ attains its minimum. Then $M_g$ admits a splitting $M_g=\ell_g\times T$ by \cite[Theorem II.6.8]{BridsonHaefliger1999}. Since $D_\Gamma$ is $2$--dimensional, $T$ is a tree. Note that $\pi\circ f(U)\subset M_g$. For each $i$, let $s_i\subset T$ be the orthogonal projection of $\pi\circ f(L'_i)$ onto the $T$--factor. If $L'_i$ is non-degenerate, then $s_i$ is a segment with length $>0$. If $L'_i$ is degenerate, then $s_i$ is a star shaped subset, i.e.\ $s_i$ is an union of segments along a vertex. We assume without loss of generality that $L'_1$ is non-degenerate. Let $P\to T$ be the concatenation of $s_{2i+1}$ for $i\ge 0$ (note that the endpoint of $s_{2i-1}$ is the starting point of $s_{2i+1}$). By the argument in the PCH case, we know $f$ restricted to the carrier of $L'_i$ in $U$ is injective for $i$ even. This implies that $P$ is locally injective at $s_i\cap s_{i+2}$ (for $i$ odd). Thus $P$ is an embedding. This and $f|_{L'_i}$ being an embedding for each $i$ imply that $f$ is an embedding.
\end{proof}

\begin{lemma}
	\label{lem:atomic intersection}
Let $S_1$ and $S_2$ be two atomic sectors. If there is a sequence of $2$--dimensional discs $\{D_i\}_{i=1}^{\infty}\subset S_1\cap S_2$ such that $D_i$ contains a ball of radius $i$ in $S_1$ (with respect to the path metric on the $1$--skeleton of $S_1$), then $d_H(S_1\cap S_2,S_1)<\infty$ and $d_H(S_1\cap S_2,S_2)<\infty$. Moreover, $S_1$, $S_2$ and $S_1\cap S_2$ are atomic sectors of the same type.
\end{lemma}

\begin{proof}
The lemma follows from a case by case inspection. Suppose $S_1$ and $S_2$ satisfy the assumption of the lemma. Let $D\subset S_1\cap S_2$ be a large disc. We only discuss the case when $S_1$ is a CCH of type I. Then other cases are either similar, or much simpler. 

Let $S_1=\cup_{i=1}^{\infty}L'_i$, $M_{g_1}=\ell_{g_1}\times T\subset D_\Gamma$ and the path $P\subset T$ be as in Lemma~\ref{lem:atomic embedding}, where $g_1\in A_\Gamma$ is an element acting on $S_1$. Since a diamond-plain sector or a DCH is contained in a block, so $S_2$ can not be one of them. Also $S_2$ can not be a plain sector since a plain sector is made of squares. If $S_2$ is a PCH, then $D$ is contained in the carrier of $L'_i$ (in $S_1$) for some $i$ even. Let $g_2\subset A_\Gamma$ be the element acting on $S_2$ and we define the tree $F_{g_2}$ as in Lemma~\ref{lem:atomic embedding}. The existence of $D$ implies that a sub-segment of $F_{g_2}$ goes along the $\ell_g$ direction of $M_g$. The convexity of $F_{g_2}$ and $M_{g_1}$ implies that $F_{g_2}\cap M_{g_2}\subset \ell_g\times\{t\}$ for some $t\in T$. Since $\pi(S_1)\subset M_{g_2}$ and each vertex of $\pi(S_2)$ is either contained in $F_{g_2}$ or adjacent to a vertex in $F_{g_2}$, we know that $S_1\cap S_2$ is contained in the carrier of $L'_i$ inside $S_1$. This rules out the possibility of $S_2$ being a PCH. The case of $S_2$ being a Coxeter-plain sector can be ruled out in a similar way. If $S_2$ is a Coxeter sector, then $D$ is contained in the carrier of $L'_i$ (in $S_1$) for some $i$ odd and $\pi(S_2)\subset M_{g_1}$. Note that $P$ and the projection of $\pi(S_2)$ on $T$ are two rays which diverge from each other at some point, thus $S_1\cap S_2$ is contained in the carrier of $L'_i$ inside $S_1$ for some odd $i$ and $S_2$ can not be a Coxeter sector.

Now suppose that $S_2$ is a CCH of type I or II and that $S_2$ is invariant under $g_2\in A_\Gamma$. The existence of $D$ implies that there is a Coxeter line $L\subset S_2$ such that $L\cap S_1$ contains a Coxeter segment $U$ (cf. Definition~\ref{def:wall line}) which is sufficiently long. Thus we have the following two cases:
\begin{enumerate}
	\item if $U$ is contained in a Coxeter line of $S_1$, then we know $L$ is a Coxeter line of $S_1$;
	\item if $U$ is not contained in a Coxeter line of $S_1$, then $N$ is contained in a thickened Coxeter line of $S_1$, hence $L$ and a Coxeter line of $S_1$ span a Coxeter flat in $\Xa_\Gamma$.
\end{enumerate}
Let $\ell\subset D_\Gamma$ be the $\pi$--image of the wall in $L$. In both cases (1) and (2), we have $\ell\subset M_{g_1}$. If $\ell$ and $\ell_{g_1}$ are not parallel, then $P$ and the projection of $\ell$ on $T$ are two rays which diverge from each other at some point (since $\ell$ is periodic), which implies that $S_1\cap S_2$ is contained in the carrier of $L'_i$ inside $S_1$ for some odd $i$. Thus $\ell$ and $\ell_{g_1}$ are parallel. Then $\pi(S_2)\subset M_{g_1}$. The projections of $\pi(S_1)$ and $\pi(S_2)$ must contain a common geodesic ray, otherwise we can reach a contradiction as before. Now one readily deduces that $d_H(S_1\cap S_2,S_1)<\infty$ and $d_H(S_1\cap S_2,S_2)<\infty$.
\end{proof}

Now we recall the notion of locally finite homology. We use $\mathbb Z/2\mathbb Z$ coefficients throughout this paper. Recall that for a topological space $X$ we can consider \emph{locally finite chains} in~$X$, which are formal sums $\Sigma_{\lambda\in\Lambda}a_{\lambda}\sigma_{\lambda}$ where $a_{\lambda}\in\mathbb Z/2\mathbb Z$, $\sigma_{\lambda}$ are singular simplices, and any compact set in $X$ intersects the images of only finitely many $\sigma_{\lambda}$ with $a_{\lambda}\neq 0$. This
gives rise to \textit{locally finite homology} of $X$, denoted by $H^{\mathrm{lf}}_{*}(X)$. We define the relative locally finite homology $H^{\mathrm{lf}}_{*}(X,Y)$ for a subset $Y\subset X$ in a similar way. For a point $z\in Z\setminus Y$, there is a homomorphism $i\colon H^{\mathrm{lf}}_{k}(Z,Y)\to H^{\mathrm{lf}}_{k}(Z,Z\setminus\{z\})\cong H_{k}(Z,Z\setminus\{z\})$ induced by the inclusion of pairs $(Z,Y)\to (Z,Z-\{z\})$. For $[\sigma]\in H^{\mathrm{lf}}_{k}(Z,Y)$, we define the \emph{support set} of $[\sigma]$, denoted $S_{[\sigma],Z,Y}$, to be $\{z\in Z\setminus Y\mid i_{\ast}[\sigma]\neq \textmd{Id}\}$. We will also use $S_{[\sigma]}$ to denote the support set if the underlying spaces $Z$ and $Y$ are clear.

\begin{lemma}
	\label{lem:atomic qi embedding}
Each atomic sector is quasi-isometrically embedded.
\end{lemma}

\begin{proof}
The cases of diamond-plain sector, Coxeter-plain sector and Coxeter sector follow from the fact that any rank $2$ abelian subgroup of $A_\Gamma$ is quasi-isometrically embedded (\cite[\lemsevenseven]{Artinmetric}). The case of plain sector has already been discussed after Definition~\ref{def:plain sector}. It remains to consider chromatic half-flats. Let $U\subset \Xa_\Gamma$ be a PCH and let $\{\ell_i\}_{i=1}^{\infty}$ be the collection of all plain lines in $U$ parallel to $\partial U$ such that $\ell_1=\partial U$ and $\ell_i$ is between $\ell_{i-1}$ and $\ell_{i+1}$. Since each $\ell_i$ is quasi-isometrically embedded with constants independent of $i$, it suffices to show that there is a constant $M$ such that $d_{A_\Gamma}(\ell_i,\ell_j)\ge M|i-j|$, where $d_{A_\Gamma}$ denotes the distance in the $1$--skeleton of $\Xa_\Gamma$. Suppose the contrary holds. Then for any given $\epsilon>0$, there exists $\ell_i,\ell_j$ such that $d_{A_\Gamma}(\ell_i,\ell_j)< \epsilon|i-j|$. Suppose $L=d_{A_\Gamma}(\ell_i,\ell_j)$. Then $|i-j|>L/\epsilon$. Let $P_0$ be a shortest edge path from a vertex $v_0\in \ell_i$ to a vertex $w_0\in \ell_j$. Let $g\in A_\Gamma$ be an element which translates $\ell_i$ by distance $L$. Let $C_1=v_0\to gv_0\to gw_0\to w_0\to v_0$ be the cycle traveling along $\ell_i$, $gP_0$, $\ell_j$ and $P_0$. Since $C_1$ has $\le 4L$ edges and $A_\Gamma$ has quadratic Dehn function (by Corollary~\ref{cor:quadratic} and Theorem~\ref{thm:msystolic}), there is a cellular $2$--chain $\sigma$ such that $\partial \sigma =C_1$, and $\sigma$ has at most $CL^2$ cells, where $C$ only depends on $A_\Gamma$. Since $\ell_i$ is quasi-isometrically embedded, $\alpha=\sum_{k=-\infty}^{\infty} g^k\sigma$ gives rise to a locally finite relative homology class $[\alpha]\in H^{\mathrm{lf}}_{2}(\Xa_\Gamma,\ell_i\cup\ell_j)$. On the other hand, let $U_{ij}$ be the region of $U$ bounded by $\ell_i$ and $\ell_j$. The fundamental class of $U_{ij}$ gives another locally finite cellular chain $\beta$ such that $\partial\beta=\partial \alpha$. Then $[\beta]\in H^{\mathrm{lf}}_{2}(\Xa_\Gamma,\ell_i\cup\ell_j)$ and $S_{[\beta]}=U_{ij}\setminus(\ell_i\cup\ell_j)$. Note that $\alpha\cup\beta$ is a periodic \textquotedblleft infinite cylinder\textquotedblright. Recall that $\Xa_\Gamma$ is contractible by By \cite[Theorem B]{CharneyDavis} and \cite[Corollary 1.4.2]{CharneyDavis}. Then $\Xa_\Gamma$ is also uniformly contractible since it is cocompact. Thus we can fill in the periodic infinite cylinder to obtain a locally finite singular $3$--chain $\gamma$ such that $\partial\gamma=\alpha-\beta$. Thus $[\alpha]=[\beta]$ and $S_{[\alpha]}=S_{[\beta]}=U_{ij}\setminus(\ell_i\cup\ell_j)$. It is clear that $S_{[\alpha]}\subset \im \alpha$, thus $U_{ij}\subset \im\alpha$. Note that both $\im\alpha$ and $U_{ij}$ are invariant under the action of the subgroup $\langle g\rangle\cong \mathbb Z$. Moreover, $\im\alpha/\langle g\rangle$ has at most $CL^2$ $2$--cells, but $U_{ij}/\langle g\rangle$ has $L^2/\epsilon$ $2$--cells (since $|i-j|>L/\epsilon$). This will be contradictory to $U_{ij}\subset \im\alpha$ if we take $\epsilon$ small enough. Similarly we deal with other chromatic half-flats, since they all have one periodic direction (the cases of DCH and CCH of type II have already been treated differently in Section~\ref{subsec:atomic sector}).
\end{proof}

\subsection{Completions of atomic sectors}
\label{subsec:completion}
A subset $A$ is \emph{coarsely contained} in another subset $B$ of a metric space if $A$ is contained in the $R$--neighborhood of $B$ for some $R$. We say $A$ is the \emph{coarse intersection} of a family of subsets $\{B_i\}_{i\in I}$ if there exists $R_0$ such that for all $R\ge R_0$, the Hausdorff distance between $A$ and $\cap_{i\in I} N_R(B_i)$ is finite.

In the rest of this subsection, we define the \emph{completion} of an atomic sector $S$, which (roughly speaking) is a ``smallest possible subset'' that is the coarse intersection of 2-dimensional quasiflats and coarsely contains $S$.
\begin{lemma}
	\label{lem:atomic is intersection of quasiflats}
Let $S$ be an atomic sector which is not a plain sector or a diamond-plain sector. Then $S$ is a coarse intersection of finitely many $2$--dimensional quasiflats.
\end{lemma}

\begin{proof}
Let $U=\cup_{i=1}^{\infty}L_i$ be a CCH of type I (see Table~\ref{t:atomic} on page \pageref{t:atomic}). We assume without loss of generality that $L_1$ is a Coxeter line and $L_2$ is a thickened Coxeter line. Let $F$ be the Coxeter-plain flat containing $L_1$. Then the line $L_1\cap L_2$ divides $F$ into two half-spaces, which we denote by $U_1$ and $U_2$. Let $U_3=\cup_{i=2}^{\infty}L_i$. We glue $U_1$, $U_2$ and $U_3$ along $L_1\cap L_2$ to form a triplane $V$. Note that there is a locally injective map $V\to\Xa_\Gamma$. By the same argument as in the proof of Lemma~\ref{lem:atomic embedding} and Lemma~\ref{lem:atomic qi embedding}, we know that $U_i\cup U_j$ is embedded and quasi-isometrically embedded for any $1\le i\neq j\le 3$. Thus $U$ is the coarse intersection of the quasiflat $U_3\cup U_1$ and the quasiflat $U_3\cup U_2$. The case of CCH of type II is similar. If $U$ is a PCH, we can assume $L_1$ and $L_2$ are not in the same block. Let $F$ be the diamond flat containing $L_1$. Then we can form a triplane as before and find three quasiflats using the argument of Lemma~\ref{lem:atomic embedding} and Lemma~\ref{lem:atomic qi embedding}. To deal with DCH, note that each DCH is contained in a large block. Then we can use the cube complex $\tilde K_n$ in Lemma~\ref{lem:diamond embedding} to conclude the proof.

If $S$ is a Coxeter sector, let $F$ be the Coxeter flat containing $S$. Let $L_1$ and $L_2$ be two Coxeter lines in $F$ that contains the two boundary rays of $S$. Then there are half-flats $H_1$ and $H_2$ of $F$ bounded by $L_1$ and $L_2$ respectively such that $S=H_1\cap H_2$. Since $F$ is quasi-isometrically embedded, $S$ is the coarse intersection of $H_1$ and $H_2$. However, each $H_i$ is the coarse intersection of two quasiflats by the argument in the previous paragraph. Thus we are done in the case of Coxeter sector. The case of Coxeter-plain sector is similar.
\end{proof}

We define the \emph{completion} of an atomic sector which is not a plain sector or a diamond-plain sector to be itself. Before discussing completions of diamond-plain sectors and plain sectors, we need the following observation.

For two vertices $a,b\subset\Gamma$, we define $a\sim b$ if $a$ and $b$ are adjacent in $\Gamma$ along an edge labeled by an odd number. If $a\sim b$, then for each singled-labeled plain line $\ell_a$ in $\Xa_\Gamma$ labeled by $a$, there is a thickened plain line $L$ such that 
\begin{enumerate}
	\item one boundary component of $L$ is $\ell_a$;
	\item another boundary component of $L$ is a single-labeled plain line labeled by $b$.
\end{enumerate}
Now we consider the equivalence relation on the vertex set of $\Gamma$ generated by $\sim$. Each equivalence class is called an \emph{odd component}. Note that a plain line whose edges are labeled by an element in an odd component can be \textquotedblleft parallelly transported\textquotedblright\ to a plain line with its edges labeled by another element in the same odd component via finitely many thickened plain lines.

If $S$ is a diamond-plain sector, let $F$ be the diamond flat containing $S$. Then $S$ is the intersection of two half-flats $H_1,H_2$ of $F$ such that $H_1$ is bounded by a plain line $L_1$ and $H_2$ is bounded by a diamond line $L_2$. Note that $H_2$ is the coarse intersection of two 2-dimensional quasiflats. However, this may not be true for $H_1$ (e.g.\ when $\Gamma$ is an edge). Let $a\in\Gamma$ be the vertex corresponding to the label of edges in $L_1$. If every vertex in the odd component containing $a$ is a leaf in $\Gamma$ (recall that a leaf in a graph is a vertex which is only adjacent to one another vertex), then we define the \emph{completion} of $S$ to be $H_2$. 
Observe that in this case either $a$ is a leaf in $\Gamma$ contained in an edge with an even label, or the component of
$a$ in $\Gamma$ is an edge (with an odd label).
If there exists a vertex $c$ in the odd component containing $a$ which is not a leaf, then the \emph{completion} of $S$ is defined to be itself. In this case, we can form a triplane containing $H_1$, hence $H_1$ is the coarse intersection of two quasiflats and $S$ is the coarse intersection of finitely many quasiflats.

Now suppose $S$ is a plain sector. We write $S=r_1\times r_2$ where $r_1$ and $r_2$ are two boundary rays of $S$. For $i=1,2$, let $V_i\subset\Gamma$ be the collection of labels of edges in $r_i$. Let $V^{\perp}_i$ be the collection of vertices in $\Gamma$ which are adjacent to each vertex in $V_i$ along an edge labeled by $2$. The set $V_i$ is \emph{good} if either $V^{\perp}_i$ has $\ge 2$ vertices, or $V_i$ is a singleton and there exists a non-leaf vertex in the odd component containing $V_i$. Since $V_1$ and $V_2$ form a join subgraph $\Gamma'\subset\Gamma$, the latter case is equivalent to $V_i$ being a singleton and not a leaf. 

If both $V_1$ and $V_2$ are good, then the \emph{completion} of $S$ is $S$ itself. In this case, $S$ is the coarse intersection of finitely many $2$--dimensional quasiflats. Now we prove this statement. Note that $A_{\Gamma'}$ is a right-angled Artin group. We assume $S$ is contained in a copy of $\Xa_{\Gamma'}$ in $\Xa_{\Gamma}$. Then there is a flat $F=\ell_1\times \ell_2$ in $\Xa_{\Gamma'}$ such that $\ell_i$ is a geodesic line extending $r_i$ (such extension may not be unique). We know $S=H_1\cap H_2$ where $H_i$ is a half-flat in $F$ bounded by $\ell_i$. Since $F$ is isometrically embedded in $\Xa_{\Gamma'}$ and $\Xa_{\Gamma'}$ is quasi-isometrically embedded in $\Xa_{\Gamma}$ (\cite[Theorem 1.2]{charney2014convexity}), we know $S$ is the coarse intersection of $H_1$ and $H_2$. Note that the labels of edges in $\ell_i$ are in $V_i$, then we deduce from the fact that $V_1$ and $V_2$ are good that each of $H_1$ and $H_2$ ia a coarse intersection of two quasiflats. 

If only one of $V_1, V_2$ is good, then $S$ is contained in a canonical half-flat $H$ such that $H$ is made of squares and $H$ is a coarse intersection of two quasiflats. Note that the boundary of $H$ is a single-labeled plain line. $H$ is defined to be the completion of $S$. 

If both $V_1$ and $V_2$ are not good, then $V_1$ and $V_2$ are singletons. We extend $r_i$ to a single-labeled plain line $\ell_i$ and define the \emph{completion} of $S$ to be the flat of the form $\ell_1\times\ell_2$.

\begin{corollary}
	\label{cor:completion of atomic}
For any atomic sector $S$, the completion $\bar S$ of $S$ is a coarse intersection of finitely many $2$--dimensional quasiflats.
\end{corollary}

\section{Half Coxeter regions in $Q'$ and $Q_R$}
\label{sec:half coxeter regions}
Let $x\in Q_R$ be a vertex of type III (see Table~\ref{t:flat} on page \pageref{t:flat}). Let $K_x\subset Q'$ be the convex hull of fake vertices of $Q'$ that are dual to $2$--cells in $\St(x,Q_R)$. Such $K_x$ is called a \emph{chamber} or an \emph{$x$--chamber}. Note that $K_x$ is either a triangle or a square, and define \emph{edges} and \emph{vertices} of $K_x$ with respect to the standard cell structure on a triangle or a square. $K_x$ is isometric to the standard fundamental domain of the action of a Euclidean Coxeter group on $\mathbb E^2$ (the defining graph of this Coxeter group is the support of $x$).  Since $K_x\subset \St(x,Q')$, the intersection of two different chambers is either empty, or a point, or an edge of both of the chambers. Each chamber $K_x$ has a \emph{support} inherited from $x$.

Let $\C'\subset Q'$ be the union of $x$--chambers with $x$ ranging over vertices of type III in $Q_R$ (it is possible that $\C'=\emptyset$). A \emph{Coxeter region} of $Q'$ is a connected component of $\C'$, and a \emph{Coxeter region} of $Q_R$ is the union of $\St(x,Q_R)$ such that the associated $x$--chamber is in a Coxeter region of $Q'$.
 
\begin{lemma}
	\label{lem:local convexity}
Let $C\subset Q_R$ be an interior $2$--cell, i.e.\ $C$ is disjoint from the boundary of $Q_R$. Let $V$ be the collection of type III vertices on $\partial C$. Suppose that $V\neq\emptyset$ and that $\partial C$ has $2n$ edges. Then all vertices of $V$ have the same support, and one of the following three possibilities happens:
\begin{enumerate}
	\item $V$ is one point;
	\item $V$ is the vertex set of an arc in $\partial C$ made of $n-1$ edges;
	\item each vertex in $\partial C$ is in $V$.
\end{enumerate}
\end{lemma}

\begin{proof}
Suppose $V$ is not a singleton. Let $v_1,v_2$ be two vertices of $V$ and let $\omega\subset\partial C$ a path of shortest combinatorial distance from $v_1$ to $v_2$. We claim each vertex of $\omega$ is in $V$. Let $v_3$ be a vertex in $\omega$ that is adjacent to $v_2$. Since $v_2$ is of type III, let $C,C'$ be the two $2$--cells in $\St(v_2,Q_R)$ that contain $v_3$. By Lemma~\ref{lem:star homeo} and Lemma~\ref{lem:real3} (2), $v_3$ is not of type O. Moreover, since $C$ and $C'$ are mapped to different blocks by $q$, $v_3$ is not of type I. If $v_3$ is of type II, let $C_1$ and $C'_1$ be other cells in $\St(v_3,Q_R)$. Since $C\cap C'$ is one edge, $C\cap C_1$ has $n-1$ edges by Lemma~\ref{lem:star homeo} and Lemma~\ref{lem:real2} (3). Thus $v_1\in C\cap C_1$. By Lemma~\ref{lem:real2} (3), $q(C)$ and $q(C_1)$ are in the same block. This contradicts that $v_1$ is of type III since different cells containing a vertex of type III have $q$--images in different blocks. Thus $v_3$ must be of type III. Moreover, $v_3$ and $v_2$ have the same support by Lemma~\ref{lem:same tag}. Now the claim follows by repeating this argument for other vertices in $\omega$.

It follows from the above claim that either $V$ is the collection of all vertices in $\partial C$, or $V$ is the collection of vertices in an arc $\omega$ of $\partial C$ that has $\le n-1$ edges. Suppose the number of edges in $\omega$ is $\ge 1$ and is $<n-1$. Let $v_1$ and $v_2$ be two endpoints of $\omega$. For $i=1,2$, let $w_i\in(\partial C\setminus\omega)$ be a vertex that is adjacent to $v_i$. Since $w_i$ is not of type III, by the discussion in the previous paragraph, $w_i$ is of type II. Moreover, there is a cell $C_i$ of $Q_R$ such that $C_i\cap C$ has $n-1$ edges and $(C_i\cap C)\cap \overline{v_iw_i}=w_i$. Since $\omega$ has at least one edge, $(C_1\cap C)\cap(C_2\cap C)$ contains at least one edge. Since  $\omega$ has $<n-1$ edges, $(C_1\cap C)\neq (C_2\cap C)$, in particular $C_1\neq C_2$. This contradicts that $Q_R$ is planar. Thus $\omega$ is either a point, or an arc with $n-1$ edges.
\end{proof}

The following is a consequence of Lemma~\ref{lem:local convexity}.

\begin{corollary}
	\label{cor:local classification}
Let $\C'\subset Q'$ be a Coxeter region and let $x\in \C'$ be a fake vertex such that the $2$--cell of $Q_R$ associated with $x$ is interior. Let $D_\epsilon$ be a disc of radius $\epsilon$ around $x$ in $Q'$. Then for $\epsilon$ small enough, exactly one of the following possibilities happens:
\begin{enumerate}
	\item $\C'\cap D_{\epsilon}=K\cap D_{\epsilon}$ where $K$ is a chamber of $\C'$ containing $x$;
	\item $\C'\cap D_{\epsilon}$ is a half disc of $D_{\epsilon}$;
	\item $\C'\cap D_\epsilon =D_\epsilon$.
\end{enumerate}
\end{corollary}

In general, we do not have control of the local structure at each point of $\C'$. This motivates us to pass to a subset of $\C'$ as follows.

\begin{definition}
	\label{def:half Coxeter region}
Let $x\in Q_R$ be a vertex of type III. Suppose there exists an edge $e\subset Q_R$ containing $x$ such that
\begin{enumerate}
	\item the geodesic line $\ell\subset Q'$ that is orthogonal to $e$ (we can view $e$ as an edge in $Q'$) and passes through the midpoint of $e$ satisfies that $d(x_0,\ell)>R$, where $x_0$ is the base point in $Q'$ as in Section~\ref{subsec:new cell structure} and $d$ denotes the CAT(0) distance;
	\item $x$ and $x_0$ are on different sides of $\ell$;
	\item let $H'$ be the half-space of $Q'$ that is bounded by $\ell$ and contains $x$, then $H'\subset Q_R$ and any $2$--cell of $Q_R$ that intersects $H'$ is interior.
\end{enumerate} 
Let $\C'\subset Q'$ be the Coxeter region containing $x$. The \emph{$(x,\ell)$--half Coxeter region} of $Q'$ is the connected component of $\C'\cap H'$ that contains $x$.
\end{definition}

\begin{remark}
There is a constant $L$ such that for any vertex $x\in Q_R$ of type III satisfying $d(x,x_0)>LR$ (we can also view $x$ as a vertex in $Q'$), there exists an edge of $Q_R$ containing $x$ that satisfies the conditions in Definition~\ref{def:half Coxeter region}.
\end{remark}

By Definition~\ref{def:half Coxeter region} (1) and (2), $H'$ is isometric to a half-plane. Let $\h'$ be the $(x,\ell)$--half Coxeter region. Then $\h'$ is locally convex in $H'$ by Lemma~\ref{lem:local convexity} and Definition~\ref{def:half Coxeter region} (3). Since $\h'$ is connected, $\h'$ is actually convex in $H'$. 

\begin{lemma}
	\label{lem:union of chambers}
$\h'$ is a union of chambers with the same support. 
\end{lemma}

\begin{proof}
Let $x\in Q_R$ be as in Definition~\ref{def:half Coxeter region}. Let $K_x$ be the $x$--chamber. Then $K_x\subset \h'$, and $K_x\cap \ell$ is an edge of $K_x$. Let $v_1$ be an endpoint of $K_x\cap \ell$ and let $\ell^-$ be the half-line of $\ell$ with endpoint $v_1$ such that $K_x\cap\h'\nsubseteq \ell^-$. Note that $v_1$ is a fake vertex in $Q'$. Let $\{K_i\}_{i=1}^k$ be the chambers of $\C'$ that contain $v_1$. By Definition~\ref{def:half Coxeter region} (3) and Lemma~\ref{lem:local convexity}, either $K_i\cap \ell$ is an edge of $K_i$, or $K_i\cap \ell=\{v_1\}$. If $(\cup_{i=1}^k K_i)\cap \ell=K_x\cap \ell$, then $\ell^-\cap\h'=\{v_1\}$ by the convexity of $\h'$. If $(\cup_{i=1}^k K_i)\cap \ell\neq K_x\cap \ell$, then there exists $K_i$ such that $K_i\cap\ell$ is an edge of $K_i$ and $K_i\cap\ell\neq K_x\cap \ell$. We can assume $K_i\subset H'$ (if $K_i$ and $K_x$ are on different sides of $\ell$, then we must be in case (3) of Lemma~\ref{lem:local convexity}). Thus $K_i\subset\h'$. Let $v_2$ be the endpoint of $K_i\cap \h'$ that is different from $v_1$. We then repeat previous discussion with $v_1$ replaced by $v_2$. Keep running this process to walk along $\ell\cap\h'$ until one has to stop at an endpoint of $\ell\cap\h'$. Then it follows that a small neighborhood of $\ell\cap\h'$ in $\h'$ is contained in a union of chambers inside $\h'$ such that each chamber either intersects $\ell$ in an edge of this chamber, or intersects $\ell$ in a point. Thus $\h'$ is a union of chambers. Lemma~\ref{lem:local convexity} and the connectedness of $\h'$ imply that all chambers in $\h'$ have the same support.
\end{proof}

By Lemma~\ref{lem:union of chambers}, $\h'$ has a cell structure such that its vertices are fake vertices of $Q'$ contained in $\h'$ and its $2$--cells are chambers with the same support. We define the \emph{$(x,\ell)$--half Coxeter region $\h\subset Q_R$ corresponding to $\h'$} to be the union of $2$--cells of $Q_R$ which are associated with fake vertices in $\h'$. Each edge $e\subset \h'$ has a dual edge in $Q_R$ which intersects $e$ in exactly one point. Thus edges of $\h'$ inherit labels from edges of $Q_R$. 

Let $\D$ be the Davis complex of the Coxeter group $W$ whose defining graph is the support of a chamber in $\h'$. The collections of walls of $\D$ cut $\D$ into another complex $\D'$, which is the dual of $\D$. Each edge of $\D'$ inherits a label from its dual edge in $\D$ (edges of $\D$ are labeled by generators of $W$).

By Lemma~\ref{lem:local convexity}, there exists a label-preserving cellular isometry from the closed star of each vertex of $\h'$ to $\D'$. Since $\h'$ is simply connected (as $\h'$ is convex in $H'$), by Corollary~\ref{cor:local classification} and a developing argument, we can produce a label-preserving cellular local isometry $i'\colon \h'\to \D'$. Note that $i$ is actually an isometric embedding, since both $\h'$ and $\D'$ are CAT(0). By Lemma~\ref{lem:star homeo} (2), there is an induced cellular embedding $i\colon \h\to\D$ such that $i(\h)$ is the union of $2$--cells dual to vertices in $i'(\h')$. We view $\h'$ (resp.\ $\h$) as subcomplexes of $\D'$ (resp.\ $\D$) via $i'$ (resp.\ $i$). A \emph{wall} in $\h'$ (resp.\ $\h$) is the intersection a wall in $\D'$ (resp.\ $\D$) with $\h'$ (resp.\ $\h$). The definition of wall does not depend the choice of the label-preserving map $i'$. We say $\h$ (or $\h'$) is \emph{irreducible} (resp.\ \emph{reducible}) if the associated Coxeter group $W$ has defining graph being a triangle (resp.\ square). 

\section{Orientation of edges of Coxeter regions} 
\label{sec:orientation of edges of Coxeter regions}
Let $x,\ell$ and $H'$ be as in Definition~\ref{def:half Coxeter region} and let $\h$, $\h'$, $\D$ and $\D'$ be as in the previous section.
 
A \emph{border} of $\h'$ is a maximal connected subset of $\partial\h'$ which is convex in $\h'$. A border of $\h'$ is \emph{fake} if this border is contained in $\ell$. Other borders are \emph{real}. If we think of $\h'$ as of a convex subcomplex of $\D'$, then each border $\B'\subset\h'$ is contained in a wall $\W$ of $\D'$. Let $\W^+$ (resp.\ $\W^-$) be the half-space of $\D'$ bounded by $\W$ that contains (resp.\ does not contain) $\h'$. Let $N_{\B'}$ be the carrier of $\B'$ in $\D$ (see the beginning of Section~\ref{sec:quasiflats in MS complex} for definition of carrier). Then $\B'$ corresponds to an \emph{outer border} $\B^o$ and an \emph{inner border} $\B^i$ of $\h$, which are defined to be the maximal subcomplexes of $N_{\B'}$ contained in $\W^-$ and $\W^+$, respectively. $\B^o$ or $\B^i$ is \emph{real} or \emph{fake} if $\B'$ is real or fake. Recall that each edge of $Q_R$ inherits an orientation via $q\colon Q_R\to\Xa_\Gamma$, so does each edge in $\h$. Now we study the orientation of edges along a border of $\h$.

The borders $\B^o$ and $\B^i$ are made of \emph{pieces}, which are maximal subsegments of the border that are contained in a $2$--cell of $N_{\B'}$. Each vertex in a border of $\h$ either belongs to one piece, or belongs to the intersection of two pieces.

\subsection{Orientation along the borders}
\begin{lemma}
	\label{lem:piece orient}
The orientation of edges on the same piece on a real outer border of $\h$ is consistent, i.e.\ there does not exist an interior vertex of a piece such that the orientation is reversed at that vertex. The same is true for real inner border.
\end{lemma}

\begin{proof}
Let $\omega$ be a piece in the lemma. Then there exists one endpoint, say $v$, of $\omega$ such that $v$ belongs to two different cells $C_1,C_2\subset\h$. Since $q$ maps $C_1$ and $C_2$ to cells in different blocks of $\Xa_\Gamma$, we know $v$ is of type II or type III (see Table~\ref{t:flat} on page \pageref{t:flat}). If $v$ is of type III, then we can enlarge $\h$ (hence $\h'$) since $v$ is in a real border, which contradicts the maximality of $\h'$. Hence $v$ is of type II. Note that if a piece is contained in a cell of $\h$ with $2n$ edges, then the piece has $n-1$ edges. Now, the statement about the outer border follows from Remark~\ref{rmk:real2}. The statement about the inner border follows from Lemma~\ref{lem:ob} below.
\end{proof}

We will be repeatedly using the following simple observation, which follows from Figure~\ref{f:precell}.

\begin{lemma}
	\label{lem:ob}
If we draw a cell of $\Xa_\Gamma$ as a regular polygon in $\mathbb E^2$, then opposite sides of this polygon have orientations pointing towards the same direction.
\end{lemma}

\begin{lemma}
	\label{lem:real ray border}
Suppose $\h'$ is irreducible. Suppose one real border $\B'$ of $\h'$ is unbounded. Let $\B^o\subset\h$ be the associated outer border. Then there exists a sub-ray $\B^o_1\subset\B^o$ such that edges in $\B^o_1$ have consistent orientation.
\end{lemma}

\begin{proof}
We shall show the inner border $\B^i$ has only finitely many vertices where the orientation is reversed. Then the same holds for $\B^o$ by Lemma~\ref{lem:ob} and the corollary follows. Since $\h'$ is a convex subcomplex of $\D'$ which contains at least one chamber, there exists a wall $\W\subset\h'$ such that $\B'\nsubseteq \W$  and $\W$ is parallel to $\B'$. Note that if $\B'$ is a ray (resp.\ line), then $\W$ is a ray (resp.\ line). Suppose $\W$ is closest to $\B'$ among all such walls. Since $\B'$ is a real border, any orientation reserving vertex in the inner border $\B^i$ is the intersection of two pieces by Lemma~\ref{lem:piece orient}. 

We now conduct a case study and first assume the associated Coxeter group is $(2,3,6)$, i.e.\ the labels of edges in the defining graph of the Coxeter graph (which is a triangle) are $2,3,6$. Depending on the structure of the carrier of $\B$, we consider two sub-cases as follows. 

\begin{figure}[h!]
	\centering
	\includegraphics[width=0.9\textwidth]{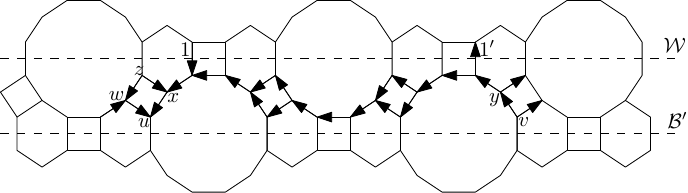}
	\caption{}
	\label{f:orient1}
\end{figure}

The first sub-case is indicated in Figure~\ref{f:orient1}. Choose orientation reversing vertices $u,v\in \B^i$ such that there are no orientation reversing vertices in $\B^i$ between $u$ and $v$. We use Lemma~\ref{lem:ob} to find orientation of edges in the boundary of the carrier $N_\W$ of $\W$. Note that $u$ may give rise to more than one orientation reversing points (e.g. $w,z$ and $x$ in Figure~\ref{f:orient1}), however, we take the \textquotedblleft rightmost\textquotedblright\ one (e.g.\ $x$ in Figure~\ref{f:orient1}). Similarly, we pick the \textquotedblleft leftmost\textquotedblright\ orientation reserving vertex $y$ determined by $v$. Note that
\begin{enumerate}
	\item $x$ and $y$ are in the same side of $\W$;
	\item both $x$ and $y$ are contained in only one $2$--cell of $N_\W$;
	\item there are no orientation reversing interior vertices in $\omega$, where $\omega$ is the shortest edge path in $\partial \W_\h$ from $x$ to $y$;
\end{enumerate}
The vertex $x$ completely determines the orientation of edges in the unique $2$--cell of $N_\W$ that contains $x$, thus edge $1$ in Figure~\ref{f:orient1} is oriented downwards. Similarly, edge $1'$ is oriented upwards. This contradicts Lemma~\ref{lem:ob}.

We now consider the second $(2,3,6)$ sub-case. Let $u$ and $v$ be as before. Again we use Lemma~\ref{lem:ob} to find orientation of other edges. When there are at least two $12$--gons between $u$ and $v$, see Figure~\ref{f:orient2}, Figure~\ref{f:orient4} and Figure~\ref{f:orient3}. When there is only one $12$--gon between $u$ and $v$, see Figure~\ref{f:orient5}. When there are no $12$--gons between $u$ and $v$, see Figure~\ref{f:orient3} left ($u'$ and $v'$ there play the roles of $u$ and $v$ respectively). 

\begin{figure}[h!]
	\centering
	\includegraphics[width=1\textwidth]{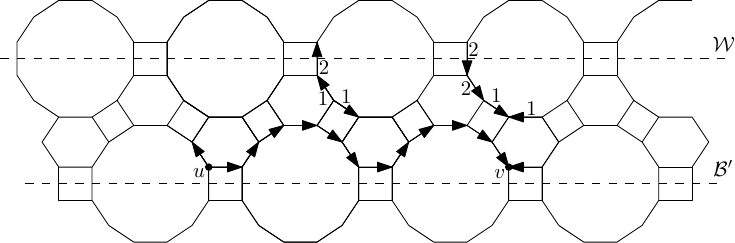}
	\caption{The orientations of edges labelled $1$ follow from Lemma~\ref{lem:ob}. The orientations
		of edges labelled $2$ follow from the form of orientation of edges in the boundary of a cell (cf. Figure~\ref{f:precell}). The two vertical edges labelled by $2$
		have opposite orientations, contradicting Lemma~\ref{lem:ob}. 
	}
	\label{f:orient2}
\end{figure}

\begin{figure}[h!]
	\centering
	\includegraphics[width=1\textwidth]{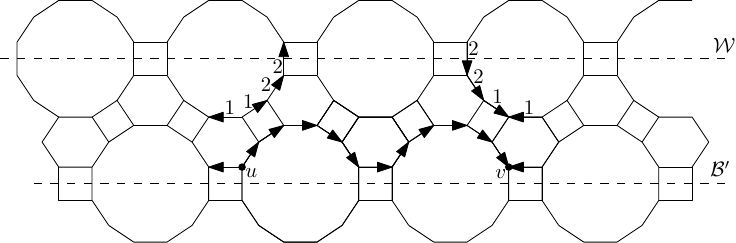}
	\caption{The orientations of edges labelled $1$ follow from Lemma~\ref{lem:ob}. The orientations
		of edges labelled $2$ follow from the form of orientation of edges in the boundary of a cell (cf.\ Figure~\ref{f:precell}). The two vertical edges labelled by $2$
		have opposite orientations, contradicting Lemma~\ref{lem:ob}. 
		}
	\label{f:orient4}
\end{figure}

\begin{figure}[h!]
	\centering
	\includegraphics[width=1\textwidth]{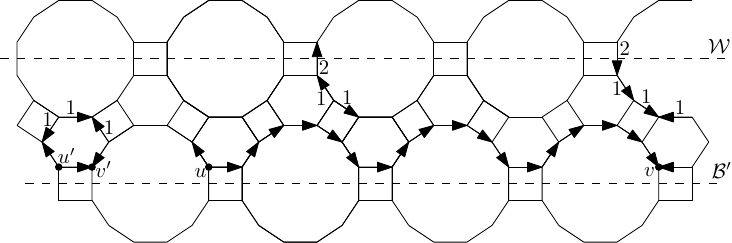}
	\caption{The orientations of edges labelled $1$ follow from Lemma~\ref{lem:ob}. The orientations
		of edges labelled $2$ follow from the form of orientation of edges in the boundary of a cell (cf.\ Figure~\ref{f:precell}). The two vertical edges labelled by $2$ 
		have opposite orientations, contradicting Lemma~\ref{lem:ob}. The orientation of the hexagon on the left contradicts the form of cells (cf.\ Figure~\ref{f:precell}).}
	\label{f:orient3}
\end{figure}

\begin{figure}[h!]
	\centering
	\includegraphics[width=1\textwidth]{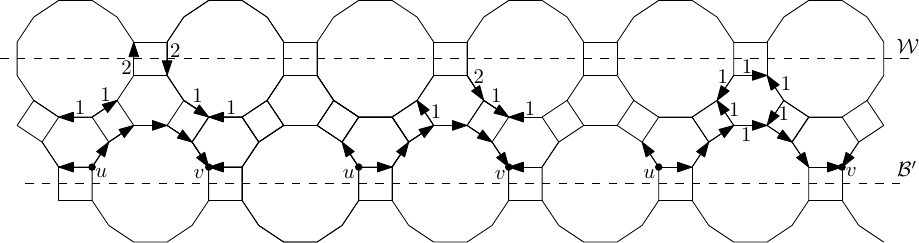}
	\caption{The following lead to contradictions: two vertical edges labelled $2$ on the left; two edges of the hexagon labelled $1$ and $2$ in the middle; orientation of the hexagon by edges labelled $1$ on the right. 
	}
	\label{f:orient5}
\end{figure}

The $(3,3,3)$ case and the $(2,4,4)$ case are similar and simpler. We omit the proof, but include Figure~\ref{f:orient10} for the convenience of the reader. Note that the second $(2,4,4)$ sub-case and the $(3,3,3)$ case are immediate since there is no gap between the carrier of $\mathcal{W}$ and the carrier of $\mathcal{B}'$.
\end{proof}
\begin{figure}[h!]
	\centering
	\includegraphics[width=1\textwidth]{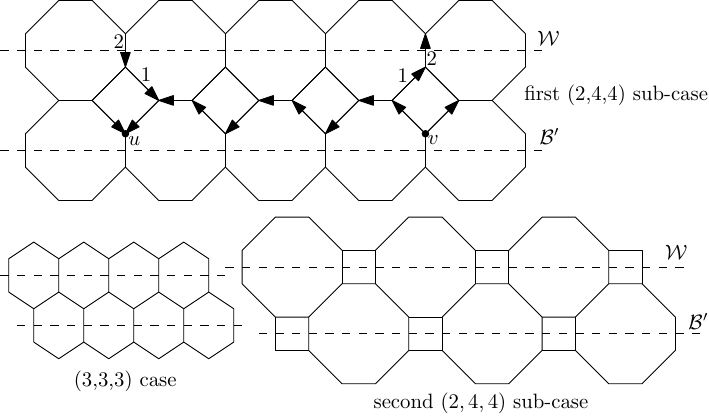}
	\caption{}
	\label{f:orient10}
\end{figure}




We record the following result which follows from the proof of Lemma~\ref{lem:real ray border}.
\begin{lemma}
	\label{lem:at most one reserving point}
Suppose $\h'$ is irreducible. Let $\B'$ be a real border of $\h'$ which is homeomorphic to $\mathbb R$. Let $\B^o\subset \h$ be the associated outer border. Then there is at most one orientation reversing vertex on $\B^o$.
\end{lemma}

\begin{prop}
	\label{prop:half coxeter region contains Coxeter ray}
Let $\h\subset Q_R$ be the $(x,\ell)$--half Coxeter region. Suppose $\h$ is irreducible and unbounded. Then $\h$ contains a Coxeter ray.
\end{prop}

\begin{proof}
Recall that the $(x,\ell)$--half Coxeter region $\h'$ associated with $\h$ can be embedded as a convex subcomplex of $\D'$. By our assumption, $\h'$ is also unbounded. Thus either $\h'$ has a real border that is unbounded, or $\h'$ has no real border and has only a fake border that is all of $\ell$. By Lemma~\ref{lem:real ray border}, it remains to deal with the latter case. We also assume the associated Coxeter group is $(2,3,6)$, since the $(2,4,4)$ and $(3,3,3)$ cases are similar and simpler.

We think of $\h$ as of a subcomplex of $\D$. By Lemma~\ref{lem:ob}, edges dual to the same wall of $\h$ have orientation pointing towards the same direction. Thus each wall of $\h$ has a well-defined orientation. Each wall $\W$ has an \emph{associated halfspace}, which is the halfspace bounded by $\W$ such that $\W$ is oriented toward this halfspace.

First we consider the case that there are two different parallel walls $\W_1$ and $\W_2$ of $\h$ with opposite orientations. We assume without loss of generality that $\W_1$ and $\W_2$ are not separated by any other wall of $\h$, and they are oriented as in Figure~\ref{f:orient6}. We define sequences of consecutive points $\{v_i\}_{i=1}^{\infty}\subset\W_1$ and $\{u_i\}_{i=1}^{\infty}\subset\W_2$ such that each $v_i$ (resp.\ $u_i$) is the center of a $12$--gon $C_{i}$ (resp.\ $C'_i$) and $\angle_{v_i}(u_i,v_{i+1})=\angle_{u_i}(v_{i+1},u_{i+1})=\pi/6$ for $i\ge 1$. Assume without loss of generality that $v_i,u_i\in \h$ for all $i$. For each $i\ge 1$, let $\mathfrak{W}_i$ be the collection of walls that are different from $\W_1$ and $\W_2$, and contain at least one of $v_j$ or $u_j$ for $j\ge i$. An element in $\mathfrak{W}_i$ is \emph{positively oriented} if its associated halfspace contains all but finitely many $v_i$, otherwise it is \emph{negatively oriented}. We claim for $i$ large enough, either each element in $\mathfrak{W}_i$ is positively oriented, or each element in $\mathfrak{W}_i$ is negatively oriented. 

Let $L$ be a straight line parallel to $\W_1$ such that $L$ is between $\W_1$ and $\W_2$ (see the dotted line in Figure~\ref{f:orient6}). Then $\mathfrak{W}_1\cap L$ is a discrete subset of $L$, which can be naturally identified with the positive integers $\mathbb Z^+$. This gives rise to a total order on $\mathfrak{W}_1$. If the smallest element $\W$ of $\mathfrak{W}_1$ is negatively oriented, then the orientation of $\W$ and $\W_1$ gives rise to an orientation-reversing vertex in the boundary of $C_1$, hence the orientation of each edge in $C_1$ is determined as in Figure~\ref{f:orient6}. Let $\W_{12}$ be the wall containing $u_1$ and $v_1$. Then the orientation of $\W_2$ and $\W_{12}$ gives rise to an orientation reversing vertex in the boundary of $C'_1$, hence the orientation of each edge in $C'_1$ is determined as in Figure~\ref{f:orient6}. By repeating this process, we know the claim holds with $i=1$. If $\W$ is positively oriented, let $\W'$ be the smallest element in $\mathfrak{W}_1$ that is negatively oriented (if such $\W'$ does not exist, then we are done). Suppose $v=\W'\cap \W_1$. If $\W'$ is the smallest element in $\mathfrak{W}_1$ that contains $v$, then we finish the proof as before. Otherwise, let $\W''$ be the biggest element in $\mathfrak{W}_1$ that contains $v$. By considering the orientation of edges in the boundary of the $2$--cell containing $v$, we know  $\W''$ is negatively oriented. Since $\W''$ intersects $\W_1$ in a 30 degree angle, we argue as before to show any element of $\mathfrak{W}_1$ which is $\ge \W''$ is negatively oriented.
\begin{figure}[h!]
	\centering
	\includegraphics[width=1\textwidth]{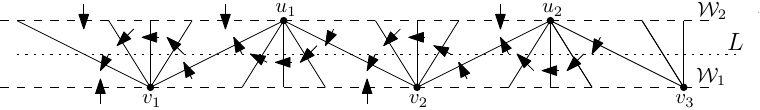}
	\caption{}
	\label{f:orient6}
\end{figure}

Now we produce a Coxeter ray from the above claim. If $\W_1$ and $\W_2$ are as in Figure~\ref{f:orient2}, then there are two centers of $12$--gons in $\W_1$ between $v_1$ and $v_2$, which we denote by $v'_1$ and $v''_1$. Now we define $\mathfrak{W}'_1$ and $\mathfrak{W}''_1$ as in the previous paragraph, and the above claim still holds for them. Since there are infinitely many walls in $\mathfrak{W}'_1\cap \mathfrak{W}_1$ and $\mathfrak{W}''_1\cap \mathfrak{W}_1$, by deleting finitely many walls from $\mathfrak{W}'_1\cup \mathfrak{W}''_1\cup \mathfrak{W}_1$, all of the rest are positively oriented (or negatively oriented). This gives rise to a Coxeter ray in the carrier of $\W_1$. If $\W_1$ and $\W_2$ are as in Figure~\ref{f:orient11}, then by the above claim, we assume without loss of generality that the orientation of edges is as in Figure~\ref{f:orient11}. We have yet to determine the orientation of $\overline{st}$. If $\W_2$ is oriented downwards (resp.\ upwards) and $\W_1$ is oriented upwards (resp.\ downwards), then the orientation reverses at $y$ (resp.\ $x$) if we consider the hexagon containing $y$ (resp.\ $x$), hence $\overline{st}$ (resp.\ $\overline{s't'}$) is oriented from $s$ to $t$ (resp.\ $s'$ to $t'$). It follows that there is a Coxeter ray contained in the carrier of $\W_1$.
\begin{figure}[h!]
	\centering
	\includegraphics[width=1\textwidth]{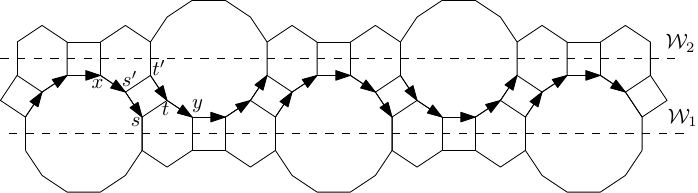}
	\caption{}
	\label{f:orient11}
\end{figure}

It remains to consider the case where each pair of parallel walls of $\h$ are oriented towards the same direction. However, it is straightforward to produce a Coxeter ray in such case.
\end{proof}

\subsection{Orientation around the corners}
Let $\h'\subset Q', H', \ell$ be as in Definition~\ref{def:half Coxeter region}. A \emph{corner} of $\h'$ is the intersection of two real borders of $\h'$. Note that if $\h'$ is bounded, then it has at least one corner. 

\begin{lemma}
	\label{lem:two right angle}
Suppose $\h'$ is irreducible. Then there does not exist a real border of $\h'$ such that it has two corners of angle $=\pi/2$. 
\end{lemma}

\begin{proof}
By Corollary~\ref{cor:local classification}, the cell in $\h$ that contains a right-angled corner of $\h'$ must be a square. Thus it suffices to consider the $(2,3,6)$ case and the $(2,4,4)$ case. The $(2,3,6)$ case has two sub-cases. The first sub-case is indicated in Figure~\ref{f:orient7}, where $k_1$ and $k_2$ are two right-angled corners.
\begin{figure}[h!]
	\centering
	\includegraphics[width=1\textwidth]{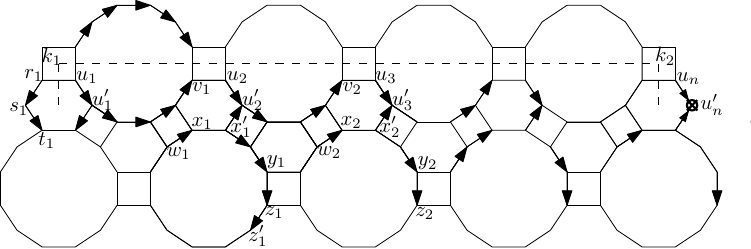}
	\caption{}
	\label{f:orient7}
\end{figure}
By Lemma~\ref{lem:piece orient}, we can assume without loss of generality that both $\overline{r_1s_1}$ and $\overline{s_1t_1}$ are oriented downwards. Thus $\overline{u_1u'_1}$ is oriented downwards since it is parallel to $\overline{s_1t_1}$. Using Lemma~\ref{lem:piece orient} again, we determine the orientation of each edge in the path $\omega_1$ from $u_1$ to $v_1$. Hence we know the orientation of $\overline{w_1x_1}$ and $\overline{u'_2x'_1}$ (since each of them is parallel to one edge in $\omega_1$), and the orientation of $\overline{z_1z'_1}$ (since it is parallel to $\overline{r_1s_1}$). Let $\omega'_1$ be the path made of 6 edges from $w_1$ to $z'_1$, containing $x_1$. The orientation of $\overline{w_1x_1}$ and $\overline{z_1z'_1}$ implies edges in $\omega'_1$ are oriented consistently. In particular, $\overline{y_1z_1}$ is oriented downwards. Using Lemma~\ref{lem:piece orient} again, we deduce the orientation of each edge in the path $\omega_2$ from $u_2$ to $v_2$, since one edge of $\omega_2$ and one edge of $\omega'_1$ are parallel. Hence $\overline{w_2x_2}$ is oriented upwards. Since $\overline{y_2z_2}$ is oriented downwards (it is parallel to $\overline{y_1z_1}$), edges in the path $\omega'_2$ of length 5 from $w_2$ to $z_2$ are oriented consistently. Now we can define $\omega_3$ and determine orientation of its edges as before. Since $\overline{x'_2u'_3}$ is oriented upwards, $u'_3$ is an orientation-reversing point. We can continue this process to deduce that $u'_{n}$ is an orientation-reversing point, however, this is impossible since $u_n$ is of type II by Lemma~\ref{lem:piece orient}. The second $(2,3,6)$ sub-case (see Figure~\ref{f:orient8}) can be studied in a similar way. Note that the orientation of the edges labeled $n$ in Figure~\ref{f:orient8} can be deduced from orientation of edges whose labels are $<n$, for example, $6$ follows from $5$ and $2$. We deduce that $u_n$ is an orientation-reversing point, which leads to a contradiction.
\begin{figure}[h!]
	\centering
	\includegraphics[width=1\textwidth]{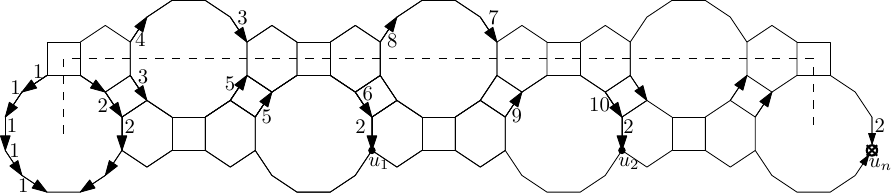}
	\caption{}
	\label{f:orient8}
\end{figure}

The $(2,4,4)$ case is similar and simpler, so we omit the proof, but we attach Figure~\ref{f:orient9} for the convenience of the reader.
\begin{figure}[h!]
	\centering
	\includegraphics[width=1\textwidth]{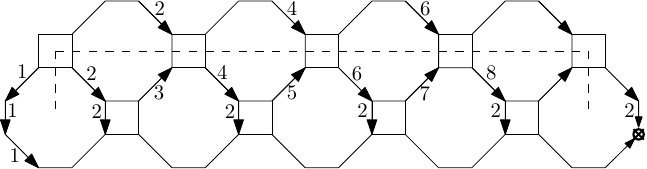}
	\caption{}
	\label{f:orient9}
\end{figure}
\end{proof}

\section{Existence of singular rays}
\label{sec:proof Existence of singular rays}

The goal of this section is proving Proposition~\ref{prop:existence of singular rays}. Let $Q',Q_R$ and $x_0$ be as in Section~\ref{subsec:local structure} and Section~\ref{subsec:new cell structure} (recall that $x_0$ is the base point in $Q'$). Let $x,\ell$ and $H'$ be as in Definition~\ref{def:half Coxeter region} and let $\h'$ and $\h$ be as in Section~\ref{sec:half coxeter regions}.

\begin{lemma}
	\label{lem:diamond}
A corner of $\h'$ with acute angle gives rise to a diamond ray.
\end{lemma}

We refer to Figure~\ref{f:development1} for the following proof.
\begin{proof}
Let $s\in\h'$ be an acute corner. Let $\Pi_1\subset\h$ be the cell containing $s$. We are in the situation of Lemma~\ref{lem:local convexity} (1) and Corollary~\ref{cor:local classification} (1). Let $K$ be the unique chamber in $\h'$ containing $s$. Let $r_K\subset Q'$ be a CAT(0) geodesic ray such that $r_K\cap K=s$ and the angles between $r_K$ and the two sides of $K$ at $s$ are equal. In all cases we have $r_K\subset H'$, in particular $r_K\subset Q_R$.

Let $\Pi,\Pi'$ be the $2$--cells in $\h$ that intersect $\Pi_1$. Then $\Pi_1\cap\Pi\cap \Pi'$ is a vertex of type III (see Table~\ref{t:flat} on page \pageref{t:flat}), which we denote by $x_1$. Note that $\Pi_1\cap \Pi$ (resp.\ $\Pi_1\cap\Pi'$) is an edge containing $x_1$. Let $y_1$ (resp.\ $y'_1$) be the other endpoint of this edge. By Lemma~\ref{lem:local convexity} (1), $y_1$ and $y'_1$ are not of type III, hence they are of type II.

Suppose $\Pi_1$ has $2n$ edges. Since $s$ is an acute corner, $n\ge 3$. Since $\Pi\cap\Pi_1$ (resp.\ $\Pi\cap\Pi'_1$) has only one edge, by Lemma~\ref{lem:star homeo} and Lemma~\ref{lem:real2}, there is a $2$--cell $C_1\subset \St(y_1,Q_R)$ (resp.\ $C'_1\subset\St(y'_1,Q_R)$) such that $C_1\cap \Pi_1$ (resp.\ $C'_1\cap\Pi_1$) is a path with $n-1$ edges, moreover, $C_1$, $C'_1$ and $\Pi_1$ are in the same block. Let $x_2$ be the vertex in $\partial\Pi_1$ that is opposite to $x_1$. Then $x_2\in (C_1\cap C'_1\cap\Pi_1)$. Note that $x_2\in r_K\subset Q_R$. Thus by Lemma~\ref{lem:star homeo} and Lemma~\ref{lem:real1}, $x_2$ is of type I. Let $\Pi_2$ be the other $2$--cell in $\St(x_2,Q_R)$ besides $C_1$, $C'_1$ and $\Pi_1$. Either $x_2$ is the tip of both $\Pi_1$ and $\Pi_2$, or $x_2$ is the tip of both $C_1$ and $C'_1$. We deduce from Lemma~\ref{lem:real1} and $n\ge 3$ that the latter case is impossible. Also we deduce from Lemma~\ref{lem:real1} that $C_1\cap \Pi_2$ is an edge. Let $y_2$ be the endpoint of this edge that is not $x_2$. We claim $y_2$ is of type I.

Since $y_2$ is an interior vertex of $Q_R$, Lemma~\ref{lem:star homeo} describes all possibilities of $\St(y_2,Q_R)$. Lemma~\ref{lem:star homeo} (1) can be ruled out immediately. Note that $y_2$ is not of type III since $C_1$ and $\Pi_2$ are in the same block. If $y_2$ is of type II, then Lemma~\ref{lem:real2} (2) and (3) imply $C_1\cap\Pi_2$ has $n-1$ edges, which is contradictory to $n\ge 3$. Thus the above claim follows.

Note that $y_2$ is not a tip of $\Pi_2$ (since $x_2$ is a tip of $\Pi_2$). Hence $y_2$ is a tip of $C_1$. By Lemma~\ref{lem:real1}, there is a $2$--cell $C_2\subset\St(y_2,Q_R)$ such that $C_2\cap \Pi_2$ is a path with $n-1$ edges. We define $C'_2$ similarly. Let $x_3$ be the vertex in $\partial\Pi_2$ opposite to $x_2$. Again we know $x_3\in r_K\subset Q_R$ and $x_3$ is contained in three cells $C_2,C'_2$ and $\Pi_2$ which are in the same block. Hence $x_3$ is of type I and the pair $(\Pi_2,\Pi_3)$ plays the role of $(\Pi_1,\Pi_2)$ in Figure~\ref{f:flat2}. By repeating this argument, we can inductively produce a sequence of $2$--cells $\{\Pi_i\}_{i=1}^{\infty}$ such that $\Pi_i\cap\Pi_{i+1}$ is a vertex of type I contained in $r_K$ and $(\Pi_i,\Pi_{i+1})$ plays the role of $(\Pi_1,\Pi_2)$ in Figure~\ref{f:flat2}, which finishes the proof.
\end{proof}
\begin{figure}[h!]
	\centering
	\includegraphics[width=1\textwidth]{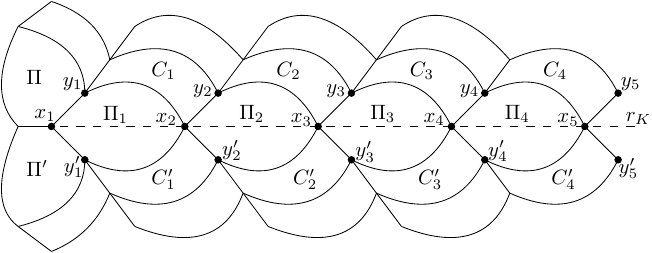}
	\caption{}
	\label{f:development1}
\end{figure}

\begin{corollary}
	\label{cor:delta point}
Suppose that outside any compact subset of $Q_R$ there is a $\triangle$--vertex (i.e. type III vertex whose support is a triangle). Then there is either a diamond ray, or a Coxeter ray in $Q_R$.
\end{corollary}

\begin{proof}
Consider an irreducible $(x,\ell)$--half Coxeter region $\h'$. If $\h'$ is unbounded, then we obtain a Coxeter ray by Proposition~\ref{prop:half coxeter region contains Coxeter ray}. If $\h'$ contains an acute corner, then we obtain a diamond ray by Lemma~\ref{lem:diamond}. Now we assume $\h'$ is bounded and does not contain any acute corner. By Lemma~\ref{lem:two right angle}, $\h'$ contains a unique right-angled corner $s$.
	
Let $K$ be the unique chamber in $\h'$ containing $s$. We extend the two borders of $\h'$ meeting at $s$ to two CAT(0) geodesic lines $\ell_1,\ell_2$ in $Q'$. Since $\h'$ is bounded and $s$ is the only corner of $\h'$, we know that $\ell,\ell_1$ and $\ell_2$ bound a triangle which contains $K$. Note that the angle between $\ell_1$ (or $\ell_2$) and $\ell$ is $\ge 30^{\circ}$. Thus by choosing $s$ sufficiently far away from the base point $x_0$ of $Q'$, we know at least one of $\ell_1$ and $\ell_2$, say $\ell_1$, satisfies that $\ell_1\subset Q_R$ and that $K$ and $x_0$ are on the same side of $\ell_1$.  

Let $\C'_K$ be the Coxeter region containing $K$. Let $\ell'_1\subset\ell_1$ be the connected component of $\C'_K\cap \ell_1$ containing $s$. By local convexity of $\C'_K$ along $\ell'_1$, we know the intersection of $\C'_K$ and an $\epsilon$--neighborhood of $\ell_1$ (for $\epsilon$ small enough) is on one side of $\ell_1$. If $\ell'_1$ is unbounded, i.e.\ $\ell'_1$ is a ray based at $s$, then by Corollary~\ref{cor:local classification} and the argument in Lemma~\ref{lem:real ray border}, we know a sub-ray of $\ell'_1$ gives rise to a Coxeter ray. If $\ell'_1$ is bounded, let $t$ be the other endpoint of $\ell'_1$. Then $t$ satisfies Corollary~\ref{cor:local classification} (1). The proof of Lemma~\ref{lem:two right angle} implies that the angle at $t$ is acute, then we can produce a diamond ray by repeating the argument in Lemma~\ref{lem:diamond} (we define a ray $r$ based at $t$ as in the first paragraph of the proof of Lemma~\ref{lem:diamond}, then $r$ and $K$ are in different sides of $\ell_1$, hence, $r$ and the base point $x_0$ are in different sides of $\ell_1$, thus $r\subset Q_R$, the rest of the proof is identical).
\end{proof}

\begin{prop}
	\label{prop:existence of singular rays}
$Q_R$ contains a singular ray.
\end{prop}

\begin{proof}
By Corollary~\ref{cor:delta point}, it suffices to consider the case that all the $\triangle$--vertices of type III are contained in a bounded subset of $Q'$. We assume without loss of generality that $Q_R$ does not contain any $\triangle$--vertex. 

Suppose there is an interior vertex $y_0\in Q_R$ such that $y_0$ is of type II or $y_0$ is a $\square$--vertex. Pick edges $e_0,e_1$ of $Q_R$ containing $y_0$ such that
\begin{enumerate}
	\item for $i=0,1$, $e_i$ is the intersection of two $2$--cells of $\St(y_0,Q_R)$ that are in different blocks;
	\item $e_0$ and $e_1$ are not contained in the same $2$--cell.
\end{enumerate}
By Lemma~\ref{lem:real2} and Lemma~\ref{lem:real3}, $e_0$ and $e_1$ always exist (such pair is unique when $y_0$ is of type II). By Lemma~\ref{lem:2pi cycle} (2), $e_0$ and $e_1$ are also edges in $Q'$, and the concatenation of $e_1$ and $e_0$ is a CAT(0) geodesic segment in $Q'$. 

Let $r^+$ and $r^-$ be two CAT(0) geodesic rays based at $y_0$ containing $e_1$ and $e_0$ respectively. Then the angle between $\overline{x_0y_0}$ (recall that $x_0$ is the base point of $Q'$) and at least one of $r^+$ and $r^-$, say $r^+$, is $\ge 90^{\circ}$. Then $d(r^+,x_0)\ge d(y_0,x_0)$ by CAT(0) geometry. Hence $r^+\subset Q_R$. We claim $r^+$ is a plain ray. Let $y_1$ be the other endpoint of $e_1$. Since $e_1$ is contained in two $2$--cells that are in different blocks, $y_1$ is not of type I. Hence $y_1$ is of type II or $y_1$ is a $\square$--vertex. By Lemma~\ref{lem:real2} and Lemma~\ref{lem:real3}, there is another edge $e_2$ containing $y_1$ such that the above two conditions hold with $e_0,e_1,y_0$ replaced by $e_1,e_2,y_1$. Note that $e_2\subset r^+$. By repeating such argument, we know $r^+$ is a union of edges $\{e_i\}_{i=1}^{\infty}$ such that $e_i\cap e_{i+1}$ is a vertex of type II for each $i$. Hence $r^+$ is a plain ray.

It remains to consider the case when every interior vertex of $Q_R$ is of type I. One readily verifies that there is a diamond ray in such case.
\end{proof}

\section{Development of singular rays}
\label{sec:development}
A \emph{flat sector (with angle $\alpha$)} in $Q'$ is a subset of $Q'$ isometric to a sector in the Euclidean plane $\mathbb E^2$ bounded by two rays meeting at an angle $\alpha$ with $0\le \alpha\le\pi$. Since $Q'$ is homeomorphic to $\mathbb R^2$, we specify an orientation of $Q'$. Let $\gamma\subset Q'$ be a CAT(0) geodesic ray. We think of $\gamma$ as being oriented from its base point to infinity (if $\gamma$ happens to be mapped to the $1$--skeleton of $\Xa_\Gamma$, then such orientation of $\gamma$ may be different from the induced orientation from $\Xa_\Gamma$). Recall that we use $d_H$ to denote the Hausdorff distance. A flat sector $S$ with angle $0\le\alpha\le \pi$ is \emph{on the right side} of $\gamma$ if
\begin{enumerate}
	\item one boundary ray $\gamma_1$ of $S$ satisfies that $d_H(\gamma_1,\gamma)<\infty$;
	\item the interior of $S$ has empty intersection with $\gamma$;
	\item $S$ is to the right of $\gamma$ with respect to the orientation of $\gamma$ and $Q'$.
\end{enumerate}
Since a flat sector with angle $0$ is just a geodesic ray of $Q'$, we know from the above definition the meaning of a geodesic ray $\gamma'$ on the right side of $\gamma$.

The goal of this section is the following theorem. Recall the Definition~\ref{d:atomic} of atomic sectors and
Table~\ref{t:atomic} on page \pageref{t:atomic}.
\begin{theorem}
	\label{thm:main1}
	Let $r\subset Q_R$ be a singular ray. Then there is a CAT(0) geodesic ray $\gamma$ in the $1$--skeleton of $Q'$ with $\gamma\subset r$. And there exists a singular ray $r_1\subset Q_R$ such that 
	\begin{enumerate}
		\item the CAT(0) geodesic ray $\gamma_1\subset r_1$ satisfies that $d_H(\gamma_1,\gamma)<\infty$;
		\item $r_1$ is a boundary ray of an atomic sector $S_1$ in $Q_R$;
		\item there is a flat sector $S'_1$ in $Q'$ such that $S'_1\subset S_1$ and $d_H(S_1,S'_1)<\infty$;
		\item $S'_1$ is on the right side of $\gamma$;
		\item the angle of $S'_1$ is $\ge \pi/6$.
	\end{enumerate}
\end{theorem}
Since $Q'$ is flat outside a compact subset, (1) actually means something stronger: a sub-ray of $\gamma$ and a sub-ray of $\gamma_1$ bound an isometrically embedded flat strip $[0,\infty)\times[0,a]$ in $Q'$.

Assuming Theorem~\ref{thm:main1}, we can deduce the following, being one of our main results (see Theorem~\ref{thm:intro1} in Introduction).
\begin{theorem}
	\label{thm:main2}
	There are finitely many mutually disjoint atomic sectors $\{S_i\}_{i=1}^{n}$ in $Q_R$ such that $d_H(\cup_{i=1}^{n}S_i,Q_R)<\infty$. Hence each $2$--quasiflat in $\Xa_\Gamma$ is at bounded distance from a union of finitely many atomic sectors. 
\end{theorem}

\begin{proof}
	By Proposition~\ref{prop:existence of singular rays}, there exists at least one singular ray $r_1$ in $Q_R$. Now we apply Theorem~\ref{thm:main1} to $r_1$ to obtain an atomic sector $S_1\subset Q_R$ and a CAT(0) sector $S'_1\subset Q'$. Let $r_2$ be a boundary ray of $S_1$ such that $d_H(r_2,r_1)=\infty$. We apply Theorem~\ref{thm:main1} to $r_2$ to produce $S_2$ and $S'_2$. Let $S_i$ and $S'_i$ be the sectors produced by repeating this process for $i$ times. Consider $\{\partial_T S'_i\}$ in the Tits boundary $\partial_T Q'$ of $Q'$. Since $Q'$ comes from a quasiflat, $\partial_TQ'$ is a circle whose length is $\ge 2\pi$ and $<\infty$. By Theorem~\ref{thm:main1}, the length of each $\partial_T S'_i$ is between $\pi/6$ and $\pi$, and $\partial_T Q'$ inherits an orientation from $Q'$ such that the terminal point of $\partial_T S'_i$ is the starting point of $\partial_T S'_{i+1}$. So there exists $n$ such that $\cup_{i=1}^n S'_i$ covers $\partial_T Q'$. We take $n$ to be the smallest possible. It is clear that $d_H(\cup_{i=1}^n S'_i,\partial_T Q')<\infty$, hence $d_H(\cup_{i=1}^{n}S_i,Q_R)<\infty$. 
	
	Now we show that it is possible to arrange $\{S_i\}_{i=1}^n$ such that they are mutually disjoint in $Q_R$ (however, we are not claiming their images in $\Xa_\Gamma$ are also mutually disjoint). It suffices to show the terminal point of $\partial_T S'_n$ is the starting point $\eta$ of $\partial_T S'_1$, then we can pass to suitable sub-sectors of $S_i$ such that they are mutually disjoint. Suppose by contradiction that $\eta$ is in the interior of $\partial S'_n$. Then there is a boundary ray $r$ of $S_1$ such that $r$ is contained in $S_n$ (modulo a compact subset of $r$), but $r$ is not at finite Hausdorff distance away from any of the boundary rays of $S_n$. However, this is impossible by definitions of atomic sectors. Now the theorem follows.
\end{proof}

In the rest of this section, we prove Theorem~\ref{thm:main1}. The proof is divided into three cases when $r$ is a plain ray (Section~\ref{subsec:development of plain rays}), a Coxeter ray (Section~\ref{subsec:development of Coxeter rays}), or a diamond ray (Section~\ref{subsec:development of diamont rays}). We conclude the proof at the end of the section.

\subsection{Development of plain rays}
\label{subsec:development of plain rays}
We start with several preparatory observations. An interior vertex $v\in Q_R$ is \emph{good} if $\St(v,Q_R)$ has four $2$--cells $\{\Pi_i\}_{i=1}^4$ as in Figure~\ref{f:flat4} such that
\begin{enumerate}
	\item $\partial\Pi_1$ and $\partial\Pi_2$ have $2n_1$ edges, $\partial\Pi_3$ and $\partial\Pi_4$ have $2n_2$ edges;
	\item $\Pi_3\cap\Pi_2=\Pi_1\cap\Pi_4=\{v\}$;
	\item $\Pi_1\cap\Pi_2$ has $n_1-1$ edges, $\Pi_3\cap\Pi_4$ has $n_2-1$ edges, $\Pi_1\cap\Pi_3$ is an edge $e_1$ and $\Pi_2\cap\Pi_4$ is an edge $e_2$.
\end{enumerate}
Suppose $v$ is good. If $\Pi_1,\Pi_2,\Pi_3,\Pi_4$ are contained in the same block, then $v$ is of type I (see Table~\ref{t:flat} on page \pageref{t:flat}). If $\{\Pi_i\}_{i=1}^4$ are contained in two different blocks, then $v$ is of type II. If $\{\Pi_i\}_{i=1}^4$ are contained in more than two different blocks, then $v$ is a $\square$--vertex. We record the following observations, which follow from the argument in the second paragraph of the proof of Lemma~\ref{lem:real2} (see also Figure~\ref{f:flat1}).

\begin{lemma}
	\label{lem:concat}
	Suppose $v$ is a good vertex and let $e_1$ and $e_2$ be as above. Let $u$ be the endpoint of $\Pi_1\cap\Pi_2$ with $u\neq v$. Then 
	\begin{enumerate}
		\item $e_1$ and $e_2$ are also edges in $Q'$ and their concatenation forms a CAT(0) geodesic segment in $Q'$;
		\item the convex hull of $e_1$ (or $e_2$) and $u$ in $Q'$ is a flat triangle with a right angle at $v$ and an angle of $\pi/2n$ at $u$, where $2n$ is the number of edges in $\partial\Pi_1$.
	\end{enumerate}
\end{lemma}

Let $Z$ be $\mathbb E^2$ tiled by unit squares. Suppose the tiling is compatible with the coordinate system on $\mathbb E^2$. So it makes sense to talk about vertical edges, or horizontal edges of $Z$. A subcomplex of $P\subset Q_R$ is \emph{quasi-rectangular}, or \emph{$Z'$--quasi-rectangular}, if there exists a homeomorphism $h$ from a convex subcomplex $Z'\subset Z$ to $P$ such that for each square $\sigma\subset Z'$, $h(\sigma)$ is a $2$--cell $C$ of $P$, $h(\partial\sigma)=\partial C$ and $h$ maps the two vertical edges of $\sigma$ to a pair of opposite edges on $\partial C$. We refer to Figure~\ref{f:simplelinethickening1} right for an example of a $[0,4]\times[0,5]$--quasi-rectangular subcomplex of $Q_R$. Each interior vertex of the quasi-rectangular region $P$ is either a vertex of type O, or a good vertex. 


Let $r\subset Q_R$ be a plain ray based at a vertex $y$. By Definition~\ref{def:plain line} and Table~\ref{t:flat}, each vertex in $r$ that is not $y$ and is not adjacent to $y$ is good. Thus by passing to a sub-ray, we assume each vertex of $r$ is good. By Lemma~\ref{lem:concat} (1), $r$ is a CAT(0) geodesic ray in $Q'$. Let $\ell_r$ be a CAT(0) geodesic line orthogonal to $r$ at $y$. Let $H_r$ be the halfspace which is bounded by $\ell_r$ and contains $r$. Up to passing to a sub-ray of $r$, we assume $H_r$ is sufficiently far away from the base point $x_0$ of $Q'$ so that $H_r$ is in the interior of $Q_R$.

The thickening $T_r$ of $r$ is defined to be the union of $2$--cells of $Q_R$ which contains an edge of $r$. We label these $2$--cells and vertices of $r$ as in Figure~\ref{f:simplelinethickening1} left, namely, the $2$--cells that contain the edge $\overline{y_{0,j-1}y_{0j}}$ are denoted by $C_{0j}$ and $C_{1j}$. Suppose $\partial C_{01}$ (resp.\ $\partial C_{11}$) is made of $2n_0$ (resp.\ $2n_1$) edges. Since each $y_{0j}$ is a good vertex, $\partial C_{ij}$ is made of $2n_i$ edges for $i=0,1$ and $j\ge 1$ and $C_{ij}\cap C_{i,j+1}$ is a path with $n_i-1$ edges.
\begin{figure}[h!]
	\centering
	\includegraphics[width=1\textwidth]{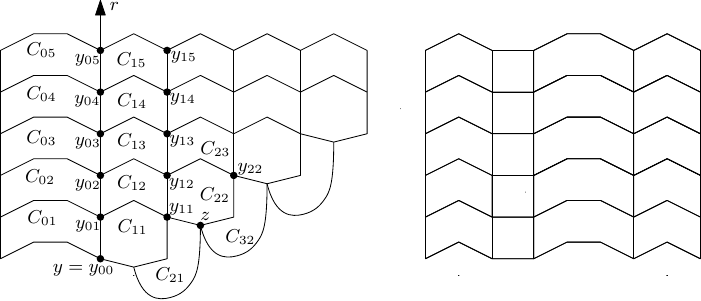}
	\caption{}
	\label{f:simplelinethickening1}
\end{figure}

Now we look at the vertex $y_{11}$ in Figure~\ref{f:simplelinethickening1}, i.e.\ $y_{11}=\partial T_r\cap C_{12}\cap C_{11}$. Since $y_{01}$ is good, by Lemma~\ref{lem:concat} (2), $y_{01},y_{11}$ and $y_{00}$ span a triangle in $Q'$ with a right-angle at $y_{01}$. Hence $y_{11}\in H_r$ and $y_{11}$ satisfies Lemma~\ref{lem:star homeo} by our choice of $H_r$.

\begin{lemma}\
	\label{lem:producing diamond-plain sector}
	\begin{enumerate}
		\item If $y_{11}$ is not good, then $r$ bounds a diamond-plain sector (see Table~\ref{t:atomic} on page \pageref{t:atomic}) on the right side of $r$.
		\item If $y_{11}$ is good, then $y_{1i}$ is good for all $i\ge 1$.
	\end{enumerate}
\end{lemma}

The reader might find Table~\ref{t:flat} helpful for the following proof. 
\begin{proof}
	We prove (1). Since $C_{11}\cap C_{12}$ has $n_1-1$ edges, $y_{11}$ is not a $\triangle$--vertex. Since $y_{11}$ is not good, it is of type I and $n_1\ge 3$. Let $C_{22}$ and $C_{21}$ be other two $2$--cells containing $y_{11}$. For $j\ge 2$, let $C_{2j}$ be the $2$--cell containing $\overline{y_{1j}y_{1,j-1}}$ such that $C_{2j}\neq C_{1j}$. Note that $C_{12}\cap C_{22}=\overline{y_{11}y_{12}}$ (since other edges of $\partial C_{12}$ are also contained in either $C_{13}$, or $C_{11}$, or $C_{02}$). By Lemma~\ref{lem:real1} (2) and (3), either $y_{11}$ is a tip of both $C_{11}$ and $C_{22}$, or $y_{11}$ is a tip of both $C_{12}$ and $C_{21}$. If the latter case is true, then by Lemma~\ref{lem:real1} (4), $C_{11}\cap C_{21}$ has one edge (since $C_{11}\cap C_{12}$ has $n_1-1$ edges) and $C_{22}\cap C_{21}$ has $n_1-1$ edges (since $C_{12}\cap C_{22}$ has one edge), which implies $y_{11}$ is good and leads to a contradiction. Thus the former case holds.
	
	Since $C_{12},C_{22}$ and $C_{13}$ are in the same block, $y_{12}$ is of type I. We deduce from $C_{12}\cap C_{22}$ is an edge and Lemma~\ref{lem:real1} (4) that $y_{12}$ is good. Moreover, since $y_{12}$ is not a tip of $C_{22}$, $y_{12}$ is a tip of both $C_{12}$ and $C_{23}$. By repeating this argument (note that $y_{1j}\in H_r\subset Q_R$), we deduce that for each $j\ge 2$, $y_{1j}$ is a good vertex of type I and $y_{1j}$ is a tip of $C_{1j}$ and $C_{2,i+1}$. 
	
	Suppose $z\in \partial C_{22}$ is the endpoint of $C_{21}\cap C_{22}$ such that $z\neq y_{11}$. It is clear that $z\in H_r$, hence $z$ satisfies Lemma~\ref{lem:star homeo}. We claim $z$ is of type I. Since $C_{21}$ and $C_{22}$ are in the same block, $z$ is either type I or type II (note that $z\in Q_R$). However, since $y_{11}$ is not good, $C_{22}\cap C_{21}$ has $< n_1-1$ edges, which contradicts Lemma~\ref{lem:real2} (3). Thus $z$ is of type I.
	
	Let $y_{22}$ be the end point of $C_{22}\cap C_{23}$ with $y_{22}\neq y_{12}$. Since both $y_{01}$ and $y_{02}$ are good, we deduce from the relation between $Q'$ and $Q_R$ that the convex hull of $y_{01},y_{02},y_{11}$ and $y_{22}$ in $Q'$ is isometric to a flat rectangle. Since $y_{12}$ is good, by Lemma~\ref{lem:concat} (2), $y_{12},y_{11}$ and $y_{22}$ span a flat triangle in $Q'$ with a right angle at $y_{12}$. Hence $y_{22}\in H_r$ and $y_{22}$ satisfies Lemma~\ref{lem:star homeo}.
	
	Since $y_{11}$ is a tip of $C_{22}$, $z$ is not a tip of $C_{22}$. Thus $z$ is a tip of both $C_{21}$ and $C_{32}$. By Lemma~\ref{lem:real1} (4), $C_{22}\cap (C_{21}\cup C_{32})$ is a half of $C_{22}$. Thus $C_{22}\cap C_{32}$ is a path from $z$ to $y_{22}$. We deduce that $y_{22}$ is not good and is of type I, and we can apply the previous discussion to $y_{22}$.
	
	The above argument enable us to build inductively a diamond ray $L=\cup_{i=1}^{\infty} C_{ii}$ and a sector between $r$ and $L$. Moreover, for each $i\ge 1$, $\cup_{j=1}^{\infty} C_{j,j+i}$ is a diamond ray. Since $y_{i,i+k}$ is good for each $i\ge 0$ and $k\ge 1$, one readily verify that this sector is a diamond-plain sector.
	
	(2) already follows from the proof of (1).
\end{proof}

By Lemma~\ref{lem:concat}, the vertices $\{y_{ii}\}_{i=0}^{\infty}$ are contained in a CAT(0) geodesic ray $L'$ emanating from $y_{00}$, and the convex hull of $r$ and $L'$ in $Q'$ is a flat sector with angle $\frac{n_1-1}{2n_1}\pi\ge \pi/3$. This flat sector is at finite Hausdorff distance from the diamond-plain sector of $Q_R$ produced above in Lemma~\ref{lem:producing diamond-plain sector}.

It follows from Lemma~\ref{lem:producing diamond-plain sector} that exactly one of the following holds:
\begin{enumerate}[label=(\alph*)]
	\item $r$ bounds a $[0,\infty)\times [0,\infty)$--quasi-rectangular region $P$ on its right side with $r=\{0\}\times [0,\infty)$;
	\item there are a $[0,n]\times [0,\infty)$--quasi-rectangular region $P_1$ and a diamond-plain sector $P_2$ such that $r$ bounds $P_1$ on its right side with $r=\{0\}\times [0,\infty)$ and $P_1\cap P_2=\{n\}\times [0,\infty)$.
\end{enumerate}

Now we look at case (a) in more detail. We identify $P$ with $[0,\infty)\times [0,\infty)$ via $h\colon [0,\infty)\times [0,\infty)\to P$ (note that not all vertices of $P$ have integer $x$--coordinates). Let $\ell'$ be a CAT(0) geodesic line in $Q'$ that contains $r$, and let $H_P$ be the half-space bounded by $\ell'$ with $P\subset H_P$. By possibly replacing $[0,\infty)\times [0,\infty)$ by $[m,\infty)\times [0,\infty)$, we assume $H_P$ is sufficiently far away from the base point $x_0$ of $Q'$ such that $H_P$ is contained in the interior of $Q_R$. There are three cases to consider.

\emph{Case 1:} Suppose each vertex of $P$ in $[0,\infty)\times\{0\}$ is either of type O, or a good vertex. If $P$ is made of squares, then it is a plain sector. Now we assume at least one $2$--cell of $P$ is not a square. For each integer $k\ge 0$, let $w_k=\{k\}\times\{0\}$. Let $\Pi_k$ and $\Pi'_k$ be the two $2$--cells in $Q_R$ that contain $w_k$ and are not contained in $P$. There exists an integer $k_0$ such that the vertex $w_{k_0}$ is contained in a cell $\Pi$ of $P$ with $\ge 6$ edges. Since $w_{k_0}$ is of type II, $\Pi_{k_0}\cap\Pi'_{k_0}$ is one edge. Now we can prove inductively that $\Pi_k\cap \Pi'_k$ is one edge for all $k\ge 0$. Thus we can enlarge $P$ to be a $[0,\infty)\times [-1,\infty)$--quasi-rectangular region. If we can repeat this process for infinitely many times to obtain a $[0,\infty)\times(-\infty,\infty)$--quasi-rectangular region $P'$, then at least one of the following cases hold:
\begin{itemize}
	\item a sub-region of $P'$ of form $[m,\infty)\times(-\infty,\infty)$ is made of squares, in which case $P'$ contains a plain sector on the right side of $r$;
	\item $P'$ is a PCH as in Definition~\ref{def:PCH} (see Table~\ref{t:atomic} on page \pageref{t:atomic}), in which case $P'$ is Hausdorff close to a half-flat in $Q'$ by Lemma~\ref{lem:concat};
	\item there exists $m\ge 0$ such that the subcomplex $P''\subset P'$ corresponding to $[m,\infty)\times(-\infty,\infty)$ is contained in a block. Since $P''$ is quasi-rectangular, it is actually contained in a diamond-plain flat. Consequently $P''$ contains diamond-plain sector on the right side of $r$.
\end{itemize}

\emph{Case 2:} Suppose there is a vertex of $P$ in $[0,\infty)\times\{0\}$, say $y_{00}=\{0\}\times\{0\}$,  of type I but not good. Then $n_1\ge 3$ (recall that $C_{11}$ has $2n_1$ edges). Let $C_{00},C_{10}$ be $2$--cells containing $y_{00}$ other than $C_{01}$ and $C_{11}$. 

If $y_{00}$ is a tip of both $C_{00}$ and $C_{11}$ (we refer to Figure~\ref{f:producediamondsector} right) then, by Lemma~\ref{lem:real1}, $C_{01}\cap (C_{00}\cup C_{11})$ is a half of $\partial C_{01}$. Thus $C_{01}\cap C_{00}$ has $n_1-1$ edges. Since $y_{00}$ is not good, $C_{10}\cap C_{11}$ has $<n_1-1$ edges. Let $z$ be the end point of $C_{10}\cap C_{11}$ other than $y_{00}$. Then $z$ is an interior vertex of the path from $y_{00}$ to $y_{10}$. Since $C_{10}$ and $C_{11}$ are in the same block, $z$ is either of type I or of type II. In the former case, since $z$ is not a tip of $C_{11}$, $z$ is a tip of $C_{10}$ and $C_{20}$. Then $C_{11}\cap(C_{10}\cup C_{20})$ is a half of $\partial C_{11}$. Hence $y_{10}$ is an interior point of the interval $C_{11}\cap C_{20}$, which contradicts that $Q_R$ is a manifold around $y_{10}$. In the later case, we still have $C_{11}\cap(C_{10}\cup C_{20})$ is a half of $\partial C_{11}$, which can be deduced from Lemma~\ref{lem:real2} (2) and (3). Thus we reach a contradiction as before.

Thus $y_{00}$ is a tip of both $C_{01}$ and $C_{10}$ (we refer to Figure~\ref{f:producediamondsector} left). In this case, we argue as in Lemma~\ref{lem:producing diamond-plain sector} to produce a diamond ray $\cup_{i=0}^{\infty}C_{i,1-i}$ such that $r$ and this diamond ray bound a diamond-plain sector $S_1$. Moreover, $S_1$ contains a flat sector $S'_1\subset Q'$ with angle $=\frac{n_1+1}{2n_1}\pi$ (recall that $2n_1$ is the number of edges in $\partial C_{11}$).

\begin{figure}[h!]
	\centering
	\includegraphics[width=1\textwidth]{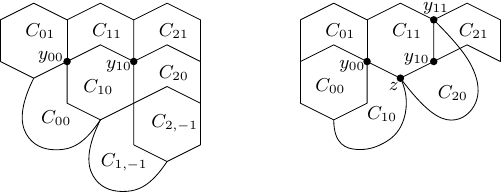}
	\caption{}
	\label{f:producediamondsector}
\end{figure}

\emph{Case 3:} Suppose there is a vertex $v$ of $P$ in $[0,\infty)\times\{0\}$ such that $v$ is a $\triangle$--vertex. In the light of Case 1 and Case 2, this is the only possibility left for discussion. Let $x_v$ be the $x$--coordinate of $v$ and let $\{v_i\}_{i=0}^{\infty}$ be consecutive vertices of $P$ on $[x_v,\infty)\times\{0\}$ with $v_0=v$. 

If each $v_i$ is of type III, then $[x_v,\infty)\times\{0\}$ or its appropriate sub-ray can be realized as part of the outer border of some half Coxeter region of $Q_R$. Then by Lemma~\ref{lem:real ray border}, there is a sub-ray $r'$ of $[x_v,\infty)\times\{0\}$ such that $r'$ is the boundary of a Coxeter ray. We deduce from $P$ being quasi-rectangular that $r'\times [0,\infty)$ is a Coxeter-plain sector, moreover, by Lemma~\ref{lem:concat}, $P$ is Hausdorff close to a flat sector of angle $\pi/2$ in $Q'$.

Now we assume at least one of $\{v_i\}$ is not of type III and assume without loss of generality that $v_1$ is the first vertex which is not of type III. Let $\Pi_1,\Pi_2$ be the two $2$--cells containing $\overline{v_0v_1}$ with $\Pi_1\subset P$. Since $v_0$ is of type III, $\Pi_1$ and $\Pi_2$ are in different blocks. Thus $v_1$ is either a $\triangle$--vertex with the same support as $v_0$ (Lemma~\ref{lem:same tag}), which can be ruled out by our assumption, or $v_1$ is of type II. Let $\Pi_3,\Pi_4$ be another two $2$--cells such that $\Pi_1\cap\Pi_3$ contains at least one edge. Suppose $\Pi_i$ has $2m_i$ edges. Since $\Pi_1\cap\Pi_2=\overline{v_0v_1}$ is one edge, by Lemma~\ref{lem:real2}, $\Pi_3\cap \Pi_1$ has $m_1-1$ edges. We deduce that $\Pi_1\cap (\Pi_2\cup\Pi_3)$ has $m_1$ edges. Since $\Pi_1\cap ([0,\infty)\times\{0\})$ has $m_1-1$ edges, $\Pi_1\cap\Pi_3$ has at least one edge which is not contained in $[0,\infty)\times\{0\}$. Hence $\Pi_3\subset P$. Moreover, $\Pi_1\cap \Pi_3$ is an edge since $P$ is quasi-rectangular. Thus $m_i=2$ and both $\Pi_1$ and $\Pi_3$ are squares.

We claim either all $v_i$ with $i\ge 2$ are of type II, or there exists $v_{i_0}$ with $i_0\ge 2$ such that $v_{i_0}$ is a $\triangle$--vertex with the same support as $v_0$ and $v_i$ is of type II for each $1\le i<i_0$. To see this, note that $v_2$ is contained in $\Pi_3$ and $\Pi_4$ which are in different blocks, so $v_2$ is either of type II or type III. If $v_2$ is of type II, let $\Pi_5$ and $\Pi_6$ be another two $2$--cells of $Q_R$ containing $v_2$. Then $\Pi_5$ is a square in $P$ and $|\partial\Pi_6|=|\partial\Pi_4|=|\partial\Pi_2|$ ($|\partial\Pi_i|$ denotes the number of edges in $\partial\Pi_i$). If $v_2$ is of type III, since $\Pi_4$ and $\Pi_2$ (and $\Pi_3$ and $\Pi_1$) are in the same block, $v_2$ is a $\triangle$--vertex with the same support as $v_0$. We now deduce the claim by repeating this argument. 

If $v_i$ is of type II for each $i\ge 2$, then there is a sub-region $P_1\subset P$ of form $[m,\infty)\times[0,\infty)$ for some $m$ such that each cell in $P_1$ is a square. Thus $P_1$ is a plain sector and it is isometric to a right-angled flat sector.

It remains to consider the case where $v_{i_0}$ in the above claim exists. We claim for all $i\ge i_0$, $v_i$ is a $\triangle$--vertex with the same support as $v_{i_0}$. Given this claim, one can find as before a sub-ray $r''$ of $[x_v,\infty)\times\{0\}$ such that $r''\times [0,\infty)$ is a Coxeter-plain sector. We argue by contradiction and assume there is $v_{i_1}$ such that $v_{i_1}$ is not of type III and $v_i$ is of type III for $i_0\le i<i_1$. We deduce as before that there exist two squares in $P$ containing $v_{i_1}$. The same holds for $v_{i_0-1}$. Then $v_{i_0-1}$ and $v_{i_1}$ give rise to two right-angled corners of a Coxeter region, and by the proof of Lemma~\ref{lem:two right angle}, this is impossible.

\subsection{Development of Coxeter rays}
\label{subsec:development of Coxeter rays}
Let $r$ be a Coxeter ray. Then $r$ is quasi-rectangular. Hence we identify $r$ with $[0,\infty)\times[0,1]$. Let $v$ be the vertex of $r$ with coordinates $(1,0)$ and let $\{v_i\}_{i=0}^{\infty}$ be consecutive vertices of $r$ on $[1,\infty)\times\{0\}$ with $v_0=v$.

\begin{lemma}\
	\label{lem:II or III}
	\begin{enumerate}
		\item If $v$ is of type III, then each vertex of $r$ on $[1,\infty)\times\{0\}$ is of type III.
		\item If $v$ is of type II, then each vertex of $r$ on $[1,\infty)\times\{0\}$ is either of type O or of type II.
	\end{enumerate}
\end{lemma}

\begin{proof}
	For (1) we use induction. Suppose $v_i$ is of type III. Let $\Pi_1$ and $\Pi_2$ be the two $2$--cells of $Q_R$ containing $\overline{v_iv_{i+1}}$ with $\Pi_1\subset r$. Since $\Pi_1$ and $\Pi_2$ are in different blocks, $v_{i+1}$ is of type II or type III. If $v_{i+1}$ is of type II, let $\Pi_3$ be the cell containing $v_{i+1}$ such that $\Pi_3\cap\Pi_1$ has at least one edge and $\Pi_3\neq\Pi_2$. We argue as in Section~\ref{subsec:development of plain rays} to deduce that $\Pi_3\subset r$ and that $\Pi_1$ and $\Pi_3$ are squares which intersect along an edge. However, this is impossible by the definition of Coxeter ray. Thus $v_{i+1}$ is of type III. (2) can be proved by using Lemma~\ref{lem:real2}
\end{proof}

Note that the two $2$--cells in $r$ containing $v_0$ are in different blocks, thus $v_0$ is either of type II or type III. By Lemma~\ref{lem:II or III}, we can assume (up to passing to a sub-ray of $r$) that either all vertices of $r$ on $[0,\infty)\times\{0\}$ are of type III, or all vertices with coordinate $(n,0)$ with $n\ge 0$ are of type II. We deduce from Lemma~\ref{lem:real2} and Lemma~\ref{lem:real3} that in either cases, each vertical edge of $r$ (i.e.\ edge of $r$ of form $\{n\}\times[0,1]$ with $n\ge 0$) is also an edge of $Q'$ and different vertical edges are the opposite sides of a flat rectangle in $Q'$. Thus there is a CAT(0) geodesic ray $\gamma\subset Q'$ such that it is based at $(0,1/2)$ and intersects each vertical edge of $r$ orthogonally at its midpoint. 

We think of $\gamma$ as being oriented from its base point to infinity and suppose $[0,\infty)\times\{0\}$ (resp.\ $[0,\infty)\times\{1\}$) is on the right (resp.\ left) side of $\gamma$ with respect to the orientation of $\gamma$ and $Q'$. $[0,\infty)\times\{0\}$ (resp.\ $[0,\infty)\times\{0\}$) is called the \emph{right border} (resp.\ \emph{left border}) of $r$.

Let $\ell_r\subset Q'$ be a CAT(0) geodesic line containing $\gamma$ and let $H'_r\subset Q'$ be the half-space bounded by $\ell_r$ on the right side of $\gamma$. By repeatedly applying Lemma~\ref{lem:parallel transport} below, we can replace $r$ by an appropriate Coxeter ray such that $H'_r$ is sufficient far away from the base point $x_0$ of $Q'$ so that $H'_r$ and its carrier in $Q_R$ are contained in the interior of $Q_R$.

\begin{lemma}
	\label{lem:parallel transport}
	Suppose $r$ is a Coxeter ray in $Q_R$ with its associated CAT(0) geodesic ray $\gamma$. There there is a Coxeter ray $r_1\subset Q_R$ on the right side of $r$ such that its associated CAT(0) geodesic $\gamma_1$ satisfies that a sub-ray of $\gamma_1$ and a sub-ray of $\gamma$ bound a flat strip in $Q'$ isometric to $[0,\infty)\times[0,1]$.
\end{lemma}

\begin{proof}
	If Lemma~\ref{lem:II or III} (2) holds, then there is another Coxeter ray $r_1$ such that $r_1\cap r=[0,\infty)\times\{0\}$. We define $\gamma_1$ to be the CAT(0) geodesic ray associated with $r_1$. If Lemma~\ref{lem:II or III} (1) holds, then there is a half Coxeter region $\h'\subset Q'$ such that a sub-ray of $\gamma'$ is a real border of $\h'$. As in the first paragraph of the proof of Lemma~\ref{lem:real ray border}, there exists a wall $\W\subset\h'$ such that $\W\cap \gamma'=\emptyset$ and $\W$ is an infinite ray parallel to $\gamma'$. We define $r_1$ to be the carrier (in $Q_R$) of an appropriate sub-ray of $\W$. Then $r_1$ is a Coxeter ray by Lemma~\ref{lem:ob}.
\end{proof}

Let $v=(1,0)\in r=[0,\infty)\times[0,1]$ be as before. We now consider the case when $v$ is of type III in more detail. Let $\h'_r\subset Q'$ be the $(v,\ell_r)$--half Coxeter region and let $\h_r\subset Q_R$ be the associated half Coxeter region (see Section~\ref{sec:half coxeter regions}). Note that $\h'_r$ contains a sub-ray of $\gamma$ by Lemma~\ref{lem:II or III} (2).

Let $K$ be a chamber of $\h'_r$. The angle between a real border and a fake border of $\h'_r$ must be a multiple of an (internal) angle of $K$, and the angle between two real borders of $\h'_r$ is equal to an angle of $K$ (Corollary~\ref{cor:local classification} (1)). Since a sub-ray of $\gamma$ is contained in the real border of $\h'_r$, a computation of angles yields that exactly one of the following possibilities holds for $\h'_r$:
\begin{enumerate}[label=(\alph*)]
	\item $\h'_r= H'_r$;
	\item $\h'_r$ has two parallel unbounded borders, one real and one fake, and one bounded border, in this case $d_H(\h'_r,\gamma)<\infty$;
	\item $\h'_r$ is a flat strip isometric to $(-\infty,\infty)\times[0,s]$;
	\item $\h'_r$ has two unbounded borders, one real and one fake, and they are not parallel, in this case $\h'_r$ has at most one bounded border and has finite Hausdorff distance from a flat sector in $\h'_r$.
\end{enumerate}

\begin{lemma}\
	\label{lem:case a and c}
	\begin{enumerate}
		\item If case (d) holds, then Theorem~\ref{thm:main1} holds with $S_1$ being a Coxeter sector.
		\item If case (a) holds, then Theorem~\ref{thm:main1} holds with $S_1$ being a Coxeter sector or a CCH of type II.
	\end{enumerate}
\end{lemma}

\begin{proof}
	We first verify assertion (1). Recall that each wall of $\h'_r$ has an orientation by Lemma~\ref{lem:ob}. A wall $\W$ of $\h'_r$ with $\W\nsubseteq \ell_r$ is \emph{positively oriented} if the half-space $H_\W$ induced by the orientation of $\W$ satisfies that $H_\W$ contains $\gamma$ up to a finite segment. Recall that $\h'_r$ has a simplicial structure where each $2$--simplex is a chamber of $\h'_r$. Let $\alpha$ be the value of the smallest angle of a chamber of $\h'_r$. A vertex $u$ of $\h'_r$ is \emph{special} if $u$ is contained in a chamber whose angle at $u$ is $\alpha$. 
	
	Let $\B_1$ and $\B_2$ be two borders of $\h'_r$ with $\B_1$ fake (hence $\B_1\subset \ell_r$). Let $\{\W_i\}_{i=1}^{\infty}$ be parallel walls in $\h'_r$ such that $\W_1=\B_1$ and $\W_i$ and $\W_{i+1}$ are not separated by other walls of $\h'_r$ for each $i\ge 1$. Note that for each $i$, $\W_i$ has a sub-ray $\W'_i$ such that the carrier of $\W'_i$ (in $Q_R$) is a Coxeter ray. This is clear when $i=1$ since $r$ is a Coxeter ray. The $i>1$ case follows from Lemma~\ref{lem:ob}. We assume without loss of generality that the starting point $w_i$ of each $\W'_i$ is a special vertex of $\h'_r$ and that each wall $\W$ of $\h'_r$ that intersects $\W_i$ transversely is positively oriented. By Lemma~\ref{lem:real ray border}, the carrier (in $Q_R$) of a sub-ray of $\B_2$ is a Coxeter ray. Hence there exists $i_0\ge 1$ such that $\W_i$ and $\W_j$ are oriented towards the same direction for any $i,j\ge i_0$. 
	
	Suppose $\W_i$ is negatively oriented for $i\ge i_0$. Let $\s'_i$ be the flat sector with angle $\alpha$ in $\h_r'$ based at $w_i$ with one side of $\s'_i$ being $\W'_i$ (such $\s'_i$ always exists and is a convex subcomplex of $\h'_r$). One readily verifies that for each $i\ge i_0$, the carrier of $\s'_i$ in $Q_R$ is a Coxeter sector. We define $S'_1$ in Theorem~\ref{thm:main1} to be any $\s'_i$ with $i\ge i_0$.
	
	It remains to consider the case where each $\W_i$ for $i\ge i_0$ is positively oriented. Though we can still consider the carrier of the sector $\s'_i$ as before, one boundary ray of this carrier does not have the correct orientation in order for it to be a Coxeter ray. So we will consider a different sector as below.
	
	 We first claim the Tits angle (cf.\ \cite[Definition II.9.4]{BridsonHaefliger1999}) between $\B_2$ and $\B_1$ is $\pi-\alpha$. Suppose the claim is not true (i.e.\ the Tits angle is $<\pi-\alpha$). Then there exists a special vertex $u\in\B_2$ of form $u=\W_{i_1}\cap\B_2$ for some $i\ge i_0$ such that $\W_u\cap\W'_1\neq\emptyset$ where $\W_u$ is the wall of $\h'_r$ containing $u$ and intersecting $\B_1$ in an angle $=\alpha$ (this can always be arranged if one chooses $i_1$ sufficiently large). We refer to Figure~\ref{f:coxetersector1}.
	\begin{figure}[h!]
		\centering
		\includegraphics[width=0.6\textwidth]{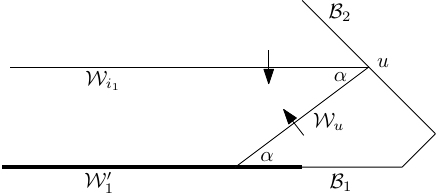}
		\caption{}
		\label{f:coxetersector1}
	\end{figure}
	Our choice implies that both $\W_u$ and $\W_{i_1}$ are positively oriented. Let $C_u$ be the $2$--cell in $\h_r$ dual to $u$. Then there is an orientation reserving vertex $y\in \partial C_u$ contained in the sector of $\h'_r$ bounded by $\W_{i_1}$ and $\W_u$. By Lemma~\ref{lem:piece orient}, $y$ is contained in an edge of $C_u$ that is dual to $\B_2$. This yields a contradiction since $\W_{i_1},\W_u$ and $\B_2$ are pairwise distinct walls. 
	
	Let $\s'_i$ be the flat sector with angle $\pi-\alpha$ in $\h_r'$ based at $w_i$ such that one boundary ray of $\s'_i$ is $\W'_i$. The existence of $\s'_i$ follows from the previous claim. By the argument in the previous paragraph, we can find orientation reversing vertex on the boundary of each $2$--cell $C$ of $\h_r$ dual to a special vertex in $\s'_i$ ($i\ge i_0$), hence the orientation of each edge of $C$ can be determined. Now one readily verify that the carrier of $\s'_i$ in $\h_r$ is a Coxeter sector and conclude the proof as before.
	
	It remains to prove assertion (2). Let $\mathfrak{W}$ be the collection of walls in $\h'_r$ intersecting $\W_1$ transversally. Let $\mathfrak{W}=\sqcup_{k=1}^m\mathfrak{W}_k$ be the decomposition into parallel classes. 
	
	Suppose there exists $k$ such that there is a pair of walls $\W,\W'$ in $\mathfrak{W}_k$ with opposite orientation. By the argument in Proposition~\ref{prop:half coxeter region contains Coxeter ray}, we know that $\W$ has a sub-ray whose carrier in $Q_R$ is a Coxeter ray. Thus there exists $i_0\ge 0$ such that all $\W_i$ are oriented towards the same direction for $i\ge i_0$, then we argue as before to produce a Coxeter sector.
	
	Suppose for each $k$, each pair of walls in $\mathfrak{W}_k$ are oriented towards the same direction. Recall that $\W_1$ contains a ray $\W'_1$ whose carrier in $Q_R$ is a Coxeter ray. Since each $\mathfrak{W}_k$ contains at least one wall that intersects $\W'_1$, either all walls in $\mathfrak{W}$ are positively oriented, or they are all negatively oriented. Hence $\W_1$ is a Coxeter line. If there exists $i_0\ge 0$ such that all $\W_i$ are oriented towards the same direction for $i\ge i_0$, then we argue as before. If such $i_0$ does not exist, then $\h_r$ is a CCH of type II.
\end{proof}
The angle $\alpha$ defined in the proof of Lemma~\ref{lem:case a and c} satisfies $\alpha\ge \pi/6$. Hence the associated CAT(0) sector has angle $\ge \pi/6$.

Now we are ready to prove Theorem~\ref{thm:main1} in the case when $r$ is a Coxeter ray. Starting from $r\cong[0,\infty)\times[0,1]$. If Lemma~\ref{lem:II or III} (2) holds, then there is another Coxeter ray $r_1$ such that $r_1\cup r$ is a quasi-rectangular region of $Q_R$, which we identify with $[0,\infty)\times[-1,1]$. By Lemma~\ref{lem:concat} (2), $r_1$ is contained in the half-space $H'_r$ defined before Lemma~\ref{lem:parallel transport}. Hence Lemma~\ref{lem:star homeo} holds for vertices of $r_1$ and we can apply Lemma~\ref{lem:II or III} to $[0,\infty)\times\{-1\}$. Then either we are able to repeat this argument for infinitely many times to produce a Coxeter-plain sector with one of its boundary ray being $r$, or after finitely many steps we find a vertex of type III on the boundary $[0,\infty)\times\{-m\}$ of the quasi-rectangular region we produced. In the latter case we assume without loss of generality that $m=0$, which leads to (a), (b), (c) and (d) discussed before. By Lemma~\ref{lem:case a and c}, it remains to consider cases (b) and (c).

Suppose (b) or (c) holds for $\h'_r$. Let $\B_1$ and $\B_2$ be two parallel unbounded borders of $\h'_r$ such that $\B_2$ is not real. By Lemma~\ref{lem:real ray border}, there exists a sub-ray of $\B_2$ such that its carrier $r_2$ in $Q_R$ is a Coxeter ray and $d_H(r,r_2)<\infty$. By the previous discussion, either $r_2$ is the boundary ray of a Coxeter-plain sector, or we find a Coxeter ray $r_3$ on the right side of $r_2$ such that $r_3$ has a vertex of type III on its right border. Note that $r_2\neq r_3$ by our choice of $r_2$. Moreover, the left border of $r_3$ satisfies Lemma~\ref{lem:II or III} (2). Now we define $\h'_{r_3}$, $H'_{r_3}$, $\gamma_3$ and $\ell_{r_3}$ as before. Let $\C'_{r_3}$ be the Coxeter region containing $\h'_{r_3}$. Since the left border of $r_3$ satisfies Lemma~\ref{lem:II or III} (2), the intersection of a sufficiently small ball around each interior point of $\gamma_3$ with $\C'_{r_3}$ satisfies Corollary~\ref{cor:local classification} (2). Now we consider the connected component $P$ of $\partial \C'_{r_3}$ that contains $\gamma_3$. By Corollary~\ref{cor:local classification} and an induction argument, we know either $P=\ell_{r_3}$, or $P$ travels into the interior of $H'_{r_3}$ at some point and remains there. Thus $\C'_{r_3}\subset H'_{r_3}$. It follows that $\C'_{r_3}=\h'_{r_3}$. We claim $\h'_{r_3}$ can not satisfy (b). By $\C'_{r_3}=\h'_{r_3}$ and Corollary~\ref{cor:local classification} (1), the only way $\h'_{r_3}$ can satisfy (b) is that the bounded border of $\h'_{r_3}$ is orthogonal to both of its unbounded borders, however, this can be ruled out by Lemma~\ref{lem:two right angle}.

From the above discussion, we deduce that at least one of following holds:
\begin{enumerate}
	\item there is a Coxeter sector or a Coxeter-plain sector $S$ on the right side of $r$ such that one boundary ray $r'$ of $S$ satisfies $d_H(r,r')<\infty$;
	\item there is a Coxeter line $L$ on the right side of $r$ such that $L$ coarsely contains $r$ and $L$ bounds a CCH of type II;
	\item there are subcomplexes $\{U_i\}_{i=1}^{\infty}$ of $Q_R$ such that
	\begin{enumerate}
		\item each $U_i$ coarsely contains $r$ and is on the right side of $r$;
		\item each $U_i$ is either a $(-\infty,\infty)\times [0,m_i]$--quasi-rectangular region ($m_i>0$) or a Coxeter region such that the associated Coxeter region $U'_i$ in $Q'$ is isometric to a flat strip $(-\infty,\infty)\times [0,m_i]$; moreover, there does not exist $i$ such that both $U_i$ and $U_{i+1}$ are quasi-rectangular;
		\item $U_i\cap U_{i+1}$ is a boundary line of both $U_i$ and $U_{i+1}$, and each vertex in $U_i\cap U_{i+1}$ is either of type O or of type II;
		\item there does not exist $i_0$ such that all $U_i$ are quasi-rectangular regions for $i\ge i_0$ (otherwise we will be in case (1) with $S$ being a Coxeter-plain region).
	\end{enumerate}
\end{enumerate}

To finish the proof of Theorem~\ref{thm:main1} in the Coxeter ray case, it remains to show $\cup_{i=1}^{\infty}U_i$ is a CCH of type I in case (3). Assume without loss of generality that $U_1$ is a Coxeter region. It suffices to prove there are no orientation reversing vertices on $U_1\cap U_2$, since one can deduce from this, Lemma~\ref{lem:ob} and Remark~\ref{rmk:real2} that there are no orientation reversing vertices on $U_i\cap U_{i+1}$ for each $i$, which implies the claim. By (3b) and (3c), $U_i\cap U_{i+1}$ is the outer border of a Coxeter region for each $i$. Hence by Lemma~\ref{lem:at most one reserving point}, there is at most one orientation reversing vertex on $U_i\cap U_{i+1}$. Suppose there is an orientation reversing vertex $v\in U_1\cap U_2$. We first look at the case $U_2$ is quasi-rectangular. Then $U_3$ is a Coxeter region and there is exactly one orientation reversing vertex in $U_2\cap U_3$ by Remark~\ref{rmk:real2}. Let $\W_1$ and $\W_2$ be two parallel boundary walls of the Coxeter region $U'_3\subset Q'$ associated with $U_3$. Suppose $2n$ is the largest possible number of edges on the boundary of a $2$--cell in $U_3$. Choose a vertex $v\in\W_1$ such that it is dual to a $2$--cell with $2n$ edges in $U_3$. Consider the zig-zag line $\ell_v$ containing $v$ as in Figure~\ref{f:orient60} such that $\ell_v$ is a union of straight line segments from a point in $\W_1$ to a point in $\W_2$ such that each segment has angle $\pi/n$ with $\W_1$. We deduce that each segment of $\ell_v$ is contained in a wall of $U'_3$. Let $\mathfrak{W}$ be the collection of walls which contain a segment of $\ell_v$. 
\begin{figure}[h!]
	\centering
	\includegraphics[width=1\textwidth]{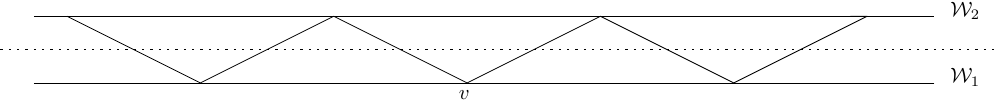}
	\caption{}
	\label{f:orient60}
\end{figure}
We assign a total order on $\mathfrak{W}$ by looking at intersection points of elements in $\mathfrak{W}$ along a straight line $\ell$ (see the dotted line in Figure~\ref{f:orient60}) between $\W_1$ and $\W_2$. Since there is exactly one orientation reversing vertex in $U_2\cap U_3$, there is exactly one pair of walls $\W,\W'$ in $\mathfrak{W}$ which are adjacent with respect to the total order such that one is oriented towards left and one is oriented towards right. Note that $u=\W\cap\W'$ is either in $\W_1$ or $\W_2$. Let $C_u$ be the $2$--cell in $U_3$ dual to $u$. By Lemma~\ref{lem:piece orient}, there are no orientation reversing vertices in $\partial C_u\cap\partial U_3$. Thus either both $\W$ and $\W'$ are oriented towards left, or they are oriented towards right. This leads to a contradiction. The case $U_2$ is a Coxeter region is similar.

Each CCH in $Q_R$ contains a CAT(0) half-flat of $Q'$ which is Hausdorff close this CCH. To see this, note that a CCH has infinitely many Coxeter lines, each of them contains a CAT(0) geodesic line of $Q'$ by a similar argument as before. A pair of such CAT(0) geodesic lines bounds a flat strip, and the union of these flat strips gives the half-flat as required.

\subsection{Development of diamond rays} 
\label{subsec:development of diamont rays}
Let $r=\cup_{i=0}^{\infty}C_i$ be a diamond ray in $Q_R$ with $C_i$ satisfying Definition~\ref{def:diamond line}. There is a CAT(0) geodesic ray $\gamma\subset r$ defined as follows. For each $C_i$, let $h_i$ be the CAT(0) geodesic segment in $Q'$ connecting the two tips of $C_i$. By Lemma~\ref{lem:real1}, Lemma~\ref{lem:exactly pi} and Lemma~\ref{lem:exactly pi1}, $\gamma=\cup_{i=1}^{\infty}$ is a CAT(0) geodesic ray in $Q'$. As before we think of $\gamma$ as being oriented from its base vertex to infinity, and we talk about right side or left side of $r$ and $\gamma$ with respect to the orientation of $\gamma$ and $Q'$.

\begin{lemma}
	\label{lem:right shift}
	Suppose $r$ is contained in the interior of $Q_R$. Then there is a diamond ray $r'$ such that $r\cap r'$ is contained in the right boundary of $r$.
\end{lemma}

\begin{proof}
	Let $v_i=C_i\cap C_{i+1}$ for $i\ge 1$. Since $C_i$ and $C_{i+1}$ are in the same block, but they only intersect along one vertex, we know from Lemma~\ref{lem:real1} and Lemma~\ref{lem:real2} that $v_i$ is of type I (see Table~\ref{t:flat} on page \pageref{t:flat}) for each $i$. Let $C'_i$ be the $2$--cell containing $v_i$ on the right side of $r$. We claim $C'_i\cap C'_{i+1}$ is a tip of both $C'_i$ and $C'_{i+1}$ for $i\ge 0$. Then $r'=\cup_{i=0}^{\infty}C'_i$ is a diamond ray by this claim.
	
	Now we prove the claim. By Lemma~\ref{lem:real1}, $C'_{i}\cap (C_i\cup C_{i+1})$ is a half of $\partial C'_i$. Let $v'_i$ be the endpoint of the path $C'_i\cap C_{i+1}$ such that $v'_i\neq v_i$. Then $v'_i$ is a tip of $C'_i$ by Lemma~\ref{cor:connected intersection} (2). Since $C'_i$ and $C_{i+1}$ are in the same block, $v'_i$ is either of type I or type II. Let $C''_{i+1}$ be the $2$--cell different from $C_{i+1},C'_i$ such that $v'_i\in C''_{i+1}$ and $C''_{i+1}\cap C_{i+1}$ contains at least an edge. By Lemma~\ref{lem:real1} and Lemma~\ref{lem:real2}, $C''_{i+1}\cup C'_i$ is a point and $C_{i+1}\cap (C''_{i+1}\cup C'_i)$ is a half of $\partial C_{i+1}$. Thus $v_{i+1}\in C''_{i+1}$. It follows that $C''_{i+1}=C'_{i+1}$ since $v_{i+1}$ is of type I. Now we know $C''_{i+1}$, $C'_i$ and $C_{i+1}$ are in the same block. Thus $v'_i$ is of type I, and $C'_{i+1}$ and $C'_i$ both contain $v'_i$ as a tip. Hence the claim follows. 
\end{proof}

In the above proof, by reading off the label of edges in the path $C_{i+1}\cap C'_i$ traveling from $v_i$ to $v'_i$, we obtain a word, and this word does not depend on $i$. We denote this word and its length by $\delta_{r,r'}$ and $|\delta_{r,r'}|$. Let $2n$ be the number of edges in the boundary of $C_0$. Then $1\le |\delta_{r,r'}|\le n-1$. Suppose $|\delta_{r,r'}|>1$. Let $v'_{-1}$ be the tip of $C'_0$ other than $v'_0$. Since $C_0$ and $C'_0$ are in the same block and $C_0\cap C'_0$ has $n-|\delta_{r,r'}|<n-1$ edges, $v'_{-1}$ can not be of type II by Lemma~\ref{lem:real2} (2) and (3). Thus $v'_{-1}$ is of type I. Let $C'_{-1}$ be the $2$--cell such that $C'_{-1}\neq C'_0$ and $C'_{-1}$ contains $v'_{-1}$ as a tip. Then $\cup_{i=-1}^{\infty}C'_i$ is a diamond ray, which is defined to be the \emph{right shift} of $r$. If $|\delta_{r,r'}|=1$, then the \emph{right shift} of $r$ is defined to be $\cup_{i=0}^{\infty}C'_i$. 

Let $\ell_r\subset Q'$ be a CAT(0) geodesic line that contains the geodesic ray $\gamma$ associated with $r$ and let $H'_r$ be the half-space bounded by $\ell_r$ on the right side of $\ell_r$. By applying finitely many right shifts to $r$, we assume without loss of generality that $H'_r$ is sufficiently far away from the base point $x_0$ of $Q'$ such that $H'_r$ and its carrier in $Q_R$ are contained in the interior of $Q_R$.

\begin{lemma}
	\label{lem:extend to diamond line1}
	Let $r'$ be a right shift of $r$. If $1<|\delta_{r,r'}|<n$, then we can extend $r$ and $r'$ to diamond lines $L$ and $L'$ in $Q_R$ respectively such that $L\cap L'$ is a boundary line of both $L$ and $L'$.
\end{lemma}

\begin{proof}
	Let $r=r_0=\cup_{i=0}^{\infty} C_i$, $r'=r'_{-1}=\cup_{i=-1}^{\infty} C'_i$. Let $v_{-1}$ be the tip of $C_0$ with $v_{-1}\neq v_0$. Then $v_{-1}$ is an endpoint of $C_0\cap C'_{-1}$, and the number of edges in $C_0\cap C'_{-1}=|\delta_{r,r'}|$ is $<n$. Thus $v_{-1}$ is type I. Let $C_{-1}\subset Q_R$ be the $2$--cell containing $v_{-1}$ as a tip such that $C_{-1}\neq C_0$. Now we can enlarge $r$ to a new diamond ray $\cup_{i=-1}^{\infty} C_i$. Note that $r_{-1}=\cup_{i=-1}^{\infty} C_i$ is contained in the interior of $Q_R$ by our choice of $H'_r$. Note that $\delta_{r_{-1},r'_{-1}}=\delta_{r,r'}>1$. Thus we can define the right shift $r'_{-2}=\cup_{i=-2}^{\infty} C'_i$ of $r_{-1}$, then enlarge $r_{-1}$ to $r_{-2}$ as before. The lemma follows by repeating this braiding process.
\end{proof}

Now we define a sequence of diamond rays $\{r_i\}_{i=0}^{\infty}$ such that $r_0=r$, and $r_{i}$ is the right shift of $r_{i-1}$. By our choice of $H'_r$, it is possible to define these diamond rays.

\begin{lemma}
	\label{lem:extend to diamond line2}
	If there exists $i_0$ such that $|\delta_{r_{i_0-1},r_{i_0}}|=1$ and $|\delta_{r_{i_0},r_{i_0+1}}|=n-1$, then there exists a diamond line $L$ containing $r_{i_0}$.
\end{lemma}

\begin{proof}
	We assume without loss of generality $i_0=1$. Let $r_0=\cup_{i=0}^{\infty} C_i$, $r_1=\cup_{i=0}^{\infty}C'_i$ and $r_2=\cup_{i=-1}^{\infty}C''_i$. Let $v'_0\in C'_0$. Then $v'_0$ is an endpoint of both $C'_0\cap C_0$ and $C'_0\cap C''_{-1}$ by Lemma~\ref{cor:connected intersection} (2). By $|\delta_{r_1,r_2}|=n-1$, we know $C'_0\cap C''_{-1}$ has $n-1$ edges. By $|\delta_{r_0,r_1}|=1$, we know $C_1\cap C'_0$ has one edge, hence $C_0\cap C'_0$ has $n-1$ edges. Now the lemma follows from the same argument in Lemma~\ref{lem:diamond} (note that $C_0$, $C'_0$ and $C''_{-1}$ play the role of $C_1,\Pi_1$ and $C'_1$ in Figure~\ref{f:development1} respectively).
\end{proof}
\medskip

Now we are ready to prove Theorem~\ref{thm:main1} in the case $r$ is a diamond line. Note that this is the 
only case left to be verified.

If there exists $i_0$ such that $|\delta_{r_{i_0},r_{i_0+1}}|=n-1$ for all $i\ge i_0$, or $|\delta_{r_{i_0},r_{i_0+1}}|=1$ for all $i\ge i_0$, then $\cup_{i=i_0}^{\infty} r_i$ is a diamond-plain sector.

We assume such $i_0$ does not exist. If there exists $i_1$ such that $1<|\delta_{r_{i_1},r_{i_1+1}}|<n$, then we can extend $r_{i_1}$ to a diamond line $L_{i_1}$ by Lemma~\ref{lem:extend to diamond line1}. Hence each $r_i$ for $i\ge i_1$ can be extended to a diamond line $L_i$ by the argument in Lemma~\ref{lem:right shift}. We deduce from the non-existence of $i_0$ that $\cup_{i=i_1}^{\infty} L_i$ is a DCH. If for all $i$, $|\delta_{r_{i_1},r_{i_1+1}}|$ is either $1$ or $n-1$, then there exists $i_2$ such that the assumption of Lemma~\ref{lem:extend to diamond line2} is satisfied. So we can extend $r_{i_2}$ to a diamond line $L_{i_2}$ and deduce as before that $L_{i_2}$ bounds a DCH. Note that each DCH in $Q_R$ contains a CAT(0) half-flat of bounded Hausdorff distance from this DCH. Indeed, each $L_i$ contains a CAT(0) geodesic line $\ell_i$, and $\ell_i$ and $\ell_{i+1}$ bound a flat strip isometric to $[0,m_i]\times \mathbb E^1$, then the union of these flat strips gives the half-flat. 
\medskip 

In Sections~\ref{subsec:development of plain rays}, \ref{subsec:development of Coxeter rays}, and \ref{subsec:development of diamont rays} we treated all the cases for the singular ray $r$. Therefore, Theorem~\ref{thm:main1} follows.

\section{Quasi-isometry invariants for $2$--dimensional Artin groups}
\label{sec:application}
\subsection{Preservation of completions of atomic sectors}
\begin{definition}
	\label{def:}
Let $\{S_j\}_{j=1}^n$ be a collection of subsets of $\Xa_\Gamma$. We say $\{S_j\}_{j=1}^n$ \emph{touch along discs} if there exist $L$, $A$ and a sequence of $(L,A)$--quasi-isometric embeddings $f_i\colon D_i\to \Xa_\Gamma$ such that each $D_i$ is a disc in the Euclidean plane with radius $m_i\to \infty$ and $f(D_i)\subset \cap_{j=1}^nN_A(S_j)$.
\end{definition}
Here $\Xa_\Gamma$ has a piecewise Euclidean structure inherited from $X_\Gamma$, which gives rise to a length metric on $\Xa_\Gamma$ that is quasi-isometric to $A_\Gamma$. We use $N_A(X)$ to denote the $A$ neighborhood of $X$. 

\begin{lemma}
	\label{lem:sector touch}
	Let $S_1$ and $S_2$ be atomic sectors. If they touch along discs, then $d_H(S_1\cap S_2,S_1)<\infty$ and $d_H(S_1\cap S_2,S_2)<\infty$.
\end{lemma}

\begin{proof}
By modifying the map $f_i$ in Definition~\ref{def:}, there are $L'$ and $A'$ independent of $i$ such that for each $D_i$, there are continuous $(L',A')$--quasi-isometric embeddings $g_i\colon D_i\to S_1$ and $h_i\colon D_i\to S_2$ such that for any $x\in D_i$, $d_{\Xa_\Gamma}(g_i(x),h_i(x))<A'$. By Lemma~\ref{lem:atomic qi embedding}, we can assume $g_i$ and $h_i$ are also $(L',A')$--quasi-isometric embeddings with respect to the intrinsic metric on $S_1$ and $S_2$. For the convenience of later computation, we assume $S_1$ and $S_2$ are tilings of flat Euclidean sectors (this can always be arranged up to quasi-isometries). Let $\sigma$ be a singular chain representing the fundamental class in $H_2(D_i,\partial D_i)$. Let $\alpha_i=(g_i)_{\ast}\sigma$ and $\beta_i=(h_i)_{\ast}\sigma$. 

We claim there exists $M>1$ independent of $i$ such that $[\partial \alpha_i]$ is nontrivial in $H_1(S_1\setminus B_{S_1}(g_i(c_i),\frac{n_i}{M}))$, where $B_{S_1}(g_i(c_i),\frac{n_i}{M})$ denotes the ball in $S_1$ centered at $g_i(c_i)$ and has radius $\frac{n_i}{M}$ with respect to the flat metric on $S_1$. To see the claim, we construct a continuous map $g^{-1}_i\colon \im g_i\to D_i$ such that $g^{-1}\circ g$ and $g\circ g^{-1}$ are identity maps up to a bounded error which depends only on $L'$ and $A'$. Hence $g^{-1}\circ g$ and $g\circ g^{-1}$ are homotopic to identity where the homotopy moves points by a distance $\le N=N(L',A')$. Now we assume without loss of generality that $n_i\gg L'$, $n_i\gg A'$ and $n_i\gg N$ for all $i$. Then the claim follows. It follows from the claim that the image of $[\alpha_i]$ is non-trivial under $H_2(S_1,S_1\setminus B_{S_1}(g_i(c_i),\frac{n_i}{M}))\to H_2(S_1,S_1\setminus\{p\})\to H_2(\Xa_\Gamma,\Xa_\Gamma\setminus\{p\})$ (recall that $S_1$ is embedded in $\Xa_\Gamma$ by Lemma~\ref{lem:atomic embedding}, hence the two maps here are both injective). Hence the support set $S_{[\alpha]}$ of the homology class $[\alpha]\in H_2(\Xa_\Gamma,g_i(\partial D_i))$ satisfies $B_{S_1}(g_i(c_i),\frac{n_i}{M})\subset S_{[\alpha]}$. 

Since $d_{\Xa_\Gamma}(g_i(x),h_i(x))<A'$ for each $x\in\partial D_i$ and $\Xa_\Gamma$ is uniformly contractible, there exists a continuous map $F\colon \partial D_i\times [0,1]\to \Xa_\Gamma$ such that $F|_{\partial D_i\times\{0\}}=g_i|_{\partial D}$, $F|_{\partial D_i\times\{1\}}=h_i|_{\partial D}$ and $d_H(\im F, g_i(\partial D_i))<A_1$ for $A_1$ independent of $i$. Then $F$ gives rise to a $2$--chain $\gamma_i$ such that $\partial \gamma_i=\partial\alpha_i-\partial\beta_i$ and $d_H(\im \gamma_i,\im \partial\alpha_i)<A_2$ for $A_2$ independent of $i$. Then $\partial (\gamma_i+\beta_i)=\partial\alpha_i$. Since $H_2(\Xa_\Gamma)$ is trivial, there exists a $3$--chain $\rho$ such that $\partial\rho=\gamma_i+\beta_i-\alpha_i$. Thus $[\gamma_i+\beta_i]=[\alpha_i]$ in $H_2(\Xa_\Gamma,g_i(\partial D_i))$ and $S_{[\gamma_i+\beta_i]}=S_{[\alpha_i]}\supset B_{S_1}(g_i(c_i),\frac{n_i}{M})$. By assuming $n_i\gg A_2$ and possibly enlarging $M$, we have $B_{S_1}(g_i(c_i),\frac{n_i}{M})\cap \im\gamma_i=\emptyset$. Thus $[\beta_i]$ is nontrivial in $H_2(\Xa_\Gamma,\Xa_\Gamma\setminus\{p\})$ for any $p\in B_{S_1}(g_i(c_i),\frac{n_i}{M})$. Hence $B_{S_1}(g_i(c_i),\frac{n_i}{M})\subset \im\beta_i\subset S_2$. Thus $S_1\cap S_2$ contains larger and larger discs and the lemma follows from Lemma~\ref{lem:atomic intersection}.
\end{proof}

\begin{lemma}
	\label{lem:sector quasiflat touch}
Let $S$ be an atomic sector and let $Q\subset\Xa_\Gamma$ be a $2$--dimensional quasiflat. Let $\{S_i\}_{i=1}^{n}$ and $Q_R$ be as in Theorem~\ref{thm:main2}. If $S$ and $Q$ touch along discs, then there exists an $S_i$ such that $d_H(S\cap S_i,S)<\infty$ and $d_H(S\cap S_i,S_i)<\infty$. In particular, $S\subset N_A(Q)$ for some $A$.
\end{lemma}

\begin{proof}
Let $q'\colon Q'\to X_\Gamma$ be the CAT(0) approximation of $Q$ as in Theorem~\ref{thm:MSquasiflats}. Let $D$ be the image of an $(L,A)$--quasi-isometric embedding from a disc of radius $m$ to $Q$. Since $q'$ is Lipschitz, $(q')^{-1}(D)$ contains a ball of radius $\frac{m}{M}$ in $Q'$ with $M$ depending on $L,A$. By Theorem~\ref{thm:main1} and Theorem~\ref{thm:main2}, there are disjoint flat sectors $\{S'_i\}_{i=1}^n$ in $Q'$ such that $d_H(Q',\cup_{i=1}^{n}S'_i)<\infty$ and $d_H(S_i,S'_i)<\infty$. Since $Q'$ is a CAT(0) plane, for sufficiently large $m$, any ball of radius $\frac{m}{M}$ in $Q'$ (with respect to the CAT(0) metric on $Q$) must contain a ball of radius $\frac{m}{M'}$ in one of the $S'_i$ for some $M'>M$ (with respect to the CAT(0) metric on $S'_i$). By Lemma~\ref{lem:atomic qi embedding}, the CAT(0) metric on $S'_i$ and the metric on $S_i$ induced from $\Xa_\Gamma$ are quasi-isometric. Let $m\to\infty$. By passing to a subsequence we know $S$ and at least one of the $S_i$ touch along discs. Then the lemma follows from Lemma~\ref{lem:sector touch}.
\end{proof}

\begin{lemma}
	\label{lem:intersection of quasiflats contain a sector}
	Let $\{Q_i\}_{i=1}^{n}$ be a finite collection of $2$--dimensional quasiflats. If they touch along discs, then there exists an atomic sector $S$ and $A>0$ such that $S\subset \cap_{i=1}^{n}N_A(Q_i)$.
\end{lemma}

\begin{proof}
By Theorem~\ref{thm:main2}, we can assume each $Q_i$ is a union of finitely many atomic sectors. By the proof of Lemma~\ref{lem:sector quasiflat touch}, we know that for each $(L,A)$--quasi-isometric embedding $D_m\to \cap_{i=1}^{n}N_A(Q_i)$ (where $D_m$ denotes a disc of radius $m$ in the Euclidean plane), there exist constants $L',A',M>1$ independent of $m$, atomic sector $S_i$ of $Q_i$ for each $i$ and an $(L',A')$--quasi-isometric embedding $D_{\frac{m}{M}}\to \cap_{i=1}^{n}N_{A'}(S_i)$. Since each $Q_i$ has only finitely many atomic sectors, we can assume $S_i$ does not change as $m\to \infty$ by passing to a subsequence. Thus the collection $\{S_i\}_{i=1}^{n}$ touch along discs and we are done by Lemma~\ref{lem:sector touch}.
\end{proof}

\begin{lemma}
	\label{lem:completion containment}
Let $S$ be an atomic sector and let $Q$ be a $2$--dimensional quasiflat. Let $Q_R$ and $\{S_i\}_{i=1}^{n}$ be as in Theorem~\ref{thm:main2}. If there exists $A>0$ such that $S\subset N_A(Q)$, then the completion $\bar S$ of $S$ satisfies that $d_H(\bar S\cap Q_R,\bar S)<\infty$. In particular, $\bar S\subset N_{A'}(Q)$ for some $A'>0$.	
\end{lemma}

\begin{proof}
By Lemma~\ref{lem:atomic is intersection of quasiflats}, it suffices to consider the case when $S$ is a diamond-plain sector or a plain sector.  By Lemma~\ref{lem:sector touch}, we can assume $S=S_1$. 

First we assume $S_1$ is a diamond-plain sector. Suppose $S_1$ is bounded by a plain ray $r_1$ and a diamond ray $r'_1$. Then the discussion in Section~\ref{subsec:development of plain rays} implies that we can assume $S_2$ has a boundary ray $r_2$ such that 
\begin{enumerate}
	\item $r_2$ is a plain ray;
	\item $r_1$ and $r_2$ bound a (possibly degenerate) quasi-rectangular region $P$, in particular, $d_H(r_1,r_2)<\infty$.
\end{enumerate}
Also we know from Section~\ref{subsec:development of plain rays} that the possibilities for $S_2$ are diamond-plain sector, Coxeter-plain sector, plain sector and PCH. Let $a\in\Gamma$ be the label of edges in $r_1$ and $\overline{ab}$ be the defining edge of the block $B$ that contains $S$. The interesting case is that every vertex in the odd component (cf.\ Section~\ref{subsec:completion}) containing $a$ is a leaf. Thus either $a$ is a leaf and $\overline{ab}$ is labeled by an even number $>2$, or $\overline{ab}$ has odd label and $\overline{ab}$ is a connected component of $\Gamma$. In both cases we can deduce the following
\begin{enumerate}
	\item $P\subset B$;
	\item $r_2$ is single labeled by either $a$ or $b$;
	\item $S_2$ is a diamond-plain sector in $B$.
\end{enumerate}
Thus a sub-half-flat of $\bar{S}$ is contained in $Q_R$.

Now assume $S_1$ is a plain sector. Write $S_1=r_1\times r'_1$, where $r_1$ and $r'_1$ are boundary rays of $S_1$. Let $V_1$ and $V'_1$ be the collection of labels of edges in $r_1$ and $r'_1$ respectively. First we consider the case when $V_1$ is good (in the sense of Section~\ref{subsec:completion}) and $V_2$ is not good. Since we always have $V_1\subset V^{\perp}_2$, $V_1$ has to be a singleton. $V_2$ either satisfies $|V_2|\ge 2$ and $V^{\perp}_2=V_1$, or $|V_2|=1$ and $V_2$ is a leaf. In the latter case we also have $V^{\perp}_2=V_1$. 

Let $S_2,r_2$ and $P$ be as before. We claim that if $P$ is non-degenerate, then $P$ is made of squares. Now we prove the claim. In the case $|V_2|\ge 2$, there exists a vertex $v\in r_1$ such that the two edges in $r_1$ containing $v$ have different labels. Since the two squares of $S_1$ containing $v$ are in different blocks, $v$ is either of type II, or is a $\square$--vertex. By Lemma~\ref{lem:real2}, the other two cells containing $v$ are also squares. Now we walk along $r_1$ from $v$ to other vertices to prove inductively that each vertex in $r_1$ which is not the endpoint of $r_1$ is contained in four squares. This shows that the first layer of $P$ is made of squares and we can repeat this argument to show $P$ is made of squares. In the case $|V_2|=1$, if a vertex $v\in r_1$ is contained in a non-square cell, then $V_2$ is adjacent to another vertex in $\Gamma$ via an edge whose label is not 2. This contradicts that $V^{\perp}_2=V_1$ and $V_2$ is a leaf, hence the first layer of $P$ is made of squares. Now we conclude the proof as before.

The above argument also shows that $S_2$ is made of squares, hence $S_2$ is a plain sector. And $V^{\perp}_2=V_1$ determines what this plain sector is. Thus a sub-half-flat of $\bar{S}$ is contained in $Q_R$. The remaining case that both $V_1$ and $V'_1$ are not good is similar.
\end{proof}

\begin{corollary}
	\label{cor:qi and completion}
Suppose that $A_{\Gamma_1}$ and $A_{\Gamma_2}$ are two-dimensional Artin groups. Let $q\colon \Xa_{\Gamma_1}\to \Xa_{\Gamma_2}$ be a quasi-isometry. Then for any atomic sector $S_1$ in $\Xa_{\Gamma_1}$, there exists an atomic sector $S_2$ in $\Xa_{\Gamma_2}$ such that $d_H(q(\bar S_1),\bar S_2)<\infty$, where $\bar S_i$ denotes the completion of $S_i$.
\end{corollary}

\begin{proof}
First we show for each $S_1$, there exists $S_2$ such that $\bar S_2\subset N_{A_1}(q(\bar S_1))$ for some $A_1>0$. By Corollary~\ref{cor:completion of atomic}, $\bar{S}_1$ is the coarse intersection of finitely many quasiflats $\{Q_i\}_{i=1}^n$. Thus $q(\bar S_1)$ is the coarse intersection of $\{q(Q_i)\}_{i=1}^n$. In particular, $\{q(Q_i)\}_{i=1}^n$ touch along discs. Lemma~\ref{lem:intersection of quasiflats contain a sector} implies there is an atomic sector $S_2\subset\cap_{i=1}^n N_A(q(Q_i))$ for some $A$. By Lemma~\ref{lem:completion containment}, $\bar S_2\subset\cap_{i=1}^n N_{A'}(q(Q_i))$ for some $A'$. Thus $\bar S_2\subset N_{A_1}(q(\bar S_1))$ for some $A_1$. Let $q^{-1}$ be a quasi-isometric inverse of $q$. As before, we know there exists an atomic sector $S_3$ of $\Xa_{\Gamma_1}$ such that $\bar S_3\subset N_{A_2}(q^{-1}(\bar S_2))$. Thus $\bar S_3\subset N_{A_3}(\bar S_1)$ for some $A_3>0$. Note that $\bar{S_1}$ is a union of finitely many atomic sectors $\{T_i\}$ such that the completion of any $T_i$ is Hausdorff close to $\bar{S}_1$. Since $S_3\subset\bar{S}_3\subset N_{A_3}(\bar S_1)$, we know from the argument in Lemma~\ref{lem:sector quasiflat touch} that $S_3$ and one of $\{T_i\}$ (say $T_1$) touch along discs. Thus we can assume $S_3=T_1$ by Lemma~\ref{lem:sector touch}. Hence $\bar S_3=\bar T_1$ has finite Hausdorff distance from $\bar S_1$. It follows that $d_H(q(\bar S_1),\bar S_2)<\infty$.
\end{proof}

\subsection{Preservation of stable lines}
\begin{definition}
	\label{def:stable line}
A \emph{stable line} in $\Xa_\Gamma$ is one of the following objects:
\begin{enumerate}
	\item a diamond line;
	\item a Coxeter line;
	\item a single labeled plain line such that its label $a$ satisfies that all edges of $\Gamma$ containing $a$ are labeled by $2$, $|a^{\perp}|\ge 2$ and $(a^{\perp})^{\perp}=a$;
	\item a single labeled plain line such that its label $a$ is not a leaf and there is an edge containing $a$ with label $\ge 3$;
	\item a single labeled plain line such that its label $a$ is a leaf, $a$ is connected to a vertex $b$ by an odd-labeled edge and $b$ is not a leaf.
\end{enumerate}
\end{definition}
Recall that $a^{\perp}$ is the collection of vertices in $\Gamma$ that are adjacent to $a$ along an edge labeled by 2.

Definition~\ref{def:stable line} is motivated by looking for a $\mathbb Z$--subgroup which is the $\mathbb Z$ factor of some $F_k\times \mathbb Z$ subgroup ($k\ge 2$), moreover, we do not want this $F_k\times \mathbb Z$ subgroup to be further contained in an $F_k\times F_{k'}$ subgroup ($k'\ge 2$). See the conjecture below for a precise formulation, which is motivated by properties of Dehn twists in mapping class groups of surfaces. 

\begin{conj}
	\label{c:Zstable}
Suppose $A_\Gamma$ has dimension $\le 2$. A $\mathbb Z$--subgroup $A$ acts on a stable line of $\Xa_\Gamma$ in the sense of Definition~\ref{def:stable line} if and only if 
\begin{enumerate}
	\item the centralizer $Z_{A_\Gamma}(A)$ of $A$ has a finite index subgroup isomorphic to $F_k\times A$, where $F_k$ is a free group with $k$ generators for $k\ge 2$;
	\item there does not exist another cyclic subgroup $B\le A_\Gamma$ such that $B\cap A$ is the trivial subgroup and the projection of $Z_{A_\Gamma}(B)\cap(F_k\times A)$ to the $F_k$ factor has finite index in $F_k$.
\end{enumerate}
\end{conj}
We will not need this conjecture in the later part of the paper.

It follows from the construction of $X_\Gamma$ and the structure of its vertex links that for each atomic sector $S\subset \Xa_\Gamma$, there exists a unique subcomplex $S'\subset X_\Gamma$ such that
\begin{enumerate}
	\item $S'$ with the induced piecewise Euclidean structure is isometric to a flat sector in the Euclidean plane;
	\item the center of any $2$--cell of $S$ is contained in $S'$, and $S'$ is the convex hull of such centers, where the convex hull is taken inside $S'$ with respect to the flat metric on $S'$.
\end{enumerate}
The subcomplex $S'$ is called the \emph{shadow} of $S$. Similarly, the completion of each atomic sector also has a \emph{shadow} which is either a CAT(0) sector or a flat plane. Note that if we view $S'$ and $S$ as subsets of $X_\Gamma$, then they are at finite Hausdorff distance. Moreover, if two atomic sectors satisfy $S_1\subset S_2$, then the shadow of $S_1$ is contained in the shadow of $S_2$. 

\begin{lemma}
	\label{lem:characterize stable line}
A singular line $L$ is stable if and only if there exists an atomic sector $S$ such that $L$ bounds $\bar{S}$ and the shadow of $\bar{S}$ has angle $=\pi$.
\end{lemma}

\begin{proof}
We prove the \textquotedblleft if\textquotedblright\ direction. Suppose that $\bar{S}$ satisfies the assumption of the lemma. If $S$ is a DCH or $S$ is a diamond-plain sector with $\bar{S}$ being a half-flat, then Definition~\ref{def:stable line} (1) holds. If $S$ is a CCH, then Definition~\ref{def:stable line} (2) holds. If $S$ is a plain sector with $\bar{S}$ being a half-flat, then Definition~\ref{def:stable line} (3) holds. If $S$ is a PCH, then Definition~\ref{def:stable line} (4) or (5) holds. These are all the possibilities.

Now we prove the converse. Let $L$ be a Coxeter line. Let $F_1$ (resp.\ $F_2$) be the Coxeter-plain flat (resp.\ Coxeter flat) containing $L$. There are two basic moves to extend $L$ to a CCH. The first move is to enlarge $L$ inside $F_1$ to a finite width infinite strip bounded by $L$, the second move is to enlarge $L$ inside $F_2$ to obtain a thickened Coxeter line with one side bounded by $L$. By alternating between these two moves for infinitely many times, we can construct a CCH bounded by $L$. The other cases are similar. Note that in Definition~\ref{def:stable line} (3), we need to use all the vertices in $a^{\perp}$ to obtain a plain sector whose completion satisfies all the requirements. Definition~\ref{def:stable line} (5) can be reduced to (4) by first constructing a thickened plain line between an $a$--labeled plain line and a $b$--labeled plain line.
\end{proof}

\begin{lemma}
	\label{lem:stable lines are intersections}
There exist a constant $A$ and a function $k\colon (0,\infty)\to (0,\infty)$ such that the following hold for any stable line $L$:
\begin{enumerate}
	\item there are three $(A,A)$--quasiflats $Q_1,Q_2,Q_3$ such that for any $C\ge A$, we have $d_H(L,\cap_{i=1}^3N_C(Q_i))<k(C)$;
	\item each $Q_i$ is at finite Hausdorff distance from the disjoint union of the completion of two atomic sectors such that their shadows have angle $=\pi$.
\end{enumerate}
\end{lemma}

\begin{proof}
Since there are finitely many $A_\Gamma$--orbits of stable lines in $\Xa_\Gamma$, it suffices to prove the lemma for a single stable line. Consider a case of a Coxeter line $L$ with its boundaries denoted by $\ell_1$ and $\ell_2$. We extend $L$ along $\ell_1$ to a CCH $U'_1$ (resp.\ $U'_2$) bounded $L$ by first applying the first move (resp.\ the second move) in Lemma~\ref{lem:characterize stable line} and then alternating between first move and second move. For $i=1,2$, let $U_i$ be the sub-half-flat of $U'_i$ bounded by $\ell_1$. Similarly, we extend $L$ along $\ell_2$ to a CCH $U_3$ bounded by $L$. By the argument in the proof of Lemma~\ref{lem:atomic embedding} and Lemma~\ref{lem:atomic qi embedding}, we know $U_i\cup U_j$ is embedded and quasi-isometrically embedded for any $1\le i\neq j\le 3$. Hence $L$ is the coarse intersection of these three quasiflats. Other cases of $L$ are similar.
\end{proof}

\begin{theorem}
	\label{thm:preservation of stable lines}
Suppose that $A_{\Gamma_1}$ and $A_{\Gamma_2}$ are two-dimensional Artin groups. Let $q\colon \Xa_{\Gamma_1}\to \Xa_{\Gamma_2}$ be an $(L,A)$--quasi-isometry. Then there exists a constant $D$ such that for any stable line $L_1\subset \Xa_{\Gamma_1}$, there is a stable line $L_2\subset \Xa_{\Gamma_2}$ such that $d_H(q(L_1),L_2)<D$.
\end{theorem}

The weaker statement $d_H(q(L_1),L_2)<\infty$ follows from Corollary~\ref{cor:qi and completion} and Lemma~\ref{lem:characterize stable line}. We need more work to make $D$ independent of $L_1$.
\begin{proof}
Suppose $L_1$ is the coarse intersection of three quasiflats $O_1,P_1,Q_1$ as in Lemma~\ref{lem:stable lines are intersections}. Then there are atomic sectors $S_1,T_1$ such that $\bar S_1\cap \bar T_1=\emptyset$ and $d_H(\bar S_1\cup \bar T_1, O_1)<\infty$. By Corollary~\ref{cor:qi and completion}, there are atomic sectors $ S_2$ and $ T_2$ in $\Xa_{\Gamma}$ such that their completions satisfy $d_H(q(\bar S_1),\bar S_2)<\infty$ and $d_H(q(\bar T_1),\bar T_2)<\infty$. Note that the shadows of $\bar S_2$ and $\bar T_2$ have angle $\le \pi$. 

Let $O_2=q(O_1)$. Let $f'\colon O'_2\to X_{\Gamma_2}$ be a simplicial map approximating the quasiflat $O_2$ satisfying all the conditions (1)--(5) in Theorem~\ref{thm:MSquasiflats}, with $O'_2$ being a CAT(0) plane. Lemma~\ref{lem:stable lines are intersections} and Theorem~\ref{thm:MSquasiflats} imply that $d_H(\im f', O_2)<D_0$ for $D_0$ independent of $L_1$. 

We claim $O'_2$ is flat. Let $\beta$ be the length of the Tits boundary $\partial_{T}O'_2$. It suffices to prove $\beta=2\pi$. We have $\beta\ge 2\pi$ by CAT(0) geometry. It remains to show $\beta\le 2\pi$. By Theorem~\ref{thm:main2}, $O_2$ is at bounded Hausdorff distance away from a union of finitely many atomic sectors $\{\widehat S_i\}_{i=1}^k$. Since $d_H(O_2,\bar S_2\cup \bar T_2)<\infty$, we know from Lemma~\ref{lem:sector touch} that up to passing to a sub-sector, each $\widehat S_i$ is contained in at least one of $\bar S_2$ or $\bar T_2$. Let $\alpha_i$ be the angle of the CAT(0) sector $\widehat S'_i\subset O'_2$ associated with $\widehat S_i$ as in Theorem~\ref{thm:main1}. Then $\beta=\sum_{i=1}^k \alpha_i$. We can assume $f'(\widehat S'_i)$ is the shadow of $\widehat S_i$. Each $f'(\widehat S'_i)$ is contained in at least one of $\bar S'_2$ or $\bar T'_2$, which are shadows of $\bar S_2$ or $\bar T_2$ respectively. Moreover, if both $f'(\widehat S'_i)$ and $f'(\widehat S'_j)$ ($i\neq j$) are contained in $\bar S'_2$, then $f'(\widehat S'_i)\cap f'(\widehat S'_j)$ can not contain a flat sector with angle $>0$ (otherwise $\widehat S_i$ and $\widehat S_j$ touch along discs, which implies $d_H(\widehat S_i,\widehat S_j)<\infty$). Thus $\sum_{i=1}^k \alpha_i$ is bounded above by the sum of the angle of $\bar S'_2$ and $\bar T'_2$, which is $2\pi$. This finishes the proof of the claim.

By Lemma~\ref{lem:well-defined}, the map $\rho$ in Section~\ref{s:general} is well-defined on $O'_2$. Thus we can give a new cell structure on $O'_2$ and a new locally injective cellular map $f\colon \widetilde O_2\to \Xa_{\Gamma}$ ($\widetilde O_2$ is $O'_2$ with the new cell structure) as in Section~\ref{subsec:new cell structure} such that each vertex of $\widetilde O_2$ satisfies Lemma~\ref{lem:star homeo}. Note that $d_H(\im f, O_2)<D_1$ for $D_1$ independent of $L_1$. By Lemma~\ref{lem:sector quasiflat touch} and Lemma~\ref{lem:completion containment}, we can assume $\bar S_2\subset \widetilde O_2$. By using Lemma~\ref{lem:star homeo} and the argument in Section~\ref{sec:development}, we deduce that the topological closure of $\widetilde O_2\setminus \bar S_2$ in $\widetilde O_2$ is one of the following forms:
\begin{enumerate}
	\item when $\bar S_2$ is bounded by a diamond line, the topological closure is a union of diamond lines;
	\item when $\bar S_2$ is bounded by a Coxeter line, it is a CCH;
	\item when $\bar S_2$ is bounded by a plain line, it is a union of thickened plain lines.
\end{enumerate}
Moreover, by the proof of Lemma~\ref{lem:atomic embedding}, $\widetilde O_2$ is embedded.

Now we define $\widetilde P_2$ and $\widetilde Q_2$ in a similar way. Assume without loss of generality that $\bar S_2$ is the coarse intersection of $\widetilde O_2$ and $\widetilde P_2$. Then by Lemma~\ref{lem:sector quasiflat touch} and Lemma~\ref{lem:completion containment}, we can further assume $\bar S_2\subset \widetilde O_2\cap\widetilde P_2$. Then the topological closure of $\widetilde P_2\setminus \bar S_2$ in $\widetilde P_2$ is one of the forms listed from (1)-(3) in the previous paragraph. By mapping $\widetilde O_2$ and $\widetilde P_2$ to a suitable tree as in Lemma~\ref{lem:atomic embedding} and Lemma~\ref{lem:diamond embedding}, we know $\widetilde O_2\cup\widetilde P_2$ gives an embedded triplane in $\Xa_{\Gamma_2}$, moreover, $\partial(\widetilde O_2\cap\widetilde P_2)$ is either a plain line, or a boundary line of a diamond line, or a Coxeter line. Let $\widehat Q_2$ be the topological closure of the symmetric difference of $\widetilde O_2$ and $\widetilde P_2$. Then $\widehat Q_2$ is homeomorphic to $\mathbb R^2$ and $d_H(\widehat Q_2,\widetilde Q_2)<\infty$. Since both $\widehat Q_2$ and $\widetilde Q_2$ are homeomorphic to $\mathbb R^2$ and are quasi-isometrically embedded, we use the argument from the proof of Lemma~\ref{lem:sector touch} to show larger and larger discs of $\widehat Q_2$ centered at a fixed vertex are contained in $\widetilde Q_2$, hence $\widehat Q_2\subset \widetilde Q_2$. Thus $\widehat Q_2=\widetilde Q_2$. Then there is a singular line $L_2$ such that $\widetilde O_2\cap\widetilde P_2\cap \widetilde Q_2$ is either contained in $L_2$, or a boundary of $L_2$. The line $L_2$ is actually a stable line by Lemma~\ref{lem:characterize stable line}. One now readily verifies that $d_H(q(L_1),L_2)<D$ for some $D$ independent of $L_1$. 
\end{proof}

\begin{remark}
Note that for each stable line, there is a $\mathbb Z$--subgroup of $A_\Gamma$ acting cocompactly on the line. Thus Theorem~\ref{thm:preservation of stable lines} also implies that the collection of $\mathbb Z$--subgroups corresponding to stable lines is preserved by quasi-isometries up to finite Hausdorff distance.
\end{remark}

\subsection{Preservation of intersection graphs}
Two singular lines are \emph{parallel} if their Hausdorff distance is finite. The \emph{parallel set} $P_L$ of a stable line $L$ is defined to be the union of all stable lines that are parallel to $L$. 

The following is motivated by the fixed set graph defined by Crisp \cite{MR2174269} and the extension graph for right-angled Artin groups defined by Kim and Koberda \cite{kim2013embedability}. 

\begin{definition}[Intersection graph]
	\label{def:intersection graph}
	Let $A_\Gamma$ be a $2$--dimensional Artin group. The \emph{intersection graph} $\I_\Gamma$ of $A_\Gamma$ is defined as follows. There is a one to one correspondence between vertices of $\I_\Gamma$ and parallelism classes of diamond lines, Coxeter lines and single labeled plain lines in the universal cover of the presentation complex of $A_\Gamma$. Let $v_1,v_2$ be two vertices representing two such parallelism classes $\p_1$ and $\p_2$. Then $v_1$ and $v_2$ are joined by an edge if and only if for $i=1,2$, there exists a $\mathbb Z$ subgroup $Z_i\le A_\Gamma$ such that $Z_i$ stabilizes a line in $\p_i$, and $Z_1,Z_2$ generate a free abelian subgroup of rank $2$ in $A_\Gamma$. The \emph{stable subgraph} of $\I_\Gamma$ is the full subgraph of $\I_\Gamma$ spanned by vertices corresponding to stable lines. 
\end{definition}

It could happen that the stable subgraph of $\I_\Gamma$ is empty. For example, this is the case when $\Gamma$ is a $4$--gon with each edge labeled by $2$. Of course, in such case, one should not expect a quasi-isometry to map any $\mathbb Z$ subgroups to other $\mathbb Z$ subgroups up to finite Hausdorff distance.

In the rest of this subsection, we prove Theorem~\ref{thm:invariance}.

\begin{lemma}
	\label{lem:finite out}
	Suppose $A_\Gamma$ is $2$--dimensional. Suppose that
	\begin{enumerate}
		\item $\Gamma$ is connected and does not contain valence one vertex;
		\item for any vertex $u\in\Gamma$ which commutes with all of its adjacent vertices, we have $(u^\perp)^\perp=u$.
	\end{enumerate}
Then the stable subgraph of $\I_\Gamma$ is all of $\I_\Gamma$. In particular, if the outer automorphism group $\mathrm{Out}(A_\Gamma)$ is finite and $\Gamma$ has more than two vertices, then the stable subgraph of $\I_\Gamma$ is all of $\I_\Gamma$.
\end{lemma}

\begin{proof}
The first statement follows from Definition~\ref{def:stable line}. It remains to prove the ``in particular'' statement. Suppose $\mathrm{Out}(A_\Gamma)$ is finite.
If $\Gamma$ has a valence one vertex $v$, then any vertex $w\in\Gamma$ adjacent to $v$ is a separating vertex (since $\Gamma$ has more than two vertices). This leads to a Dehn twist automorphism defined in \cite[pp.\ 1383]{MR2174269} and hence $\mathrm{Out}(A_\Gamma)$ is infinite. Thus $\Gamma$ does not have valence one vertices. Similarly, we know $\Gamma$ is connected. Let $u\in\Gamma$ be vertex such that all edges of $\Gamma$ containing $u$ are labeled by $2$. If there is a vertex $v\neq u$ such that $v\in (u^\perp)^\perp$, then sending $v$ to $vu$ and fixing all other generators yields an infinite order element in $\mathrm{Out}(A_\Gamma)$. Thus we must have $(u^\perp)^\perp=u$ for all such $u$. 
\end{proof}

\begin{lemma}
	\label{lem:uniform}
	Let $L\subset\Xa_\Gamma$ be a diamond line, or a Coxeter line, or a single-labeled plain line. Then any singular line $L'$ parallel to $L$ is of the same type. Moreover, there exists a $\mathbb Z$--subgroup $Z\le A_\Gamma$ such that 
	\begin{enumerate}
		\item $Z$ stabilizes any singular line that is parallel to $L$;
		\item for any $g\in Z$, there is a constant $C>0$ only depending on $g$ such that $d(x,gx)<C$ for any vertex $x$ in the parallel set of $L$.
	\end{enumerate}
\end{lemma}

\begin{proof}
Suppose that $L'$ and $L$ are parallel with their stabilizers denoted by $Z$ and $Z'$. Then by \cite[Corollary 2.4]{MR2867450}, $Z_0=Z\cap Z'$ is of finite index in both $Z$ and $Z'$. Then $L'$ and $L$ are of the same type by the following:
\begin{itemize}
	\item if $L$ is a diamond line, then the fixed point set Fix$(Z)$ of the $Z$--action on the Deligne complex $D_\Gamma$ is a vertex of rank 2, and Fix$(Z)$=Fix$(Z_0)$ by \cite[Lemma 8 (ii)]{MR2174269};
	\item if $L$ is a single-labeled plain line, then Fix$(Z)$ is a tree containing a vertex of rank 1, and Fix$(Z)$=Fix$(Z_0)$ by \cite[Lemma 8 (i)]{MR2174269};
	\item if $L$ is a Coxeter line, then Fix$(Z)$ is empty and $Z$ acts by translations on a geodesic line in $D_\Gamma$.
\end{itemize}

Now we prove the ``moreover" statement. By Section~\ref{subsec:singular lines}, each such $L$ is quasi-isometrically embedded with constants independent of $L$. Thus (2) follows from (1). It suffices to prove (1). Let $L,L',Z$ and $Z'$ be as before. If $L$ is a single-labeled plain line, then Fix$(Z)=$Fix$(Z_0)=$Fix$(Z')$. By the definition of $D_\Gamma$, we know $L$ (resp.\ $L'$) corresponds to a vertex in Fix$(Z)$ (resp.\ Fix$(Z')$), thus $Z$ also stabilizes $L'$ and (1) holds for single-labeled plain lines. If $L$ is a diamond line, then the above discussion implies that $L$ and $L'$ are in the same block $B$. Thus the centralizer of the stabilizer of $B$ will stabilize all the diamond lines in $B$.
	
	It remains to consider the case $L$ is a Coxeter line. We claim there is a finite chain $L_0=L,L_1,\ldots,L_{n-1},L_n=L'$ such that each $L_i$ is a Coxeter line and for $0\le i\le n-1$, either $L_i$ and $L_{i+1}$ are contained in the same Coxeter-plain flat, or they are contained in the same Coxeter flat. The claim will imply that the stabilizer of $L$ also stabilizes $L'$, which finishes the proof.
	
	Now we prove the above claim. Let $\W$ and $\W'$ be the walls of $L$ and $L'$ respectively. Let $\pi \colon X^b_\Gamma\to D_\Gamma$ be the $A_\Gamma$--equivariant simplicial map defined at the beginning of Section~\ref{subsec:properties of atomic sectors}. Let $\ell=\pi(\W)$ and $\ell'=\pi(\W')$. Then $\ell$ and $\ell'$ are parallel geodesic lines in a $CAT(0)$ space $D_\Gamma$. If $\ell=\ell'$, then by the definition of $D_\Gamma$, we know $L$ and $L'$ are contained in the same Coxeter-plain flat, hence the claim follows. Now we assume $\ell\neq\ell'$, then they bound a flat strip $E\subset D_\Gamma$. Since $D_\Gamma$ is $2$--dimensional, $E$ is a subcomplex of $D_\Gamma$. Let $F$ be the Coxeter-plain flat containing $L$. Let $L^\perp$ be a thickened plain line in $F$ intersecting $L$ in a $2$--cell $C$. Let $v\in C$ be the center of $C$ and let $u,w$ be two vertices in $\W$ adjacent to $v$ (with respect to the simplicial structure on $X^b_\Gamma$). Then $u$ and $w$ are in two different boundary lines of $L^\perp$. Let $\Delta\subset E$ be the triangle containing $\bar{v}=\phi(v)$ and $\bar{u}=\phi(u)$. Then $\Delta$ has a rank 0 vertex $\bar{x}$. Let $x=\rho^{-1}(\bar x)$. Then by the definition of $D_\Gamma$, we know $x$ and $u$ are contained in the same boundary line of $L^\perp$. Let $C_x$ be a $2$--cell in $L^\perp$ containing $x$. Then $\ell$ cuts $\rho(C_x)$ into two halves such that one of them (denoted by $K$) contains both $\Delta$ and $\rho(w)$. Since $K\cap E$ is a convex subcomplex of $K$, we have $K\subset E$. Let $L_x$ be the Coxeter line in $F$ intersecting $L^\perp$ in $C_x$. Similarly, we know that $\ell$ cuts $\rho(L_x)$ into two halves with one of them contained in $E$. Let $F_1$ be the Coxeter flat containing $L_x$, and let $\W_x$ and $\W'_x$ be two distinct parallel walls in $F_1$ such that $\W_x\subset L_x$, $x$ is between $\W_x$ and $\W'_x$, and the distance between $\W_x$ and $\W'_x$ is as small as possible. Then $\rho$ maps the region of $F_1$ between $\W_x$ and $\W'_x$ to a flat strip $U$ in $D_\Gamma$. Since a half of $\rho(L_x)$ is contained in $U\cap E$, by convexity, we know  $U\subset E$. Now we repeat the previous argument with $L$ replaced by the carrier of $\W'_x$. This process terminates after finitely many steps and gives rise to the finite chain as required.
\end{proof}

\begin{theorem}
	\label{thm:invariance}
	Let $A_{\Gamma}$ and $A_{\Gamma'}$ be two $2$--dimensional Artin groups and let $q\colon A_\Gamma\to A_{\Gamma'}$ be a quasi-isometry. Then $q$ induces an isomorphism from the stable subgraph of $\I_{\Gamma}$ onto the stable subgraph of $\I_{\Gamma'}$. 
If both $A_\Gamma$ and $A_{\Gamma'}$ satisfy conditions (1) and (2) in Lemma~\ref{lem:finite out}, then $q$ induces an isomorphism between their intersection graphs. 

Thus if both $A_\Gamma$ and $A_{\Gamma'}$ have finite outer automorphism group and $\Gamma$ has more than two vertices, then $q$ induces an isomorphism between their intersection graphs.
\end{theorem}

\begin{proof}
	Let $\Theta_\Gamma$ and $\Theta_{\Gamma'}$ be the stable subgraphs of $\I_\Gamma$ and $\I_{\Gamma'}$ respectively. It follows from Theorem~\ref{thm:preservation of stable lines} that $q_\ast$ induces a well-defined injective map on the vertex set of $\Theta_\Gamma$. Now we show if two vertices $v_1,v_2$ are adjacent in $\Theta_\Gamma$, then $q_{\ast}(v_1)$ and $q_{\ast}(v_2)$ are adjacent in $\Theta_{\Gamma'}$.
	
	Let $L_1$ and $L_2$ be two stable lines representing $v_1$ and $v_2$. For $i=1,2$, let $P_i$ be the parallel set of $L_i$ and suppose $\mathbb Z_i$ is a $\mathbb Z$--subgroup stabilizing $L_i$. We assume $\mathbb Z_1$ and $\mathbb Z_2$ generate a free abelian subgroup of rank 2. Let $Q=\mathbb Z_1\oplus\mathbb Z_2$ and we view $Q$ as a subset of $\Xa_\Gamma$. Then $Q$ is at finite Hausdorff distance away from a union of stable lines parallel to $L_1$. By Theorem~\ref{thm:preservation of stable lines}, there exists a stable line $L'_i$ such that $Q'=q(Q)$ is at finite Hausdorff distance away from a union of stable lines parallel to $L'_1$. Let $g'_1$ be a generator of a $\mathbb Z$--subgroup $\mathbb Z'_1$ that stabilizes $L'_1$. Then Lemma~\ref{lem:uniform} (2) implies that $d_H(Q',g'_1 Q')<\infty$. Similarly, we define $L'_2,g'_2$ and $\mathbb Z'_2$ and we have $d_H(Q',g'_2 Q')<\infty$.
	
	Note that $Q'$ is a quasiflat. Now we assume there is a continuous quasi-isometric embedding $f\colon \mathbb E^2\to \Xa_{\Gamma'}$ such that $d_H(\im f, Q')<\infty$. Let $[\mathbb E^2]\in H^{\mathrm{lf}}_{2}(\mathbb E^2)$ be the fundamental class and let $[\alpha]=f_{\ast}([\mathbb E^2])\in H^{\mathrm{lf}}_{2}(\Xa_{\Gamma'})$. By \cite[Lemma 4.3]{bks}, the support set $S_{[\alpha]}$ of $[\alpha]$ is non-empty and satisfies $d_H(\im f, S_{[\alpha]})<\infty$. Let $f_i$ be $f$ post-composed with the action of $g'_i$. Lemma~\ref{lem:uniform} (2) implies that there is a uniform bound on $d(f(x),f_i(x))$ for any $x\in \mathbb E^2$. Since $\Xa_{\Gamma'}$ is uniformly contractible, there is a proper homotopy between $f$ and $f_i$. Thus $[\alpha]=(f_i)_{\ast}[\mathbb E^2]=(g'_i)_{\ast}[\alpha]$. Thus the support set $S_{[\alpha]}$ is invariant under the subgroup $H=\langle\mathbb Z'_1,\mathbb Z'_2\rangle$ of $A_{\Gamma'}$. Then $H$ is quasi-isometric to $\mathbb E^2$, hence is virtually $\mathbb Z\oplus\mathbb Z$. It follows that a finite index subgroup of $\mathbb Z'_1$ and a finite index subgroup of $\mathbb Z'_2$ generate a free abelian subgroup. Hence $q_{\ast}(v_1)$ and $q_{\ast}(v_2)$ are adjacent.
	
	In summary, we now have a graph embedding $q_\ast\colon \Theta_\Gamma\to\Theta_{\Gamma'}$. By considering the quasi-isometry inverse of $q$ and repeating the above argument, we know $q_\ast$ is actually an isomorphism. The ``in particular" statement of the theorem follows from Lemma~\ref{lem:finite out}.
\end{proof}

\begin{corollary}
	\label{cor:injective2}
	Let $A_{\Gamma}$ be a $2$--dimensional Artin group such that $\Gamma$ satisfies (1) and (2) in Lemma~\ref{lem:finite out}.
If $q\colon A_\Gamma\to A_\Gamma$ is an $(L,A)$--quasi-isometry inducing the identity map on $\I_\Gamma$, then there exists $C=C(L,A,\Gamma)$ such that $d(q(x),x)\le C$ for any $x\in A_\Gamma$. Thus the map $\QI(A_\Gamma)\to \Aut(\I_\Gamma)$ described in Theorem~\ref{thm:invariance} is an injective homomorphism.

Thus if we know in addition that the homomorphism $A_\Gamma\to \Aut(\I_\Gamma)$ induced by the action $A_\Gamma\act \Aut(\I_\Gamma)$ has finite index image, then any finitely generated groups quasi-isometric to $A_\Gamma$ is virtually $A_\Gamma$.
\end{corollary}

\begin{proof}
First, observe that, by virtue of Definition~\ref{def:intersection graph} and Theorem~\ref{thm:preservation of stable lines} the map $\QI(A_\Gamma)\to \Aut(\I_\Gamma)$ is a homomorphism.
	
Let $x\in A_\Gamma$. It follows from the assumption on $\Gamma$ that there exist $F_1,F_2,F_3\subset \Xa_\Gamma$ containing $x$ such that
\begin{enumerate}
	\item each $F_i$ is either a diamond-plain flat (cf.\ Definition~\ref{def:diamond-plain flat}) or a cubical flat spanned by two commuting elements of $\Gamma$;
	\item $x=F_1\cap F_2\cap F_3$, moreover, there exists $f\colon \mathbb R_{\ge 0}\to \mathbb R_{\ge 0}$ such that $\cap_{i=1}^3N_r(F_i)$ is contained in the $f(r)$--ball around $x$.
\end{enumerate}
Note that each $F_i$ corresponds to an abelian subgroup of $A_i\le A_\Gamma$ such that $\cap_{i=1}^3A_i$ is the trivial subgroup, thus (2) follows from \cite[Corollary 2.4]{MR2867450}. We can assume that the function $f$ depends only on $A_\Gamma$ (but not on $x$) by applying a left translation. As $q$ induces the identity map on $\I_\Gamma$, it follows from the third paragraph of the proof of Theorem~\ref{thm:invariance} that if we consider the support set $S_i$ associated with $q(F_i)$ as there, then $S_i$ is invariant under a finite index subgroup of $A_i$. Thus $d_H(q(F_i),F_i)<\infty$. Actually this can be improved to $d_H(q(F_i),F_i)<C$ for $C$ depending only on $L,A,A_\Gamma$, as $d_H(S_i,q(F_i))<C=C(L,A,A_\Gamma)$ and $S_i=F_i$. This and the property (2) above imply the corollary.
\end{proof}

\section{Artin groups of type CLTTF}
\label{sec:CLTTF}
In this section, we restrict ourselves to groups $A_\Gamma$ belonging to the class of CLTTF Artin groups, as in Definition~\ref{def:CLTTF}. Note that there are no Coxeter lines in $\Xa_\Gamma$. Let $\I_\Gamma$ be the intersection graph of $A_\Gamma$ and let $\Theta_\Gamma$ be the stable subgraph (cf.\ Definition~\ref{def:intersection graph}).

\subsection{The fixed set graphs and chunks} 
\label{subsec:fixed set graphs}
We recall several notions from \cite{MR2174269}. Crisp defined a notion of \emph{fixed set graph} for $A_\Gamma$ in \cite[Section 4]{MR2174269}. It follows from \cite[Lemma 9]{MR2174269} that a single-labeled line is stable if and only if the centralizer of the stabilizer of this line is not virtually abelian (recall that we require $A_\Gamma$ to be of large type). Now by \cite[Lemma 12]{MR2174269} and the sentence before this lemma, we know that the there is a natural isomorphism between the fixed set graph and $\Theta_\Gamma$.

Crisp divides the vertex set of $\Theta_\Gamma$ into two disjoint subsets, one being $\V_\Gamma$, which consists of vertices coming from diamond lines, and one being $\F_\Gamma$, which consists of vertices coming from stable plain lines. Note that there is a one to one correspondence between elements in $\V_\Gamma$ and blocks in $\Xa_\Gamma$. It is clear that there is an action $A_\Gamma\act \Theta_\Gamma$ which preserves $\V_\Gamma$ and $\F_\Gamma$. 

We choose an identification of $A_\Gamma$ with the $0$--skeleton of $\Xa_\Gamma$. Let $\Gamma_2$ be the first subdivision of $\Gamma$. Each vertex of $\Gamma_2$ corresponds to a unique single-labeled plain line or a diamond line that contains the identity element $\ast$ of $A_\Gamma$. Let $\widehat \Gamma$ be the full subgraph of $\Gamma_2$ spanned by the non-terminal vertices. Then vertices of $\widehat \Gamma$ give stable lines containing $\ast$ (however, it is possible that a terminal vertex of $\Gamma_2$ also gives a stable line). This induces a graph embedding $\widehat \Gamma\to \Theta_\Gamma$. From now on, we shall identify $\widehat \Gamma$ as a subgraph of $\Theta_\Gamma$ via such embedding. Note that $\Theta_\Gamma$ is the union of $A_\Gamma$ translates of $\widehat \Gamma$.

Similarly, the vertices of $D_\Gamma$ that are associated with left cosets of $A_\Gamma$ that contain $\ast$ span a convex subcomplex $K$ of $D_\Gamma$, which is called the \emph{fundamental chamber} of $D_\Gamma$. As a simplicial complex, $K$ is isomorphic to the cone over $\Gamma_2$. $D_\Gamma$ is the union of $A_\Gamma$ translates of $K$.

There is another cell structure on $D_\Gamma$, where one ignores all the edges between a rank $0$ vertex and a rank $2$ vertex and combines triangles of $D_\Gamma$ into squares. We denote $D_\Gamma$ with this new cell structure by $\square_\Gamma$. When $\Gamma$ is CLTTF, we know that $\square_\Gamma$ is a CAT(0) cube complex (actually, a more general version of this fact was proved in \cite{CharneyDavis}). 

\begin{lemma}
	\label{lem:cycle in links}
\cite[Lemma 39]{MR2174269} 
Let $v\in \square_\Gamma$ be a rank $2$ vertex and let $m$ be the label of the defining edge of the block that is associated with $v$.
Every simple cycle in $\lk(v,\square_\Gamma)$ has edge length at least $2m$, and there exists at least one simple cycle of $2m$ edges in $\lk(v,\square_\Gamma)$.
\end{lemma}

\begin{prop}
	\label{prop:local rigid}
\cite[Proposition 40]{MR2174269} 
Let $v$ be a rank $2$ vertex in the fundamental chamber of $\square_\Gamma$ and let $A_{s,t}$ be the standard subgroup of $A_\Gamma$ associated with $v$ such that its defining edge is $\overline{st}\subset\Gamma$. Let $E$ be the edge in $\lk(v,\square_\Gamma)$ spanned by vertices corresponding to $A_s$ and $A_t$. If $\tau$ is a graph automorphism of $\lk(v,\square_\Gamma)$ that fixes $E$, then either $\tau$ is the identity on $\lk(v,\square_\Gamma)$, or $\tau$ is induced by the group inversion such that $s\to s^{-1}$ and $t\to t^{-1}$.
\end{prop}

The following notion of chunks was introduced in \cite[Section 7]{MR2174269}
\begin{definition}
Let $A$ be a connected full subgraph of $\Gamma$. $A$ is \emph{indecomposable} if, for every decomposition $\Gamma=\Gamma_1\cup_T\Gamma_2$ of $\Gamma$ over a separating edge or vertex $T$, either $A\subset\Gamma_1$ or $A\subset\Gamma_2$. A \emph{chunk} of $\Gamma$ is a maximal indecomposable (connected and full) subgraph of $\Gamma$. Clearly, any two chunks of $\Gamma$ intersect, if at all, along a single separating edge or vertex. A chunk of $\Delta$ shall be said to be \emph{solid} if it contains a simple closed circuit of $\Delta$. Note that any chunk is either solid or consists of just a single edge of the graph.

$\widehat \Gamma$ can be viewed both as a subset of $\Gamma$ and as a subgraph of $\Theta_\Gamma$. A \emph{(solid) chunk} of $\widehat \Gamma$ is the intersection of $\widehat \Gamma$ with a (solid) chunk of $\Gamma$. This defines a subgraph of $\widehat \Gamma$ which shall be thought of as lying inside $\Theta_\Gamma$. A \emph{(solid) chunk} of $\Theta_\Gamma$ is any translate of a (solid) chunk of $\widehat \Gamma$ by an element of $A_\Gamma$. 
\end{definition}

An isomorphism $\phi$ from $\Theta_\Gamma$ to $\Theta_{\Gamma'}$ is a \emph{$\V\F$--isomorphism} if $\phi(\V_\Gamma)=\V_{\Gamma'}$ and $\phi(\F_\Gamma)=\F_{\Gamma'}$.

\begin{prop}
	\label{prop:VF}
	\cite[Proposition 41]{MR2174269} Suppose $A_\Gamma$ and $A_{\Gamma'}$ are CLTTF. If $\Gamma$ is not a tree, then any isomorphism $\Theta_\Gamma\to\Theta_{\Gamma'}$ is a $\V\F$--isomorphism.
\end{prop}

\begin{prop}
	\label{prop:chunk}
\cite[Proposition 23]{MR2174269} Suppose $A_\Gamma$ and $A_{\Gamma'}$ are CLTTF. Any $\V\F$--isomorphism $\Theta_\Gamma\to\Theta_{\Gamma'}$ maps each solid chunk of $\Theta_\Gamma$	onto a solid chunk of $\Theta_{\Gamma'}$.
\end{prop}

\subsection{Quasi-isometries of CLTTF Artin groups}
\label{subsec:quasi-isometries}
\begin{lemma}
	\label{lem:graph}
	Suppose that $A_\Gamma$ and $A_{\Gamma'}$ are CLTTF Artin groups. Let $q\colon \Xa_\Gamma\to \Xa_{\Gamma'}$ be an $(L,A)$--quasi-isometry. Then 
	\begin{enumerate}
		\item $q$ induces a graph isomorphism $q_\ast\colon \Theta_\Gamma\to\Theta_{\Gamma'}$;
		\item if in addition $\Gamma$ is not a tree, then there exists a constant $D=D(L,A,\Gamma,\Gamma')$ such that for any block $B\subset\Xa_\Gamma$, there exists a block $B'\subset\Xa_{\Gamma'}$ such that $d_H(q(B),B')<D$; moreover, $q_\ast$ sends the vertex in $\Theta_\Gamma$ associated with $B$ to the vertex in $\Theta_{\Gamma'}$ associated with $B'$.
	\end{enumerate}
\end{lemma}

\begin{proof}
Theorem~\ref{thm:invariance} implies (1). Now we prove (2). By Proposition~\ref{prop:VF}, $q_{\ast}\colon \Theta_\Gamma\to\Theta_{\Gamma'}$ is a $\V\F$--iso\-mor\-phism. Now suppose $B$ is the parallel set of a diamond line $L$. By Theorem~\ref{thm:preservation of stable lines}, there exists a stable line $L'\subset\Xa_{\Gamma'}$ such that $d_H(P_{L'},q(P_L))<D$ for $D=D(L,A,\Gamma_1,\Gamma_2)$. Note that $L'$ has to be a diamond line. Thus $P_{L'}$ is a block $B'$ and $d_H(B',q(B))<D$. 
\end{proof}

\begin{theorem}
	\label{thm:CLTTF1}
Let $A_\Gamma$ and $A_{\Gamma'}$ be CLTTF Artin groups. Suppose $\Gamma$ does not have separating vertices and edges. Let $q\colon A_\Gamma\to A_{\Gamma'}$ be an $(L,A)$--quasi-isometry. Then there exist a constant $D=D(L,A,\Gamma)$ and a bijection $\tilde q\colon  A_\Gamma\to A_{\Gamma'}$ such that
\begin{enumerate}
	\item $d(q(x),\tilde q(x))<D$ for any $x\in A_\Gamma$;
	\item if $\Gamma_1\subset\Gamma$ is an edge (resp.\ vertex), then for any $g\in A_\Gamma$, there exists an edge (resp.\ vertex) $\Gamma'_1\subset\Gamma'$ and $g'\in A_{\Gamma'}$ such that $\tilde q$ maps $g A_{\Gamma_1}$ bijectively onto $g'A_{\Gamma'_1}$; and similar properties hold for $\tilde q^{-1}$.
\end{enumerate}
Moreover, such $\tilde q$ is unique.
\end{theorem}

\begin{proof}
Pick a vertex $x\in\Xa_\Gamma$ and let $\{B_i\}_{i=1}^{n}$ be the collection of blocks containing $x$. Note that $x=\cap_{i=1}^{n} B_i$. Let $\{v_i\}_{i=1}^n$ be the collection of vertices in $\V_\Gamma$ corresponding to $\{B_i\}_{i=1}^{n}$. Let $g\in A_\Gamma$ be the element such that it sends the identity element $\ast$ to $x$. Then $\{v_i\}_{i=1}^n\subset g\widehat\Gamma\subset \Theta_\Gamma$. Since $\Gamma$ has no separating vertices and edges, $g\widehat \Gamma$ is a chunk. Let $q_\ast\colon \Theta_\Gamma\to\Theta_{\Gamma'}$ be the graph isomorphism in Lemma~\ref{lem:graph}, which is a $\V\F$--isomorphism by Proposition~\ref{prop:VF}. Let $B'_i$ be the block such that $d_H(q(B_i),B'_i)<D$ as in Lemma~\ref{lem:graph} (2). Proposition~\ref{prop:chunk} implies that $q_\ast(g\widehat\Gamma)$ is a chunk of $\Theta_{\Gamma'}$, thus the vertices of $\Theta_{\Gamma'}$ associated with $\{B'_i\}_{i=1}^n$ are contained in a chunk. Hence $\cap_{i=1}^{n} B'_i\neq\emptyset$. Since $\cap_{i=1}^{n} B'_i$ is bounded, it is a vertex $x'$ of $\Xa_{\Gamma'}$. We define $\tilde q(x)=x'$. It is clear that for any vertex $x\in \Xa_\Gamma$, we have $d(q(x),\tilde q(x))<D$ for $D=D(L,A,\Gamma,\Gamma')$. By construction, $\tilde q$ satisfies all the requirements except we still need to verify the bijectivity and uniqueness of $\tilde q$.

We define a map $F$ from $A_\Gamma$ to chunks in $\widehat{\Gamma'}$ as follows. Let $x,x'$ and $g$ be as before and let $g'\in A_{\Gamma'}$ be the element sending $x'$ to the identity element $\ast'\in A_{\Gamma'}$. Then $F(x)$ is defined to be the chunk $g'(q_\ast(g\widehat\Gamma))$. We claim $F$ is a constant function. It suffices to show if $x_1$ and $x_2$ are two adjacent vertices in $\Xa_\Gamma$, then $F(x_1)=F(x_2)$. Suppose $F(x_i)=g'_i(q_\ast(g_i\widehat\Gamma))$ for $i=1,2$. Then $g_1\widehat\Gamma\cap g_2\widehat\Gamma$ contains the closed star (in $g_1\widehat\Gamma$) of some $\F$--type vertex $w$, where $w$ corresponds to the single-labeled plain line $\ell$ passing through $x_1$ and $x_2$. Our assumption on $\Gamma$ implies $g_1\widehat\Gamma\cap g_2\widehat\Gamma$ has at least two edges. Let $\ell'$ be the stable line associated with $q_\ast(w)$ that passes through $\tilde q(x_1)$. Then $\tilde q(x_2)\subset \ell'$, hence $g'_1$ and $g'_2$ differ by an element in the stabilizer of $\ell'$. Now one readily deduces that $F(x_1)\cap F(x_2)$ contains at least two edges that intersect along an $\F$--type vertex. Thus $F(x_1)=F(x_2)$ by the definition of chunks.

Note that the image of $F$ gives a subgraph $\Gamma'_1\subset\Gamma'$. Note that if $x_1$ and $x_2$ are adjacent in $\Xa_\Gamma$, then $\tilde q(x_1)$ and $\tilde q(x_2)$ are in the same left coset of $A_{\Gamma'_1}$ in $A_{\Gamma'}$. Thus $\tilde q(x)$ is contained in a left coset of $A_{\Gamma'_1}$. Since $\im \tilde q$ is $A$--dense in $A_{\Gamma'}$, we have $\Gamma'_1=\Gamma'$. Now we repeat the previous discussion for a quasi-isometry inverse of $q$ to deduce the bijectivity statement in the theorem. To see the uniqueness, note that two blocks have finite Hausdorff distance if and only if they are equal. Thus for any block $B\subset \Xa_\Gamma$, there is a unique block $B'\subset \Xa_\Gamma$ such that $d_H(q(B),B')<\infty$. By using (2) and repeating the opening paragraph of the proof, we know $\tilde q$ is unique.
\end{proof}

\begin{corollary}
	\label{cor:qi and iso}
Let $A_\Gamma$ and $A_{\Gamma'}$ be CLTTF Artin groups. Suppose $\Gamma$ does not have separating vertices and edges. Then $A_\Gamma$ and $A_{\Gamma'}$ are quasi-isometric if and only if $\Gamma$ and $\Gamma'$ are isomorphic as labeled graphs. 
\end{corollary}

\begin{proof}
Let $\tilde q$ be as in Theorem~\ref{thm:CLTTF1}. Then by Theorem~\ref{thm:CLTTF1} (2), $\tilde q$ induces an isomorphism $\phi\colon D_\Gamma\to D_{\Gamma'}$. By restricting $\phi$ to the fundamental chamber of $D_\Gamma$, we obtain a graph isomorphism $\Gamma\to\Gamma'$. By Lemma~\ref{lem:cycle in links}, this isomorphism preserves labels of edges.
\end{proof}

\begin{theorem}
	\label{thm:CLTTF2}
Let $A_\Gamma$ be a CLTTF Artin group such that $\Gamma$ does not have separating vertices and edges. Let $\QI(A_\Gamma)$ be the quasi-isometry group of $A_\Gamma$. Let $\Isom(A_\Gamma)$ be the isometry group of $A_\Gamma$ with respect to the word distance with respect to the standard generating set. Then the following hold.
\begin{enumerate}
	\item Any quasi-isometry from $A_\Gamma$ to itself is uniformly close to an element in $\Isom(A_\Gamma)$.
	\item There are isomorphisms $\QI(A_\Gamma)\cong \Isom(A_\Gamma)\cong \Aut(D_\Gamma)$, where $\Aut(D_\Gamma)$ is the simplicial automorphism group of $D_\Gamma$.
\end{enumerate}
\end{theorem}

\begin{proof}
We first define a homomorphism $h_1\colon \Aut(D_\Gamma)\to \Isom(A_\Gamma)$ as follows. Let $\phi\in \Aut(D_\Gamma)$. By looking at the restriction of $\phi$ to rank $0$ vertices, we obtain a bijection $\varphi\colon A_\Gamma\to A_\Gamma$ satisfying Lemma~\ref{lem:graph} (2). We claim $\varphi\in \Isom(A_\Gamma)$. First we show if $x_1,x_2\in A_\Gamma$ are adjacent vertices in $\Xa_\Gamma$, then so is $\varphi(x_1)$ and $\varphi(x_2)$. Up to pre-composing and post-composing $\varphi$ with suitable left translations, we assume $x_1$ is the identity element $\ast$, and $\varphi(x_1)=\ast$. Then the corresponding $\phi$ sends the fundamental chamber to itself. By the proof of Corollary~\ref{cor:qi and iso}, $\phi$ induces a label-preserving automorphism $\alpha$ of $\Gamma$. By modifying $\varphi$ by an automorphism of $A_\Gamma$ induced by $\alpha$, we assume $\phi$ is the identity on the fundamental chamber. Thus there exists a rank $2$ vertex $v$ in the fundamental chamber such that the standard subgroup $A_{s,t}$ associated with $v$ contains $x_1=\varphi(x_1)$ and $x_2=\varphi(x_2)$. Note that $\phi$ induces a graph automorphism $\tau\colon \lk(v,\square_\Gamma)\to\lk(v,\square_\Gamma)$ which is compatible with $\varphi|_{A_{s,t}}\colon A_{s,t}\to A_{s,t}$. Since $\phi$ is the identity on the fundamental chamber, $\tau$ fixes the edge in $\lk(v,\square_\Gamma)$ that are spanned by vertices associated with $A_s$ and $A_t$. By Proposition~\ref{prop:local rigid}, $\varphi|_{A_{s,t}}$ is either the identity or a group inversion. Thus $\varphi(x_1)$ and $\varphi(x_2)$ are adjacent in $\Xa_\Gamma$. We deduce that $\varphi$ is $1$--Lipschitz with respect to the word metric. Similarly, we know $\varphi^{-1}$ is $1$--Lipschitz. Thus $\varphi\in \Isom(A_\Gamma)$. We define $h_1(\phi)=\varphi$ and one readily verifies that $h_1\colon \Aut(D_\Gamma)\to \Isom(A_\Gamma)$ is a homomorphism.

It is clear there is a homomorphism $h_2\colon \Isom(A_\Gamma)\to\QI(A_\Gamma)$. By Theorem~\ref{thm:CLTTF1}, there is a homomorphism $h_3\colon \QI(A_\Gamma)\to \Aut(D_\Gamma)$. It is clear that $h_3\circ h_2\circ h_1$ is the identity map. Thus the theorem follows.
\end{proof}

A group $G$ is \emph{strongly rigid} if any element in $\QI(G)$ is uniformly close to an automorphism of $G$.

\begin{theorem}
	\label{thm:largechar2}
Let $A_\Gamma$ be a large-type and triangle-free Artin group. Then $A_\Gamma$ is strongly rigid if and only if $\Gamma$ satisfies all of the following conditions:
\begin{enumerate}
	\item $\Gamma$ is connected and has $\ge 3$ vertices;
	\item $\Gamma$ does not have separating vertices and edges;
	\item any label preserving automorphism of $\Gamma$ which fixes the neighborhood of a vertex is the identity.
\end{enumerate}
Moreover, if a large-type and triangle-free Artin group $A_\Gamma$ satisfies all the above conditions and $H$ is a finitely generated group quasi-isometric to $A_\Gamma$, then there exists a finite index subgroup $H'\le H$ and a homomorphism $\phi\colon H'\to A_\Gamma$ with finite kernel and finite index image.
\end{theorem}

Recall that a \emph{global inversion} of $A_\Gamma$ is the automorphism of $A_\Gamma$ sending each generator to its inverse. This will be used in the following proof.
\begin{proof}
First we prove the ``if" direction. Let $q\colon A_\Gamma\to A_\Gamma$ be a quasi-isometry. Suppose $\tilde q$ is the map in Theorem~\ref{thm:CLTTF1}. Let $\phi\colon D_\Gamma\to D_\Gamma$ be the isomorphism induced by $\tilde q$. As in Theorem~\ref{thm:CLTTF2}, by modifying $\tilde q$ by a left translation and an automorphism of $A_\Gamma$ induced by label-preserving graph automorphism of $\Gamma$, we assume $\phi$ is the identity on the fundamental chamber $K$. Also by Proposition~\ref{prop:local rigid}, $\phi$ is either the identity or induced by the group inversion on the link of each rank $2$ vertex of $K$. By connectedness of $\Gamma$, we know if $\tilde q$ is an inversion on $A_s$ for some vertex $s\in\Gamma$, then $\tilde q$ is an inversion on $A_t$ for any vertex $t\in \Gamma$. Thus by possibly modifying $\tilde q$ by a global inversion, we can assume $\phi$ is the identity on both $K$ and the closed star of any rank $2$ vertex in $K$. Now the argument in the last paragraph of \cite[pp. 1436]{MR2174269} implies that $\phi$ is the identity on $D_\Gamma$ (assumption (3) is used in this step). Hence $\tilde q$ is the identity and the if direction follows.

Before turning to the ``only if" direction we prove the last statement in the theorem. Since $H$ is quasi-isometric $A_\Gamma$, there is a discrete and cobounded quasi-action of $H$ on $A_\Gamma$ (see \cite[Section 2]{kleiner2001groups} for setting up on quasi-actions). By Theorem~\ref{thm:CLTTF2}, we can replace each quasi-isometry in the quasi-action by a unique element in $\Isom(A_\Gamma)$, which gives a homomorphism $h\colon H\to \Isom(A_\Gamma)$. Since the quasi-action is discrete, $h$ has finite kernel. Let $N$ be the finite subgroup of $\Aut(A_\Gamma)$ generated by global inversion and automorphisms arise from label-preserving graph automorphisms of $\Gamma$. Then $\Isom(A_\Gamma)\cong A_\Gamma\rtimes N$ by the previous paragraph (the copy of $A_\Gamma$ in $\Isom(A_\Gamma)$ corresponds to left translations). Now the coboundedness of the quasi-action implies $\im h$ has finite index in $\Isom(A_\Gamma)$. Since $A_\Gamma$ is finite index in $\Isom(A_\Gamma)$, the moreover statement follows.

It remains to prove the ``only if" direction. If $\Gamma$ is not connected, then $A_\Gamma$ is a free product and clearly there are quasi-isometries which are not uniformly close to any automorphisms. Thus $\Gamma$ is connected. If $\Gamma$ has only two vertices, then $A_\Gamma$ is quasi-isometric to either $\mathbb Z\oplus\mathbb Z$ or $F_2\oplus\mathbb Z$ ($F_2$ denotes the free group of two generators). Again $A_\Gamma$ is not strongly rigid. Thus (1) holds.

Now we assume $\Gamma$ has a separating vertex or edge. Let $\Gamma=\Gamma_1\cup_T\Gamma_2$, where $\Gamma_1$ and $\Gamma_2$ are full subgraphs of $\Gamma$ such that $T$ is an edge or a vertex, $\Gamma_1\cup\Gamma_2=\Gamma$ and $\Gamma_1\cap\Gamma_2=T$. We also assume $T\neq \Gamma_1$ and $T\neq\Gamma_2$. Let $g$ be an element in the centralizer of $A_T$. Then there is an automorphism $\phi$ of $A_\Gamma$ defined by $\phi(v)=gvg^{-1}$ for any vertex $v\in\Gamma_2$, and $\phi(v)=v$ for any vertex in $\Gamma_1$. This is called a \emph{Dehn twist automorphism} by Crisp. We denote the standard presentation complex of $A_\Gamma$ by $P_\Gamma$. Let $f\colon P_\Gamma\to P_\Gamma$ be a homotopy equivalence such that $f$ induces $\phi$ on the fundamental groups. We can assume $f|_{P_{\Gamma_1}}$ is the identity map and $f(P_{\Gamma_2})\subset P_{\Gamma_2}$. We pull apart $P_{\Gamma_1}$ and $P_{\Gamma_2}$ in $P_\Gamma$ to form a graph of spaces $X=P_{\Gamma_1}\cup P_T\times[0,1]\cup P_{\Gamma_2}$ with the attaching maps defined in the natural way. Then $f$ induces a homotopy equivalence $g\colon X\to X$. Let $T$ be the associated Bass-Serre tree. There is a projection map $p\colon \widetilde X\to T$ where $\widetilde X$ is the universal cover of $X$. Let $\tilde g\colon \widetilde X\to\widetilde X$ be a lift of $g$. We can assume $\tilde g$ is the identity on a lift $\widetilde P_{\Gamma_1}\cup \widetilde P_T\times[0,1]$ of $P_{\Gamma_1}\cup P_T\times[0,1]$ in $\widetilde X$. Let $E\subset T$ be the edge which is the $p$--image of such lift and let $V_1$ be the vertex of $E$ associated with $\widetilde P_{\Gamma_1}$. The midpoint of $E$ divides $T$ into two halfspace. Let $H_1$ be the halfspace containing $V_1$ and $H_2$ be the other halfspace. Now we define a new map $\tilde h\colon \widetilde X\to \widetilde X$ such that $\tilde h$ is the identity on $p^{-1}(H_1)$ and $\tilde h=\tilde g$ on $p^{-1}(H_2)$. One readily verifies that $\tilde h$ is a quasi-isometry, and $\tilde h$ is not uniformly close to an automorphism.


Now we assume (1) and (2) hold, but (3) fails. Let $v\in\Gamma$ be a vertex and let $\alpha$ be a non-trivial label-preserving automorphism of $\Gamma$ which fixes the closed star of $v$ pointwise. Let $f$ be the automorphism of $A_\Gamma$ induced by $\alpha$. Then $f$ induces an automorphism $\phi\colon \square_\Gamma\to \square_\Gamma$. Let $K$ be the fundamental chamber in $\square_\Gamma$ and let $E$ be the edge in $K$ between the rank $0$ vertex in $K$ and the rank $1$ vertex in $K$ corresponding to the subgroup generated by $v$. Let $H_E$ be the hyperplane in $\square_\Gamma$ dual to $E$. Suppose $\Gamma_1$ is the full subgraph of $\Gamma$ spanned by the closed star of $v$. Then there is a natural isometric embedding $\square_{\Gamma_1}\to\square_{\Gamma_2}$. Note that $H_E\subset \square_{\Gamma_1}$ and $\phi$ fixes $H_E$ pointwise. Let $H_1$ and $H_2$ be two halfspaces bounded by $H_E$ such that $v\in H_1$. Now we define an automorphism $\phi'$ of $\square_\Gamma$ such that $\phi'$ is the identity on $H_1$ and $\phi'=\phi$ on $H_2$. Then $\phi'$ induces an isometry $f'\colon A_\Gamma\to A_\Gamma$ as in the proof of Theorem~\ref{thm:CLTTF2}. Now we show $f'$ is not uniformly close to an automorphism. Since we already assume (1) and (2) hold, by \cite[Theorem 1]{MR2174269}, $\Aut(A_\Gamma)$ is generated by automorphisms induced by label-preserving automorphisms of $\Gamma$, global inversions and inner automorphisms. On the other hand, $f'|_K$ is induced by $\alpha$ and $f'$ is the identity on some translate of $K$. No elements in $\Aut(A_\Gamma)$ can induce such automorphism of $\square_\Gamma$. Thus $f'$ is not uniformly close to any automorphism by Theorem~\ref{thm:CLTTF2} (2).
\end{proof}

\section{Ending remarks and open questions}
\label{rmk:ending remark}
Here we leave some remarks concerning perspectives of extending the results of Section~\ref{sec:application} and Section~\ref{sec:CLTTF} to more general $2$--dimensional Artin groups and ask several related questions. 

\begin{question}
	\label{que:1}
	Let $A_\Gamma$ and $A_{\Gamma'}$ be two $2$--dimensional Artin groups with finite outer automorphism groups. Suppose $\Gamma$ has more than two vertices. Suppose $A_\Gamma$ and $A_{\Gamma'}$ are quasi-isometric. Are they isomorphic?
\end{question}

By Theorem~\ref{thm:invariance}, Question~\ref{que:1} reduces to the following question:
\begin{question}
	\label{que:2}
	Let $A_\Gamma$ and $A_{\Gamma'}$ be two $2$--dimensional Artin groups with finite outer automorphism groups. Suppose that $\Gamma$ has more than two vertices and suppose that $\I_{\Gamma}$ and $\I_{\Gamma'}$ are isomorphic as graphs. Are $A_\Gamma$ and $A_{\Gamma'}$ isomorphic as groups?
\end{question}

To answer this question, one need to generalize \cite[Proposition 23]{MR2174269}. This question has a positive answer in the right-angled case \cite{bestvina2008asymptotic,MR3692971}. Even in the case of large-type Artin groups answering Question~\ref{que:1} and Question~\ref{que:2} will be interesting.

In view of \cite[Proposition 23]{MR2174269}, it is also natural to ask when an analogue of \textquotedblleft Ivanov's theorem\textquotedblright\ holds for $2$--dimensional Artin groups.
\begin{question}
	\label{que:3}
	Let $A_\Gamma$ be a $2$--dimensional Artin group. Find conditions on $\Gamma$ such that each automorphism of $A_\Gamma$ induces a graph automorphism of $\I_\Gamma$ and vice versa.
\end{question}

Of course,	 one can ask more questions of this kind by relaxing the assumptions of the results in Section~\ref{subsec:quasi-isometries}. A question of different flavor, motivated by the work of Kim and Koberda \cite{kim2014geometry}, and Behrstock, Hagen and Sisto \cite{MR3650081}, is the following.

\begin{question}
	Suppose that $A_\Gamma$ is $2$--dimensional and that $\Gamma$ is connected. Is $\I_\Gamma$ Gromov-hyperbolic? Does $A_\Gamma$ have a Mazur-Minsky hierarchy structure with $\I_\Gamma$ being the \textquotedblleft largest Gromov-hyperbolic space\textquotedblright?
\end{question}

It is known that in the right-angled case, $\I_\Gamma$ is always a quasi-tree \cite{kim2014geometry}. We suspect this is not true outside the right-angled world, since the intersection pattern of flats will be substantially more complicated. 

\begin{question}
Under what condition, the graph $\I_\Gamma$ is not a quasi-tree?
\end{question}


\bibliography{mybib}{}

\begin{bibdiv}
\begin{biblist}

\bib{Appel1984}{incollection}{
      author={Appel, Kenneth~I.},
       title={On {A}rtin groups and {C}oxeter groups of large type},
        date={1984},
   booktitle={Contributions to group theory},
      series={Contemp. Math.},
      volume={33},
   publisher={Amer. Math. Soc., Providence, RI},
       pages={50\ndash 78},
         url={http://dx.doi.org/10.1090/conm/033/767099},
      review={\MR{767099}},
}

\bib{AppelSchupp1983}{article}{
      author={Appel, Kenneth~I.},
      author={Schupp, Paul~E.},
       title={Artin groups and infinite {C}oxeter groups},
        date={1983},
        ISSN={0020-9910},
     journal={Invent. Math.},
      volume={72},
      number={2},
       pages={201\ndash 220},
         url={http://dx.doi.org/10.1007/BF01389320},
      review={\MR{700768}},
}

\bib{brady2002two}{inproceedings}{
      author={Brady, Noel},
      author={Crisp, John},
       title={Two-dimensional {A}rtin groups with {${\rm CAT}(0)$} dimension
  three},
        date={2002},
   booktitle={Proceedings of the {C}onference on {G}eometric and
  {C}ombinatorial {G}roup {T}heory, {P}art {I} ({H}aifa, 2000)},
      volume={94},
       pages={185\ndash 214},
         url={http://dx.doi.org/10.1023/A:1020962804856},
      review={\MR{1950878}},
}

\bib{bell2005three}{article}{
      author={Bell, Robert~W.},
       title={Three-dimensional {FC} {A}rtin groups are {CAT}(0)},
        date={2005},
        ISSN={0046-5755},
     journal={Geom. Dedicata},
      volume={113},
       pages={21\ndash 53},
         url={http://dx.doi.org/10.1007/s10711-005-3691-9},
      review={\MR{2171297}},
}

\bib{Bestvina1999}{article}{
      author={Bestvina, Mladen},
       title={Non-positively curved aspects of {A}rtin groups of finite type},
        date={1999},
        ISSN={1465-3060},
     journal={Geom. Topol.},
      volume={3},
       pages={269\ndash 302 (electronic)},
         url={http://dx.doi.org/10.2140/gt.1999.3.269},
      review={\MR{1714913}},
}

\bib{BridsonHaefliger1999}{book}{
      author={Bridson, Martin~R.},
      author={Haefliger, Andr\'e},
       title={Metric spaces of non-positive curvature},
      series={Grundlehren der Mathematischen Wissenschaften [Fundamental
  Principles of Mathematical Sciences]},
   publisher={Springer-Verlag, Berlin},
        date={1999},
      volume={319},
        ISBN={3-540-64324-9},
         url={http://dx.doi.org/10.1007/978-3-662-12494-9},
      review={\MR{1744486}},
}

\bib{MR3650081}{article}{
      author={Behrstock, Jason~A.},
      author={Hagen, Mark~F.},
      author={Sisto, Alessandro},
       title={Hierarchically hyperbolic spaces, {I}: {C}urve complexes for
  cubical groups},
        date={2017},
        ISSN={1465-3060},
     journal={Geom. Topol.},
      volume={21},
      number={3},
       pages={1731\ndash 1804},
         url={http://dx.doi.org/10.2140/gt.2017.21.1731},
      review={\MR{3650081}},
}

\bib{behrstock2017quasiflats}{article}{
      author={Behrstock, Jason~A.},
      author={Hagen, Mark~F.},
      author={Sisto, Alessandro},
       title={Quasiflats in hierarchically hyperbolic spaces},
        date={2017},
     journal={preprint arXiv:1704.04271},
}

\bib{behrstock2015hierarchically}{article}{
      author={Behrstock, Jason~A.},
      author={Hagen, Mark~F.},
      author={Sisto, Alessandro},
       title={Hierarchically hyperbolic spaces {II}: {C}ombination theorems and
  the distance formula},
        date={2019},
        ISSN={0030-8730},
     journal={Pacific J. Math.},
      volume={299},
      number={2},
       pages={257\ndash 338},
         url={https://doi.org/10.2140/pjm.2019.299.257},
      review={\MR{3956144}},
}

\bib{MR2727658}{article}{
      author={Behrstock, Jason~A.},
      author={Januszkiewicz, Tadeusz},
      author={Neumann, Walter~D.},
       title={Quasi-isometric classification of some high dimensional
  right-angled {A}rtin groups},
        date={2010},
        ISSN={1661-7207},
     journal={Groups Geom. Dyn.},
      volume={4},
      number={4},
       pages={681\ndash 692},
         url={http://dx.doi.org/10.4171/GGD/100},
      review={\MR{2727658}},
}

\bib{behrstock2012geometry}{article}{
      author={Behrstock, Jason~A.},
      author={Kleiner, Bruce},
      author={Minsky, Yair},
      author={Mosher, Lee},
       title={Geometry and rigidity of mapping class groups},
        date={2012},
        ISSN={1465-3060},
     journal={Geom. Topol.},
      volume={16},
      number={2},
       pages={781\ndash 888},
         url={http://dx.doi.org/10.2140/gt.2012.16.781},
      review={\MR{2928983}},
}

\bib{bestvina2008asymptotic}{article}{
      author={Bestvina, Mladen},
      author={Kleiner, Bruce},
      author={Sageev, Michah},
       title={The asymptotic geometry of right-angled {A}rtin groups. {I}},
        date={2008},
        ISSN={1465-3060},
     journal={Geom. Topol.},
      volume={12},
      number={3},
       pages={1653\ndash 1699},
         url={http://dx.doi.org/10.2140/gt.2008.12.1653},
      review={\MR{2421136}},
}

\bib{bks}{article}{
      author={Bestvina, Mladen},
      author={Kleiner, Bruce},
      author={Sageev, Michah},
       title={Quasiflats in {${\rm CAT}(0)$} 2-complexes},
        date={2016},
        ISSN={1472-2747},
     journal={Algebr. Geom. Topol.},
      volume={16},
      number={5},
       pages={2663\ndash 2676},
         url={http://dx.doi.org/10.2140/agt.2016.16.2663},
      review={\MR{3572343}},
}

\bib{BradyMcCammond2000}{article}{
      author={Brady, Thomas},
      author={McCammond, Jonathan~P.},
       title={Three-generator {A}rtin groups of large type are biautomatic},
        date={2000},
        ISSN={0022-4049},
     journal={J. Pure Appl. Algebra},
      volume={151},
      number={1},
       pages={1\ndash 9},
         url={http://dx.doi.org/10.1016/S0022-4049(99)00094-8},
      review={\MR{1770639}},
}

\bib{brady2000three}{article}{
      author={Brady, Thomas},
      author={McCammond, Jonathan~P.},
       title={Three-generator {A}rtin groups of large type are biautomatic},
        date={2000},
        ISSN={0022-4049},
     journal={J. Pure Appl. Algebra},
      volume={151},
      number={1},
       pages={1\ndash 9},
         url={https://doi.org/10.1016/S0022-4049(99)00094-8},
      review={\MR{1770639}},
}

\bib{brady2010braids}{article}{
      author={Brady, Thomas},
      author={McCammond, Jonathan~P.},
       title={Braids, posets and orthoschemes},
        date={2010},
        ISSN={1472-2747},
     journal={Algebr. Geom. Topol.},
      volume={10},
      number={4},
       pages={2277\ndash 2314},
         url={http://dx.doi.org/10.2140/agt.2010.10.2277},
      review={\MR{2745672}},
}

\bib{behrstock2008quasi}{article}{
      author={Behrstock, Jason~A.},
      author={Neumann, Walter~D.},
       title={Quasi-isometric classification of graph manifold groups},
        date={2008},
        ISSN={0012-7094},
     journal={Duke Math. J.},
      volume={141},
      number={2},
       pages={217\ndash 240},
         url={http://dx.doi.org/10.1215/S0012-7094-08-14121-3},
      review={\MR{2376814}},
}

\bib{brieskorn1972artin}{article}{
      author={Brieskorn, Egbert},
      author={Saito, Kyoji},
       title={Artin-{G}ruppen und {C}oxeter-{G}ruppen},
        date={1972},
        ISSN={0020-9910},
     journal={Invent. Math.},
      volume={17},
       pages={245\ndash 271},
         url={http://dx.doi.org/10.1007/BF01406235},
      review={\MR{0323910}},
}

\bib{CharneyCrispAutomorphism}{article}{
      author={Charney, Ruth},
      author={Crisp, John},
       title={Automorphism groups of some affine and finite type {A}rtin
  groups},
        date={2005},
        ISSN={1073-2780},
     journal={Math. Res. Lett.},
      volume={12},
      number={2-3},
       pages={321\ndash 333},
         url={http://dx.doi.org/10.4310/MRL.2005.v12.n3.a4},
      review={\MR{2150887}},
}

\bib{charney2007relative}{article}{
      author={Charney, Ruth},
      author={Crisp, John},
       title={Relative hyperbolicity and {A}rtin groups},
        date={2007},
        ISSN={0046-5755},
     journal={Geom. Dedicata},
      volume={129},
       pages={1\ndash 13},
         url={https://doi.org/10.1007/s10711-007-9178-0},
      review={\MR{2353977}},
}

\bib{charney1995k}{article}{
      author={Charney, Ruth},
      author={Davis, Michael~W.},
       title={The {$K(\pi,1)$}-problem for hyperplane complements associated to
  infinite reflection groups},
        date={1995},
        ISSN={0894-0347},
     journal={J. Amer. Math. Soc.},
      volume={8},
      number={3},
       pages={597\ndash 627},
         url={http://dx.doi.org/10.2307/2152924},
      review={\MR{1303028}},
}

\bib{CharneyDavis}{article}{
      author={Charney, Ruth},
      author={Davis, Michael~W.},
       title={The {$K(\pi,1)$}-problem for hyperplane complements associated to
  infinite reflection groups},
        date={1995},
        ISSN={0894-0347},
     journal={J. Amer. Math. Soc.},
      volume={8},
      number={3},
       pages={597\ndash 627},
         url={http://dx.doi.org/10.2307/2152924},
      review={\MR{1303028}},
}

\bib{charney2016problems}{article}{
      author={Charney, Ruth},
       title={Problems related to {A}rtin groups},
        date={2016},
     journal={Preprint available at
  \url{http://people.brandeis.edu/~charney/papers/Artin_probs. pdf}},
}

\bib{Charney1992}{article}{
      author={Charney, Ruth},
       title={Artin groups of finite type are biautomatic},
        date={1992},
        ISSN={0025-5831},
     journal={Math. Ann.},
      volume={292},
      number={4},
       pages={671\ndash 683},
         url={http://dx.doi.org/10.1007/BF01444642},
      review={\MR{1157320}},
}

\bib{MR1314589}{article}{
      author={Charney, Ruth},
       title={Geodesic automation and growth functions for {A}rtin groups of
  finite type},
        date={1995},
        ISSN={0025-5831},
     journal={Math. Ann.},
      volume={301},
      number={2},
       pages={307\ndash 324},
         url={http://dx.doi.org/10.1007/BF01446631},
      review={\MR{1314589}},
}

\bib{charney2019artin}{article}{
      author={Charney, Ruth},
      author={Morris-Wright, Rose},
       title={Artin groups of infinite type: {T}rivial centers and acylindrical
  hyperbolicity},
        date={2019},
        ISSN={0002-9939},
     journal={Proc. Amer. Math. Soc.},
      volume={147},
      number={9},
       pages={3675\ndash 3689},
         url={https://doi.org/10.1090/proc/14503},
      review={\MR{3993762}},
}

\bib{charney2003k}{article}{
      author={Charney, Ruth},
      author={Peifer, David},
       title={The {$K(\pi,1)$}-conjecture for the affine braid groups},
        date={2003},
        ISSN={0010-2571},
     journal={Comment. Math. Helv.},
      volume={78},
      number={3},
       pages={584\ndash 600},
         url={https://doi.org/10.1007/s00014-003-0764-y},
      review={\MR{1998395}},
}

\bib{charney2014convexity}{article}{
      author={Charney, Ruth},
      author={Paris, Luis},
       title={Convexity of parabolic subgroups in {A}rtin groups},
        date={2014},
        ISSN={0024-6093},
     journal={Bull. Lond. Math. Soc.},
      volume={46},
      number={6},
       pages={1248\ndash 1255},
         url={http://dx.doi.org/10.1112/blms/bdu077},
      review={\MR{3291260}},
}

\bib{MR2174269}{article}{
      author={Crisp, John},
       title={Automorphisms and abstract commensurators of 2-dimensional
  {A}rtin groups},
        date={2005},
        ISSN={1465-3060},
     journal={Geom. Topol.},
      volume={9},
       pages={1381\ndash 1441},
         url={http://dx.doi.org/10.2140/gt.2005.9.1381},
      review={\MR{2174269}},
}

\bib{cumplido2017loxodromic}{article}{
      author={Cumplido, Mar\'{\i}a},
       title={On loxodromic actions of {A}rtin-{T}its groups},
        date={2019},
        ISSN={0022-4049},
     journal={J. Pure Appl. Algebra},
      volume={223},
      number={1},
       pages={340\ndash 348},
         url={https://doi.org/10.1016/j.jpaa.2018.03.013},
      review={\MR{3833463}},
}

\bib{calvez2016acylindrical}{article}{
      author={Calvez, Matthieu},
      author={Wiest, Bert},
       title={Acylindrical hyperbolicity and {A}rtin-{T}its groups of spherical
  type},
        date={2017},
        ISSN={0046-5755},
     journal={Geom. Dedicata},
      volume={191},
       pages={199\ndash 215},
         url={https://doi.org/10.1007/s10711-017-0252-y},
      review={\MR{3719080}},
}

\bib{deligne}{article}{
      author={Deligne, Pierre},
       title={Les immeubles des groupes de tresses g\'en\'eralis\'es},
        date={1972},
        ISSN={0020-9910},
     journal={Invent. Math.},
      volume={17},
       pages={273\ndash 302},
         url={http://dx.doi.org/10.1007/BF01406236},
      review={\MR{0422673}},
}

\bib{davis2017determining}{article}{
      author={Davis, Michael~W.},
      author={Huang, Jingyin},
       title={Determining the action dimension of an {A}rtin group by using its
  complex of abelian subgroups},
        date={2017},
        ISSN={0024-6093},
     journal={Bull. Lond. Math. Soc.},
      volume={49},
      number={4},
       pages={725\ndash 741},
         url={https://doi.org/10.1112/blms.12061},
      review={\MR{3725492}},
}

\bib{MR2208796}{article}{
      author={Digne, Fran\c{c}ois},
       title={Pr\'esentations duales des groupes de tresses de type affine
  {$\widetilde A$}},
        date={2006},
        ISSN={0010-2571},
     journal={Comment. Math. Helv.},
      volume={81},
      number={1},
       pages={23\ndash 47},
         url={http://dx.doi.org/10.4171/CMH/41},
      review={\MR{2208796}},
}

\bib{MR2985512}{article}{
      author={Digne, Fran\c{c}ois},
       title={A {G}arside presentation for {A}rtin-{T}its groups of type
  {$C_n$}},
        date={2012},
        ISSN={0373-0956},
     journal={Ann. Inst. Fourier (Grenoble)},
      volume={62},
      number={2},
       pages={641\ndash 666},
         url={http://dx.doi.org/10.5802/aif.2690},
      review={\MR{2985512}},
}

\bib{eskin1997quasi}{article}{
      author={Eskin, Alex},
      author={Farb, Benson},
       title={Quasi-flats and rigidity in higher rank symmetric spaces},
        date={1997},
        ISSN={0894-0347},
     journal={J. Amer. Math. Soc.},
      volume={10},
      number={3},
       pages={653\ndash 692},
         url={http://dx.doi.org/10.1090/S0894-0347-97-00238-5},
      review={\MR{1434399}},
}

\bib{elsner2017}{article}{
      author={Elsner, Tomasz},
       title={Quasi-flats in systolic complexes},
        date={2019},
        ISSN={0002-9939},
     journal={Proc. Amer. Math. Soc.},
      volume={147},
      number={6},
       pages={2365\ndash 2374},
         url={https://doi.org/10.1090/proc/14427},
      review={\MR{3951417}},
}

\bib{gordon2004artin}{article}{
      author={Gordon, Cameron~McA.},
       title={Artin groups, 3-manifolds and coherence},
        date={2004},
        ISSN={1405-213X},
     journal={Bol. Soc. Mat. Mexicana (3)},
      volume={10},
      number={Special Issue},
       pages={193\ndash 198},
      review={\MR{2199348}},
}

\bib{MR3203644}{incollection}{
      author={Godelle, Eddy},
      author={Paris, Luis},
       title={Basic questions on {A}rtin-{T}its groups},
        date={2012},
   booktitle={Configuration spaces},
      series={CRM Series},
      volume={14},
   publisher={Ed. Norm., Pisa},
       pages={299\ndash 311},
         url={http://dx.doi.org/10.1007/978-88-7642-431-1_13},
      review={\MR{3203644}},
}

\bib{hamenstaedt2005geometry}{article}{
      author={Hamenstaedt, Ursula},
       title={Geometry of the mapping class groups {III}: Quasi-isometric
  rigidity},
        date={2005},
     journal={preprint},
}

\bib{huang2015cocompactly}{article}{
      author={Huang, Jingyin},
      author={Jankiewicz, Kasia},
      author={Przytycki, Piotr},
       title={Cocompactly cubulated 2-dimensional {A}rtin groups},
        date={2016},
        ISSN={0010-2571},
     journal={Comment. Math. Helv.},
      volume={91},
      number={3},
       pages={519\ndash 542},
         url={http://dx.doi.org/10.4171/CMH/394},
      review={\MR{3541719}},
}

\bib{huang2016groups}{article}{
      author={Huang, Jingyin},
      author={Kleiner, Bruce},
       title={Groups quasi-isometric to right-angled {A}rtin groups},
        date={2018},
        ISSN={0012-7094},
     journal={Duke Math. J.},
      volume={167},
      number={3},
       pages={537\ndash 602},
         url={https://doi.org/10.1215/00127094-2017-0042},
      review={\MR{3761106}},
}

\bib{haettel20166}{article}{
      author={Haettel, Thomas},
      author={Kielak, Dawid},
      author={Schwer, Petra},
       title={The 6-strand braid group is {${\rm CAT}(0)$}},
        date={2016},
        ISSN={0046-5755},
     journal={Geom. Dedicata},
      volume={182},
       pages={263\ndash 286},
         url={http://dx.doi.org/10.1007/s10711-015-0138-9},
      review={\MR{3500387}},
}

\bib{Artinmetric}{article}{
      author={Huang, Jingyin},
      author={Osajda, Damian},
       title={Metric systolicity and two-dimensional {A}rtin groups},
        date={2019},
        ISSN={0025-5831},
     journal={Math. Ann.},
      volume={374},
      number={3-4},
       pages={1311\ndash 1352},
         url={https://doi.org/10.1007/s00208-019-01823-6},
      review={\MR{3985112}},
}

\bib{Artinsystolic}{article}{
      author={Huang, Jingyin},
      author={Osajda, Damian},
       title={Large-type {A}rtin groups are systolic},
        date={2020},
        ISSN={0024-6115},
     journal={Proc. Lond. Math. Soc. (3)},
      volume={120},
      number={1},
       pages={95\ndash 123},
         url={https://doi.org/10.1112/plms.12284},
      review={\MR{3999678}},
}

\bib{holt2011artin}{article}{
      author={Holt, Derek~F.},
      author={Rees, Sarah},
       title={Artin groups of large type are shortlex automatic with regular
  geodesics},
        date={2012},
        ISSN={0024-6115},
     journal={Proc. Lond. Math. Soc. (3)},
      volume={104},
      number={3},
       pages={486\ndash 512},
         url={https://doi.org/10.1112/plms/pdr035},
      review={\MR{2900234}},
}

\bib{huang2016quasi}{article}{
      author={Huang, Jingyin},
       title={Quasi-isometric classification of right-angled {A}rtin groups
  {II}: several infinite out cases},
        date={2016},
     journal={preprint arXiv:1603.02372},
}

\bib{MR3692971}{article}{
      author={Huang, Jingyin},
       title={Quasi-isometric classification of right-angled {A}rtin groups,
  {I}: {T}he finite out case},
        date={2017},
        ISSN={1465-3060},
     journal={Geom. Topol.},
      volume={21},
      number={6},
       pages={3467\ndash 3537},
         url={http://dx.doi.org/10.2140/gt.2017.21.3467},
      review={\MR{3692971}},
}

\bib{MR3654109}{article}{
      author={Huang, Jingyin},
       title={Top-dimensional quasiflats in {$\rm CAT(0)$} cube complexes},
        date={2017},
        ISSN={1465-3060},
     journal={Geom. Topol.},
      volume={21},
      number={4},
       pages={2281\ndash 2352},
         url={http://dx.doi.org/10.2140/gt.2017.21.2281},
      review={\MR{3654109}},
}

\bib{huang2016commensurability}{article}{
      author={Huang, Jingyin},
       title={Commensurability of groups quasi-isometric to {RAAG}s},
        date={2018},
        ISSN={0020-9910},
     journal={Invent. Math.},
      volume={213},
      number={3},
       pages={1179\ndash 1247},
         url={https://doi.org/10.1007/s00222-018-0803-3},
      review={\MR{3842063}},
}

\bib{kim2013embedability}{article}{
      author={Kim, Sang-hyun},
      author={Koberda, Thomas},
       title={Embedability between right-angled {A}rtin groups},
        date={2013},
        ISSN={1465-3060},
     journal={Geom. Topol.},
      volume={17},
      number={1},
       pages={493\ndash 530},
         url={http://dx.doi.org/10.2140/gt.2013.17.493},
      review={\MR{3039768}},
}

\bib{kim2014geometry}{article}{
      author={Kim, Sang-hyun},
      author={Koberda, Thomas},
       title={The geometry of the curve graph of a right-angled {A}rtin group},
        date={2014},
        ISSN={0218-1967},
     journal={Internat. J. Algebra Comput.},
      volume={24},
      number={2},
       pages={121\ndash 169},
         url={https://doi.org/10.1142/S021819671450009X},
      review={\MR{3192368}},
}

\bib{kleiner2001groups}{article}{
      author={Kleiner, Bruce},
      author={Leeb, Bernhard},
       title={Groups quasi-isometric to symmetric spaces},
        date={2001},
        ISSN={1019-8385},
     journal={Comm. Anal. Geom.},
      volume={9},
      number={2},
       pages={239\ndash 260},
         url={http://dx.doi.org/10.4310/CAG.2001.v9.n2.a1},
      review={\MR{1846203}},
}

\bib{kapovich1997quasi}{article}{
      author={Kapovich, Michael},
      author={Leeb, Bernhard},
       title={Quasi-isometries preserve the geometric decomposition of {H}aken
  manifolds},
        date={1997},
        ISSN={0020-9910},
     journal={Invent. Math.},
      volume={128},
      number={2},
       pages={393\ndash 416},
         url={http://dx.doi.org/10.1007/s002220050145},
      review={\MR{1440310}},
}

\bib{kleiner1997rigidity}{article}{
      author={Kleiner, Bruce},
      author={Leeb, Bernhard},
       title={Rigidity of quasi-isometries for symmetric spaces and {E}uclidean
  buildings},
        date={1997},
        ISSN={0764-4442},
     journal={C. R. Acad. Sci. Paris S\'er. I Math.},
      volume={324},
      number={6},
       pages={639\ndash 643},
         url={http://dx.doi.org/10.1016/S0764-4442(97)86981-9},
      review={\MR{1447034}},
}

\bib{kapovich2004relative}{article}{
      author={Kapovich, Ilya},
      author={Schupp, Paul},
       title={Relative hyperbolicity and {A}rtin groups},
        date={2004},
        ISSN={0046-5755},
     journal={Geom. Dedicata},
      volume={107},
       pages={153\ndash 167},
         url={https://doi.org/10.1007/s10711-004-9285-0},
      review={\MR{2110760}},
}

\bib{LSbook}{book}{
      author={Lyndon, Roger~C.},
      author={Schupp, Paul~E.},
       title={Combinatorial group theory},
      series={Classics in Mathematics},
   publisher={Springer-Verlag, Berlin},
        date={2001},
        ISBN={3-540-41158-5},
         url={http://dx.doi.org/10.1007/978-3-642-61896-3},
        note={Reprint of the 1977 edition},
      review={\MR{1812024}},
}

\bib{McCammond2010}{article}{
      author={McCammond, Jon},
       title={Combinatorial descriptions of multi-vertex 2-complexes},
        date={2010},
        ISSN={0019-2082},
     journal={Illinois J. Math.},
      volume={54},
      number={1},
       pages={137\ndash 154},
         url={http://projecteuclid.org/euclid.ijm/1299679742},
      review={\MR{2776989}},
}

\bib{MR3351966}{article}{
      author={McCammond, Jon},
       title={Dual euclidean {A}rtin groups and the failure of the lattice
  property},
        date={2015},
        ISSN={0021-8693},
     journal={J. Algebra},
      volume={437},
       pages={308\ndash 343},
         url={http://dx.doi.org/10.1016/j.jalgebra.2015.04.021},
      review={\MR{3351966}},
}

\bib{jonproblems}{article}{
      author={McCammond, Jon},
       title={The mysterious geometry of {A}rtin groups},
        date={2017},
        ISSN={2426-0312},
     journal={Winter Braids Lect. Notes},
      volume={4},
      number={Winter Braids VII (Caen, 2017)},
       pages={Exp. No. 1, 30},
         url={https://doi.org/10.5802/wbln.17},
      review={\MR{3922033}},
}

\bib{MR2867450}{article}{
      author={Mosher, Lee},
      author={Sageev, Michah},
      author={Whyte, Kevin},
       title={Quasi-actions on trees {II}: {F}inite depth {B}ass-{S}erre
  trees},
        date={2011},
        ISSN={0065-9266},
     journal={Mem. Amer. Math. Soc.},
      volume={214},
      number={1008},
       pages={vi+105},
         url={http://dx.doi.org/10.1090/S0065-9266-2011-00585-X},
      review={\MR{2867450}},
}

\bib{olshanskii2017flat}{article}{
      author={Olshanskii, A.~Yu.},
      author={Sapir, M.~V.},
       title={On flat submaps of maps of nonpositive curvature},
        date={2019},
        ISSN={0002-9947},
     journal={Trans. Amer. Math. Soc.},
      volume={371},
      number={7},
       pages={4869\ndash 4894},
         url={https://doi.org/10.1090/tran/7487},
      review={\MR{3934470}},
}

\bib{paris2002artin}{article}{
      author={Paris, Luis},
       title={Artin monoids inject in their groups},
        date={2002},
        ISSN={0010-2571},
     journal={Comment. Math. Helv.},
      volume={77},
      number={3},
       pages={609\ndash 637},
         url={http://dx.doi.org/10.1007/s00014-002-8353-z},
      review={\MR{1933791}},
}

\bib{Peifer}{article}{
      author={Peifer, David},
       title={Artin groups of extra-large type are biautomatic},
        date={1996},
        ISSN={0022-4049},
     journal={J. Pure Appl. Algebra},
      volume={110},
      number={1},
       pages={15\ndash 56},
         url={http://dx.doi.org/10.1016/0022-4049(95)00094-1},
      review={\MR{1390670}},
}

\bib{Pride}{article}{
      author={Pride, Stephen~J.},
       title={On {T}its' conjecture and other questions concerning {A}rtin and
  generalized {A}rtin groups},
        date={1986},
        ISSN={0020-9910},
     journal={Invent. Math.},
      volume={86},
      number={2},
       pages={347\ndash 356},
         url={http://dx.doi.org/10.1007/BF01389074},
      review={\MR{856848}},
}

\bib{lek}{book}{
      author={van~der Lek, Harm},
       title={The homotopy type of complex hyperplane complements},
   publisher={PhD thesis, Katholieke Universiteit te Nijmegen},
        date={1983},
         url={https://books.google.ca/books?id=L0ByHAAACAAJ},
}

\bib{wortman2006quasiflats}{article}{
      author={Wortman, Kevin},
       title={Quasiflats with holes in reductive groups},
        date={2006},
        ISSN={1472-2747},
     journal={Algebr. Geom. Topol.},
      volume={6},
       pages={91\ndash 117},
         url={http://dx.doi.org/10.2140/agt.2006.6.91},
      review={\MR{2199455}},
}

\bib{yang2016statistically}{article}{
      author={Yang, Wen-yuan},
       title={{Statistically Convex-Cocompact Actions of Groups with
  Contracting Elements}},
        date={2018},
        ISSN={1073-7928},
     journal={Int. Math. Res. Not. IMRN},
}

\bib{Zeeman1964}{article}{
      author={Zeeman, Erik~Christopher},
       title={Relative simplicial approximation},
        date={1964},
     journal={Proc. Cambridge Philos. Soc.},
      volume={60},
       pages={39\ndash 43},
      review={\MR{0158403}},
}

\end{biblist}
\end{bibdiv}
\bibliographystyle{plain}
\end{document}